\documentclass[psamsfont,a4paper,11pt]{amsart}
\usepackage{fullpage, amsmath, amssymb, amsthm, amscd, setspace, bm, graphicx, indentfirst, multirow, tikz, enumerate, verbatim, appendix}
\usepackage{adjustbox,amsfonts,array,graphicx,booktabs,tabularx,multirow,multicol,stmaryrd, tabu}
\usepackage[utf8]{inputenc}
\usepackage{cite}
\usepackage{hyperref}
\usepackage{mathabx,epsfig}
\title{Lattice Models, Hamiltonian Operators, and Symmetric Functions}
\author{Andrew Hardt}
\email{ahardt@illinois.edu}
\date{\today}

\newtheorem{theorem}{Theorem}[section]

\newtheorem{lemma}[theorem]{Lemma}
\newtheorem{proposition}[theorem]{Proposition}
\newtheorem{corollary}[theorem]{Corollary}
\newtheorem{remark}[theorem]{Remark}

\newcommand{\Z}{\mathbb{Z}}
\newcommand{\R}{\mathbb{R}}

\newcommand{\C}{\mathbb{C}}

\newcommand{\ba}{\overline}

\newcommand{\wt}{\text{wt}}
\newcommand{\vx}[1]{\mathtt{#1}}
\begin{document} \maketitle

\begin{abstract}
We give general conditions for the existence of a Hamiltonian operator whose discrete time evolution matches the partition function of certain solvable lattice models. In particular, we examine two classes of lattice models: the classical six-vertex model and a generalized family of $(2n+4)$-vertex models for each positive integer $n$. These models depend on a statistic called \emph{charge}, and are associated to the quantum group $U_q(\widehat{\mathfrak{gl}}(1|n))$ \cite{BBBG-metaplectic-solvability}. Our results show a close and unexpected connection between Hamiltonian operators and the Yang-Baxter equation.

The six-vertex model can be associated with Hamiltonians from classical Fock space, and we show that such a correspondence exists precisely when the Boltzmann weights are free fermionic. This allows us to prove that the free fermionic partition function is always a (skew) supersymmetric Schur function and then use the Berele-Regev formula to correct a result from \cite{BBF-Schur-polynomials} on the free fermionic domain-wall partition function. In this context, the supersymmetric function involution takes us between two lattice models that generalize the \emph{vicious walker} and \emph{osculating walker} models. We also observe that the supersymmetric Jacobi-Trudi formula, through Wick's Theorem, can be seen as a six-vertex analogue of the Lindstr\"{o}m-Gessel-Viennot Lemma.

Then, we prove a sharp solvability criterion for the six-vertex model with charge that provides the proper analogue of the free fermion condition. Building on results in \cite{BBBG-Hamiltonian}, we show that this criterion exactly dictates when a charged model has a Hamiltonian operator acting on a Drinfeld twist of $q$-Fock space. The resulting partition function is then a (skew) supersymmetric LLT polynomial, and almost all supersymmetric LLT polynomials appear as partition functions of our lattice models. We also prove a Cauchy identity for skew supersymmetric LLT polynomials. (see 2024 author's note below)
\end{abstract}

\section{Introduction}

This paper discusses the connections between two types of integrability phenomena arising from quantum groups: solvable lattice models and Hamiltonian operators coming from Heisenberg algebras. Solvable lattice models are parametrized by modules of quantum groups, and are sources of many special functions from geometry, representation theory, and symmetric function theory. The word ``solvable'' refers to the existence of a solution to the Yang-Baxter equation, and such a solution guarantees that the partition function of the model has symmetries or functional equations. Examples of special functions that have been studied using lattice models include Macdonald polynomials \cite{GarbaliDeGierWheeler-Macdonald, GarbaliWheeler-modified-Macdonald}, Grothendieck polynomials \cite{ZJWheeler, KnutsonZJ-schubert-puzzles-I, KnutsonZJ-schubert-puzzles-II, FrozenPipes, BuciumasScrimshaw-bumpless-model}, LLT polynomials \cite{Corteel-LLT, Aggarwal-Borodin-Wheeler, CFYWZZ-super-LLT}, and metaplectic Whittaker functions \cite{BBCFG-metaplectic,BBBG-metaplectic-solvability, BBBG-metahori}.

Hamiltonian operators, sometimes called half-vertex operators, have long been studied in physics, including in soliton theory and hierarchies of differential equations \cite{Kac-book, JimboMiwa-solitons}. Discrete time evolution Hamiltonians have been used to study special functions like Schur polynomials \cite{ZinnJustin-six-vertex} and Hall-Littlewood polynomials \cite{FodaWheeler-Hall-Littlewood, Jing-HL-vertex, Korff-cylindric-Hecke}. A general framework for such constructions from a combinatorial perspective was given by Lam \cite{Lam}. He considered an arbitrary Heisenberg algebra representation and proved a generalized boson-fermion correspondence, obtaining two dual families of polynomials associated to each Hamiltonian operator. Furthermore, Lam showed these functions have Pieri and Cauchy identities determined by the structure constants of the Heisenberg algebra. Both Macdonald polynomials and LLT polynomials fit into this framework, as do many related polynomials and specializations.

We consider Heisenberg algebra representations with bases indexed by partitions, along with dual representations whose dual basis is also indexed by partitions. One obtains families of ``skew'' functions from these representations by choosing an element of the Heisenberg algebra, acting on a basis vector in the representation and pairing it with a dual basis vector. We use bra-ket notation to express this. If $h$ is an element of the Heisenberg algebra, then \[\langle\mu|h|\lambda\rangle = \langle\mu|h\otimes 1|\lambda\rangle = \langle\mu|1\otimes h|\lambda\rangle\] refers to the pairing of $\langle\mu|$ with the vector obtained from the action of $h$ on $|\lambda\rangle$, or equivalently, the pairing of $\lambda\rangle$ with the vector obtained from the action of $h$ on $\langle\mu|$.

Our Hamiltonians are exponential operators involving infinite sums in the Heisenberg algebra of the form \[H = \sum_{k\ge 1} s_kJ_k, \hspace{20pt} e^H = \sum_{m=1}^\infty \frac{H^m}{m!},\] where $J_k$ is the $k$th \emph{current operator} which in the case of \emph{Fock space} acts on particles by displacing them by $k$ units. We thus obtain a family of functions \[\sigma_{\lambda/\mu} := \langle\mu|e^H|\lambda\rangle,\] which are sometimes called \emph{$\tau$-functions} in the literature and are related to the KP hierarchy of differential equations \cite{Kac-book, JimboMiwa-solitons, ZinnJustin-six-vertex}.

Since up to isomorphism there is only one Heisenberg algebra and only one highest-weight representation of this algebra \cite[Proposition~2.1]{Kac-Raina}, our functions really depend on the \emph{realization} of the representation i.e. our choice of basis. There are two ways to modify our family of polynomials:
\begin{itemize}
    \item Change the Heisenberg algebra generators via a substitution;
    \item Change the (realization of the) representation of the Heisenberg algebra;
\end{itemize}
We will see that the first type of modification corresponds to changing the Boltzmann weights of our lattice model, while the second corresponds (at least in the cases we consider) to choosing a lattice model with a different combinatorial and representation-theoretic structure. In this paper, we look at two well known representations of Heisenberg algebras. The first, (classical) Fock space, is a representation of the Lie algebra $\mathfrak{gl}_\infty$ that has been studied in many places e.g. \cite{Kac-book, Kac-Raina, JimboMiwa-solitons}. The second, $q$-Fock space, is a quotient of the tensor algebra of the standard $U_q(\widehat{\mathfrak{sl}}_n)$ evaluation module. At $n=1$ or $q=1$, $q$-Fock space degenerates into classical Fock space.

Our goal is to understand when the $\tau$-function associated to a Hamiltonian operator equals the partition function of a rectangular lattice model. These lattice models are grids of vertices; each edge in the grid is assigned  a \emph{spin}, and the spins around a vertex must have one of several \emph{admissible} configurations, in which case the vertex is assigned a nonzero \emph{Boltzmann weight}. If every vertex is admissible, we call this global configuration an \emph{admissible state}. Then the partition function is defined as \[Z := \sum_{\text{state } \mathfrak{s}} \prod_{\text{vertex } v} \wt(v).\] Our lattice models have fixed side boundary conditions, and we allow the top and bottom boundaries to each depend on a partition. As shorthand, we'll often write $\mathfrak{S}$ to denote the set $\{\mathfrak{S}_{\lambda/\mu}|\lambda,\mu\}$ of lattice models with fixed Boltzmann weights and fixed side boundary conditions, but where the top and bottom boundaries can vary.

We say that a lattice model $\mathfrak{S}$ and a Hamiltonian operator $e^H$ \emph{match} if \begin{align*}Z(\mathfrak{S}_{\lambda/\mu}) =  {\Large{*}} \cdot\langle\mu|e^H|\lambda\rangle \hspace{20pt} \text{for all strict partitions $\lambda,\mu$,}\end{align*} where ${\Large{*}}$ represents any function of the Boltzmann weights of $\mathfrak{S}$ independent of $\lambda$ and $\mu$. In practice, these are easily computable as simple products involving the Boltzmann weights. Our condition for a Hamiltonian to match a lattice model is equivalent to requiring that the time-evolution of the Hamiltonian equals the row transfer matrix of the lattice model as an operator on the set of partitions.

Zinn-Justin \cite{ZinnJustin-six-vertex} gives a nice exposition of this connection in the case of the five-vertex model, and the resulting $\tau$-functions are Schur functions. Subsequently, Brubaker and Schultz \cite{Brubaker-Schultz} find Hamiltonians for certain six-vertex lattice models associated to Whittaker functions, and prove that for these models, the partition function is a supersymmetric Schur function. Korff \cite{Korff-cylindric-Hecke} studies another case of the six-vertex model, proves that his partition functions are characters of the Hecke algebra and cylindrical analogues, and connects the lattice model to vertex operators of Jing \cite{Jing-HL-vertex} for Hall-Littlewood polynomials.

Our results take steps towards turning these examples into a theory. The following is our main result for six-vertex models.

\begin{theorem}[Theorems \ref{Delta-lattice-Hamiltonian} and \ref{Gamma-lattice-Hamiltonian}, and Corollary \ref{Partition-Function-is-Supersymmetric-Schur}] \label{main-result-classical}
\leavevmode
\begin{enumerate}
\item[(a)] Given a six-vertex model $\mathfrak{S}$ as above, it matches a Hamiltonian from classical Fock space precisely when the Boltzmann weights of the lattice model are free fermionic.
\item[(b)] In this case, the partition function is (up to a simple factor) a supersymmetric Schur function.
\end{enumerate}
\end{theorem}

The free fermion point of the six-vertex model has long been known to be associated with special phenomena \cite{FanWu1, FanWu2}. In the five-vertex non-intersecting path model, free fermionic models are those with no attraction or repulsion between lattice paths, and their evolution is governed by entropy. Their partition functions can be expressed as determinants via the Lindstr\"{o}m-Gessel-Viennot Lemma \cite{Lindstrom-LGV, GesselViennot-LGV}. The free fermion point is also central to the solvability of the six-vertex model \cite{Baxter, KorepinBogoliubovIzergin, BBF-Schur-polynomials}.

Our proof of the first part of Theorem \ref{main-result-classical} uses Wick's Theorem \cite{AlexandrovZabrodin}, along with generating function manipulations. The second part follows directly from the first part using a result by Brubaker and Schultz \cite{Brubaker-Schultz}, and corrects a result from \cite{BBF-Schur-polynomials}.

The free fermionic six-vertex model is known to always be solvable \cite[Theorem~1]{BBF-Schur-polynomials} (after Baxter \cite{Baxter, Baxter-YBE} in the case where the weights are symmetric). Having a Hamiltonian operator for these models gives an alternative method to explore the partition functions and prove identities.

We work with two types of six-vertex models, which we call $\mathfrak{S}$ and $\mathfrak{S}^*$, and which are in a sense dual models. Both are four-parameter families which parametrize the free fermionic Boltzmann weights, and we show that both models match with Hamiltonians. The models are generalizations of the $\Delta$ and $\Gamma$ models explored in \cite{BBF-Schur-polynomials}. In addition, they each generalize the five-vertex vicious and osculating models defined in \cite{Fisher-vicious-walks, Brak-osculating}. One specialization sends $\mathfrak{S}$ to the vicious model and $\mathfrak{S}^*$ to the osculating model, while another specialization does the reverse.

This pair of five-vertex specializations is a special case of a more general phenomenon. We define an involution on our lattice models consisting of simple geometric manipulations that send $\mathfrak{S}$ and $\mathfrak{S}^*$ to each other, while doing the same to the vicious and osculating models. On the Hamiltonian side, this involution exactly matches a generalization of the symmetric function involution explored by Zinn-Justin \cite{ZinnJustin-six-vertex} that arises from particle-hole duality.

In addition, the existence of a correspondence between lattice models and Hamiltonians means that we can use the structure of one to prove identities for the other. In particular, our use of Wick's theorem for classical Fock space provides a Jacobi-Trudi identity for the free fermionic partition function. This can be seen as a generalization of the lattice version of the Lindstr\"{o}m-Gessel-Viennot Lemma and reduces to that result in the case of vicious walkers. Somewhat serendipitously, we are also able to use the Edrei-Thoma theorem to prove a positivity result for the free fermionic partition function.

The models $\mathfrak{S}$ require a particular choice of side boundary conditions. In Section \ref{boundary-conditions-section}, we use creation and deletion operators to give operator definitions for the free fermionic partition function with arbitrary boundary conditions. In general, this operator is not nicely behaved, but in the case of domain-wall boundary conditions, we get the following result:

\begin{theorem}[Theorem \ref{partition-function-all-boundaries}, Corollary \ref{domain-wall-corollary}] \label{domain-wall-proposition}
The partition function of the free fermionic lattice model with domain wall boundary conditions can be expressed as:
\begin{itemize}
    \item A supersymmetric Schur polynomial times an extra factor, and as
    \item A Schur polynomial times a (different) extra factor.
\end{itemize}
In both cases, the extra factors are (easily-computed) functions of the Boltzmann weights independent of $\lambda$ and $\mu$.
\end{theorem}

The second half of the paper relates Hamiltonians associated to Drinfeld twists of $q$-Fock space to six-vertex lattice models with \emph{charge}, an extra parameter giving a ``mod $n$'' behavior to the Boltzmann weights. This connection was first explored by Brubaker, Bump, Buciumas, and Gustafsson \cite{BBBG-Hamiltonian} in their study of metaplectic Whittaker functions. The resulting model has $2n+4$ vertices, a substantial increase in complexity from the standard six-vertex model. For a certain set of Boltzmann weights, Brubaker, Bump, Buciumas, and Gustafsson showed that the model is solvable via a module of $U_q(\widehat{\mathfrak{gl}}(1|n))$.

Reshetikhin \cite{Reshetikhin-twist} defined a large class of Drinfeld twists of quantum groups. When applied to the $q$-Fock space studied by Kashiwara, Miwa, and Stern \cite{Kashiwara-Miwa-Stern}, a certain subset of these twists are \emph{shift invariant} in the sense that their defining \emph{wedge relations} only depend on the difference between two indices and not the indices themselves. Brubaker, Bump, Buciumas, and Gustafsson show that the associated Hamiltonians match the metaplectic lattice models from \cite{BBCFG-metaplectic}, and so we will call these spaces \emph{metaplectic Fock spaces}. The rank of the quantum group is the same as the modulus on the charged models, and in the case of the models in \cite{BBBG-Hamiltonian}, the Drinfeld twist gives the Gauss sums for the metaplectic Whittaker functions.

First, we prove a criterion for solvability of charged lattice models. Solvability turns out to be closely related to a condition we call the \emph{generalized free fermion condition}, which consists of the free fermion condition at ``zero charge'' and an additional \emph{charge condition}. In the case $n=1$, these conditions reduce to the classical free fermion condition.

\begin{theorem}[Theorem \ref{charge-solvability-theorem}]
The six-vertex model with charge is solvable in precisely two cases: \begin{enumerate}
    \item The Boltzmann weights satisfy the generalized free fermion condition associated to a metaplectic Fock space.
    \item The Boltzmann weights satisfy the conditions (\ref{independence-condition}) and (\ref{non-free-fermion-charge-equation}).
\end{enumerate}
\end{theorem}

The conditions in the second case are quite restrictive and the solution is not very interesting as many of the R-vertex weights are 0. Therefore, for practical purposes, solvability is equivalent to the generalized free fermion condition.

Remarkably, the generalized free fermion condition is also precisely the condition required for a six-vertex model with charge to match a Hamiltonian from a metaplectic Fock space. It was quite unexpected to see the same condition arise naturally from these two very different computations, and we do not currently have an explanation for this phenomenon.

We define charged models $\mathfrak{S}^q$ and $\mathfrak{S}^{*,q}$ which parametrize the generalized free fermionic models. These models generalize $\mathfrak{S}$ and $\mathfrak{S}^*$, as well as the charged lattice models in \cite{BBBG-Hamiltonian}. We show that $\mathfrak{S}^q$ and $\mathfrak{S}^{*,q}$ match Hamiltonian operators. More precisely,

\begin{theorem}[Theorems \ref{charged-Delta-lattice-Hamiltonian}, \ref{charged-Gamma-lattice-Hamiltonian}, and \ref{Partition-Function-is-Supersymmetric-LLT}] \label{main-result-quantum}
\leavevmode
\begin{enumerate}
\item[(a)] The six-vertex model with charge matches a Hamiltonian operator on $q$-Fock space precisely when its Boltzmann weights satisfy the generalized free fermion condition.
\item[(b)] In this case, the partition function is (up to a simple factor) a supersymmetric LLT polynomial. All supersymmetric LLT polynomials with nonzero parameters appear as such partition functions.
\end{enumerate}
\end{theorem}
Brubaker, Buciumas, Bump, and Gustafsson showed that their model gives a supersymmetric LLT polynomial; our contribution here is to show that this is true of \emph{all} generalized free fermionic models, and that these are precisely the models associated to Hamiltonians on metaplectic Fock spaces. The generalized free fermionic models give us most, although not all, values of supersymmetric LLT polynomials as their partition functions. Finally, we use Hamiltonians to prove a Cauchy identity for (skew) supersymmetric LLT polynomials. This generalizes results of Lam \cite{Lam-LLT} and Brubaker, Bump, Buciumas, and Gustafsson \cite{BBBG-Hamiltonian}.

In general, the relationship between Hamiltonian operators and solvability is unclear. Both are phenomena involving representations of quantum groups; however, they are different quantum groups! The Hamiltonian operator is associated to a quantum affine algebra, while the \emph{$R$-matrix} involved in the solvability of our models arises from a quantum affine \emph{super}algebra. Furthermore, the ways that these representations interact with the lattice model are quite different. A vector in Fock space in a sense represents all the vertical edges in a single row of the lattice model at once, while the $R$-matrix method gives an intertwiner at each vertex of a pair of quantum group modules (and a module interpretation for the vertical edges in \cite{BBBG-metaplectic-solvability} is not known).

One nice potential consequence of the relationship of solvability with Hamiltonians is that Hamiltonians could provide a better method for generating algebraic conditions for solvability. It is often difficult to determine whether a complicated lattice model is solvable, and in the case of charged models our proof of solvability is substantially more computationally intense than our proof of a matching Hamiltonian. It is unclear whether a connection between these phenomena exists more generally

Other interesting directions are the study of vertex operators associated to these lattice models, which can be obtained from our half-vertex operators using the procedure in \cite[Chapter~14]{Kac-book}, and cylindrical boundary conditions, using our results from Section \ref{side-boundary-part-fun-section}.

Sections \ref{background-section}-\ref{boundary-conditions-section} are on the topic of classical Fock space and the six-vertex model without charge. Section \ref{background-section} gives preliminaries on Fock space, Hamiltonian operators, and symmetric functions corresponding to Hamiltonians. We give several symmetric function identities, many of which are also proved in \cite{Lam} or \cite{ZinnJustin-six-vertex}. Then in Section \ref{six-vertex-section}, we define our lattice models and describe two different relationships between their partition function. In Section \ref{proof-of-classical-main-result-section}, we prove Theorem \ref{main-result-classical}(a). Section \ref{corollaries-section} covers several topics relating to the free fermionic partition function, including a proof of Theorem \ref{main-result-classical}(b), involutions, identities, and positivity. In Section \ref{boundary-conditions-section}, we find a Fock space operator that matches with any boundary conditions, and prove Theorem \ref{domain-wall-proposition}.

Sections \ref{q-Fock-space-section}-\ref{charged-partition-function-section} concern $q$-Fock spaces and six-vertex models with charge. Section \ref{q-Fock-space-section} introduces $q$-Fock space and the action of Hamiltonian operators, while Section \ref{charged-models-section} defines our charged lattice models and the generalized free fermion condition. We present the solvability criterion for charged models in Section \ref{solvability-section}, with some computational details in Appendix \ref{appendix-equations}. The proof of Theorem \ref{main-result-quantum}(a) is in Section \ref{proof-of-quantum-main-result-section}, and the proof of Theorem \ref{main-result-quantum}(b), along with the Cauchy identity for supersymmetric LLT polynomials, is in Section \ref{charged-partition-function-section}.

\textbf{Acknowledgements:} This work was partially supported by NSF RTG grants DMS-1745638 and DMS-1937241, and by the University of Minnesota's Doctoral Dissertation Fellowship. I would like to thank Leonid Petrov, with whom I shared the results from Section \ref{proof-of-classical-main-result-section} in February 2021, for his time and insightful comments. I would also like to thank the members of the Solvable Lattice Seminar for allowing me to share many of these results in June 2021; in particular Daniel Bump, Travis Scrimshaw, Slava Naprienko, and Valentin Buciumas for helpful discussions. Thanks also to Daniel Bump for suggestions relating to Section \ref{berele-regev-section}, to Richard Kenyon for answering questions about the free fermion point, and to Darij Grinberg for catching some typos and confusing wording. This work is part of my Ph.D. thesis research at the University of Minnesota, written under the guidance and support of Ben Brubaker.

\textbf{Author's Note 2021:} In the final stages of writing this paper, the author was informed of another, independent work by Aggarwal, Borodin, Petrov, and Wheeler \cite{AggarwalBorodinPetrovWheeler} which also explores the free fermionic six vertex model. Both this paper and that one study four-parameter families of polynomials, but two of their parameters are column parameters, whereas all of ours are row parameters. This allows our model to have more direct ties to Hamiltonians and emphasize the role of the Heisenberg algebra, although both papers use Wick's theorem in the analysis of the associated partition function. Their paper doesn't consider lattice models with charge or $q$-Fock space Hamiltonians, and instead has many interesting applications involving probability measures and domino tilings.

\textbf{Authors Note 2024:} We give a brief chronology (to the author's best knowledge) on lattice models for supersymmetric LLT polynomials, since there were many independent developments between 2018 and 2021. These polynomials were defined by Lam \cite{Lam-LLT} in 2005, and Brubaker, Buciumas, Bump, and Gustafsson \cite{BBBG-Hamiltonian} gave a lattice model for a particular specialization in 2018. Three groups independently gave the same lattice model for all (ordinary) LLT polynomials: Curran, Yost-Wolff, Zhang, and Zhang \cite{CYWZZ-LLT} (2019), Corteel, Gitlin, Keating, and Meza \cite{Corteel-LLT} (Dec. 2020), and Aggarwal, Borodin, and Wheeler \cite{Aggarwal-Borodin-Wheeler} (Jan. 2021). \cite{CYWZZ-LLT} proved solvability only in special cases, while \cite{Corteel-LLT} has a combinatorial proof, and \cite{Aggarwal-Borodin-Wheeler} has a proof using representation theory.

The present paper was first posted to the ArXiv in September 2021 and contains lattice models for all supersymmetric LLT polynomials with nonzero parameters (Theorem \ref{Partition-Function-is-Supersymmetric-LLT}. This includes almost all specializations, but leaves out the ordinary LLT polynomials. To the author's knowledge, this was the most general such lattice model at the time it was first posted. In addition, the Cauchy identity for supersymmetric LLT polynomials (Proposition \ref{super-LLT-Cauchy}) is fully general, and to the author's knowledge was the first appearance of such a formula.

Very shortly after, Curran, Frechette, Yost-Wolff, Zhang, and Zhang \cite{CFYWZZ-super-LLT} (Oct. 14, 2021) and Gitlin and Keating \cite{GitlinKeating-super-LLT} (Oct. 19, 2021) gave lattice models for all supersymmetric LLT polynomials. Both papers included proofs of the Cauchy identity, the latter as a consequence of the Yang-Baxter equation, while the former used a similar approach to the present paper and then leveraged the Cauchy identity to prove the Yang-Baxter equation.

\section{Fock Space and Hamiltonian Operators} \label{background-section}

In this section, we define (classical) Fock space and its Hamiltonian operators. We explore the ways in which Hamiltonians generalize symmetric and supersymmetric function theory in terms of an involution as well as Jacobi-Trudi, Cauchy, and Pieri rules.

\subsection{Partitions}

A partition $\lambda$ of length $\ell$ is a weakly decreasing sequence of numbers \[\lambda_1\ge \lambda_2\ge\ldots\lambda_\ell\ge 0.\] It is called \emph{strict} if all the inequalities above except the last one are strict.

Notice that $\lambda$ can be padded by trailing zeroes. In fact, we will often need pairs of partitions of the same length, so we will often do this.

Let $\rho := \rho_\ell = (\ell-1,\ell-2,\ldots,1,0)$. We will often \emph{shift} a partition by $\rho$: \[\lambda\pm\rho := \lambda\pm\rho_{\ell(\lambda)} = (\lambda_1\pm (\ell(\lambda)-1), \lambda_2\pm (\ell(\lambda)-2), \ldots, \lambda_{\ell(\lambda)}).\] $\rho$-shifting has the nice property that \[\lambda \text{ is a partition} \hspace{20pt} \text{if and only if} \hspace{20pt} \lambda+\rho \text{ is a strict partition}.\] Let $\lambda'$ be the conjugate partition, $\lambda'_i = |\{k|\lambda_k\ge i\}|$. Notice that the length $\ell(\lambda')$ is not well-defined; this shouldn't be a problem since we may choose $\lambda'$ to have any number of trailing zeroes.

If $\lambda$ is a strict partition, choose an integer $M\ge \lambda_1$. Set $\ba{\lambda}$ to be the partition obtained by reversing $\lambda$ in the range $[0,M]$ and swapping its parts and nonparts: \[\ba{\lambda} = \rho_{M+1} \setminus (M-\lambda_1, M-\lambda_2,\ldots,M-\lambda_{\ell(\lambda)}),\] where the $\setminus$ symbol refers to set subtraction on parts of the partition.

\begin{lemma} \label{partition-lemma}
For all strict partitions $\lambda$ and all $M\ge \lambda_1$, $\ba{\lambda}-\rho = (\lambda-\rho)'$.
\end{lemma}

\begin{proof}
Let $\ell = \ell(\lambda)$, and let $\mu_1 > \mu_2 > \cdots > \mu_{M+1-\ell}$ be the strict partition made up of all the integers in $[0,M]$ that are not parts of $\lambda$. Then we have \begin{align*} (\lambda-\rho)'_i &= |\{k|\lambda_k - (\ell-k) \ge i\}| \\&= |\{k|\text{ there exist $\ge i$ parts $\mu_j$ of $\mu$ with $\lambda_k>\mu_j$}\}| \\&= |\{k|\lambda_k > \mu_{M-\ell+2-i}\}| \\&= M - \mu_{M-\ell+2-i} - (M-\ell+1-i) \\&= (\ba{\lambda}-\rho)_i,\end{align*} where the last equality holds since $\ell(\mu) = M-\ell+1$, $(\ba{\lambda})_i = M - \mu_{M-\ell+2-i}$, and $(\rho_{\ell(\mu)})_i = M-\ell+1-i$.
\end{proof}

Let $\lambda,\mu$ be partitions with the same length such that $\lambda_i\ge \mu_i$ for all $i$. Then we call the pair $(\lambda,\mu)$ a \emph{skew} partition and denote it $\lambda/\mu$.

\subsection{Fock Space}

Let us define the Clifford algebra \[A = \langle \psi_i^*, \psi_i| i\in\Z-\frac{1}{2}\rangle,\] with relations \begin{align*}\psi_i\psi_j + \psi_j\psi_i = 0, \hspace{20pt} \psi^*_i\psi^*_j + \psi^*_j\psi^*_i = 0, \hspace{20pt} \psi_i\psi^*_j + \psi^*_j\psi_i = \delta_{i,j}.\end{align*} Let $\psi^*(t) = \sum_{k\in\Z-1/2} \psi^*_k t^{k+1/2}$.

The Fock space $\mathcal{F}$ and its dual $\mathcal{F}^*$ are both $A$-modules. Let $\mathcal{W} = \left(\oplus_{i\in\Z-1/2} \C\psi_i\right) \oplus \left(\oplus_{i\in\Z-1/2} \C\psi_i^*\right)$. We call elements of $\mathcal{W}$ \emph{free fermions}.

Define subspaces $\mathcal{W}_{ann} = \left(\oplus_{i<0} \C\psi_i\right) \oplus \left(\oplus_{i > 0} \C\psi_i^*\right)$ and $\mathcal{W}_{cr} = \left(\oplus_{i>0} \C\psi_i\right) \oplus \left(\oplus_{i<0} \C\psi_i^*\right)$. Then $\mathcal{F} := A/A\mathcal{W}_{ann}$ is a left $A$-module, while $\mathcal{F}^* := \mathcal{W}_{cr}A\backslash A$ is a right $A$-module.

$\mathcal{F}$ is a cyclic module generated by the vector $|0\rangle := 1 \mod A\mathcal{W}_{ann}$ and $\mathcal{F}^*$ is a cyclic module generated by the vector $\langle 0| := 1 \mod \mathcal{W}_{cr}A$.

One can represent each of these vectors by the particle diagram

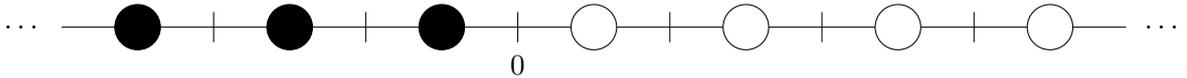
\vspace{10pt}

\begin{tikzpicture}
  \draw (0,0)--(14,0);
  \draw (2,0.2)--(2,-0.2);
  \draw (4,0.2)--(4,-0.2);
  \draw (6,0.2)--(6,-0.2);
  \draw (8,0.2)--(8,-0.2);
  \draw (10,0.2)--(10,-0.2);
  \draw (12,0.2)--(12,-0.2);
  \draw[fill=black] (1,0) circle (.3);
  \draw[fill=black] (3,0) circle (.3);
  \draw[fill=black] (5,0) circle (.3);
  \draw[fill=white] (7,0) circle (.3);
  \draw[fill=white] (9,0) circle (.3);
  \draw[fill=white] (11,0) circle (.3);
  \draw[fill=white] (13,0) circle (.3);
  \node at (6,-0.5) {0};
  \node at (-0.5,0) {\ldots};
  \node at (14.5,0) {\ldots};
\end{tikzpicture}

For a strict partition $\lambda$, we define \[|\lambda\rangle := \psi^*_{\lambda_1-\frac{1}{2}}\psi^*_{\lambda_2-\frac{1}{2}}\ldots \psi^*_{\lambda_n-\frac{1}{2}}|0\rangle \hspace{20pt} \text{and} \hspace{20pt} \langle\lambda| := \langle 0|\psi_{\lambda_n-\frac{1}{2}}\ldots \psi_{\lambda_2-\frac{1}{2}}\psi_{\lambda_1-\frac{1}{2}}.\]

We represent the basis vectors of $\mathcal{F}$ and $\mathcal{F}^*$ as particle diagrams as well. For example, the partition $\lambda=(3,1)$ is represented by the following diagram.

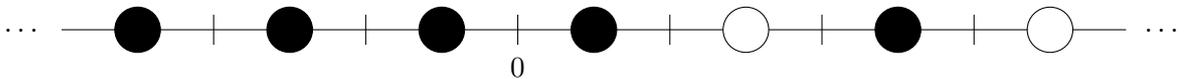
\vspace{10pt}

\begin{tikzpicture}
  \draw (0,0)--(14,0);
  \draw (2,0.2)--(2,-0.2);
  \draw (4,0.2)--(4,-0.2);
  \draw (6,0.2)--(6,-0.2);
  \draw (8,0.2)--(8,-0.2);
  \draw (10,0.2)--(10,-0.2);
  \draw (12,0.2)--(12,-0.2);
  \draw[fill=black] (1,0) circle (.3);
  \draw[fill=black] (3,0) circle (.3);
  \draw[fill=black] (5,0) circle (.3);
  \draw[fill=black] (7,0) circle (.3);
  \draw[fill=white] (9,0) circle (.3);
  \draw[fill=black] (11,0) circle (.3);
  \draw[fill=white] (13,0) circle (.3);
  \node at (6,-0.5) {0};
  \node at (-0.5,0) {\ldots};
  \node at (14.5,0) {\ldots};
\end{tikzpicture}

$\{|\lambda\rangle\}$ and $\{\langle\lambda|\}$ form a set of dual bases of $\mathcal{F}$ and $\mathcal{F}^*$ with respect to the bilinear form \[\mathcal{F}^*\otimes_A \mathcal{F} \to \C\] defined by \[\langle\lambda|\otimes_A |\mu\rangle \mapsto \langle \lambda|\mu\rangle := \delta_{\lambda,\mu},\] extended linearly.

We can use a similar ``bra-ket'' notation to write more complicated pairings: we write $\langle\mu|h|\lambda\rangle$ to represent the image under the bilinear form of the quantity $\langle\mu|h \otimes_A 1|\lambda\rangle = \langle\mu|1 \otimes_A h|\lambda\rangle$.

This symmetric bilinear form gives rise to a linear form $\langle\cdot\rangle$ on $A$, called the \emph{vacuum expectation value.} (In many sources, these definitions are done in reverse).

We define: \[\langle a\rangle := \langle 0|a|0\rangle.\] In particular, we have \[\langle 1\rangle = 1, \hspace{10pt} \langle \psi_i\psi_j\rangle = \langle \psi^*_i\psi^*_j\rangle = 0,\] \[ \langle \psi_i\psi^*_j\rangle = \begin{cases} 1, & \text{if $i=j<0$,} \\ 0, & \text{otherwise.}\end{cases} \hspace{10pt} \langle \psi^*_i\psi_j\rangle = \begin{cases} 1, & \text{if $i=j>0$,} \\ 0, & \text{otherwise.}\end{cases}\]

We can use the vacuum expectation value to define the \emph{normal ordering}.

\[:\psi_i\psi^*_j: \hspace{5pt} := \psi_i\psi^*_j - \langle \psi_i\psi^*_j\rangle = \begin{cases} \psi_i\psi^*_j, & \text{if $i>0$,} \\ -\psi^*_j\psi_i, & \text{if $i<0$.}\end{cases}\] \[:\psi_j^*\psi_i: \hspace{5pt} := \psi^*_j\psi_i - \langle \psi^*_j\psi_i\rangle = \begin{cases} \psi^*_j\psi_i, & \text{if $i<0$,} \\ -\psi_i\psi^*_j, & \text{if $i>0$.}\end{cases}\] Note that unless $i=j$, $\psi_i\psi^*_j = -\psi^*_j\psi_i$. On the other hand, $\psi_i\psi^*_i|\lambda\rangle\ne 0$ whenever $\lambda$ does not have a part of size $i+1/2$, whereas $\psi^*_i\psi_i|\lambda\rangle\ne 0$ whenever $i<0$ or $i>0$ and $\lambda$ does have a part of size $i+1/2$. This would lead to trivial infinite quantities in some of our upcoming definitions. However, using $:\psi_i\psi^*_j: = -\psi^*_j\psi_i$ removes this complication since $:\psi_i\psi^*_i:|\lambda\rangle\ne 0$ only for finitely many $i$, precisely those $i>0$ where $\lambda$ has a part of size $i+1/2$.

One can therefore define the Lie algebra, \[\mathfrak{gl}(\infty) := \left\{\sum_{ij} a_{ij} :\psi_i\psi^*_j: | \exists N \text{ such that } a_{ij}=0 \text{ if } |i-j|>N \right\} \oplus \C\cdot 1.\] With this definition, $\mathcal{F}$ and $\mathcal{F}^*$ are both $\mathfrak{gl}(\infty)$-modules. They are reducible, but decompose into irreducible representations \[\mathcal{F} = \bigoplus_{\ell\in\Z} \mathcal{F}_\ell, \hspace{20pt} \mathcal{F}^* = \bigoplus_{\ell\in\Z} \mathcal{F}^*_\ell,\] \[\mathcal{F}_\ell = \mathfrak{gl}(\infty)|\ell\rangle, \hspace{20pt} \mathcal{F}^*_\ell = \langle\ell|\mathfrak{gl}(\infty)\]

Next, we will define important elements in $\mathfrak{gl}(\infty)$ called \emph{Hamiltonian operators}.

For $n\in\Z$, let \[J_n = \sum_{i\in\Z-\frac{1}{2}} :\psi^*_{i-n}\psi_i:.\] These are called \emph{current operators}, and they generate a Heisenberg algebra $\mathcal{H} := \langle J_m|m\in\Z, m\ne 0\rangle$ since $[J_m,J_n] = m\delta_{m,-n}$. Now, let $\{s_k|k\in \Z\setminus\{0\}\}$ be a doubly infinite family of parameters. Let \[H_+ = \sum_{k\ge 1} s_kJ_k, \hspace{20pt} e^{H_+} = \sum_{m\ge 0} \frac{H_+^m}{m!}.\] $e^{H_+}$ is the \emph{Hamiltonian operator} of \cite{JimboMiwa-solitons, Kac-book, Lam, Brubaker-Schultz}. Similarly, we define \[H_- := \sum_{k\ge 1} s_{-k}J_{-k}.\]

We will also sometimes write $s_k = \sum_{j=1}^N s^{(j)}_k$, where the $s^{(j)}_k$ are indeterminates. In this case, we have \[H_{\pm} = \sum_{j=1}^N \phi_{\pm j}, \hspace{20pt} \text{where} \hspace{20pt}\phi_{\pm j} = \sum_{k\ge 1} s^{(j)}_{\pm k} J_{\pm k}.\]  Note that all the $\phi_j$ commute. If we want to make the parameters $s_{\pm k}^{(j)}$ clear for a Hamiltonian operator, we write \[H_\pm = H_\pm(s_{\pm 1}^{(1)},\ldots,s_{\pm 1}^{(N)}, s_{\pm 2}^{(1)},\ldots,s_{\pm 2}^{(N)}, \ldots).\]

The action of a current operator $J_k$ on a vector is to move a particle $k$ spots to the left. The action of $e^{H_+}$ is to move any number of particles any number of spaces to the left. The parameters $s_k$ keep track of which moves we have done. Similarly, the action of $e^{H_-}$  allows us to move any number of particles any number of spaces to the right.

The following result is both useful and classical. It appears in many forms \cite{AlexandrovZabrodin}.

\begin{proposition}[Wick's Theorem] \begin{align*}\langle \psi_{i_1}\ldots \psi_{i_r} e^{H_\pm} \psi^*_{j_1}\ldots \psi^*_{j_s}\rangle = \begin{cases} \det_{1\le p,q\le r} \langle \psi_{i_p}e^{H_\pm}\psi^*_{j_q}\rangle, & \text{if } r=s \\ 0, & \text{otherwise}.\end{cases}\end{align*}
\end{proposition}

One of the best motivations to study Hamiltonian operators is the boson-fermion correspondence. For all $\ell\in\Z$, let $V_\ell\cong \C[z]$. Let $\mathcal{H}$ act on $V_\ell$ by the \emph{bosonic action}: \[J_k\cdot P := \begin{cases} k\cdot\frac{\partial P}{\partial ks_k}, & k>0 \\ s_k\cdot P, & k<0,\end{cases}\]

\begin{proposition}[Boson-Fermion Correspondence] \cite[Theorem~1.1]{JimboMiwa-solitons} \label{boson-fermion-correspondence}
The following map \[\mathcal{F}_\ell\to V_\ell, \hspace{20pt} a|0\rangle \mapsto \langle\ell|e^{H_+} a|0\rangle\] is an isomorphism of $\mathcal{H}$-modules. In other words, \[h\langle \ell| e^{H_+} a|0\rangle = \langle\ell|e^{H_+} h a|0\rangle, \hspace{20pt} a\in A, a|0\rangle\in \mathcal{F}_\ell, h\in \mathcal{H}\] where the action by $h$ on the left side is bosonic, while on the right side it is fermionic.
\end{proposition}

There is another way to write $\mathcal{F}_\ell$. Let $W = \bigoplus_{i\in\Z} \langle v_i\rangle$ and \[\bigwedge^\infty W = v_{i_1}\wedge v_{i_2}\wedge\ldots,\] with the usual wedge relation $v_i\wedge v_j = -v_j\wedge v_i$. Then \[\mathcal{F}_\ell = \left\{v_{i_1}\wedge v_{i_2}\wedge\ldots \in \bigwedge^\infty W | i_m = \ell-m \text{ for all } m >> 0\right\},\] and the action of current operators can be expressed as \[J_k\cdot (v_{m_1}\wedge v_{m_2}\wedge\ldots) = \sum_{i\ge 1} (v_{m_1}\wedge\ldots \wedge v_{m_{i-1}} \wedge v_{m_i-k} \wedge v_{m_{i+1}}\wedge\ldots).\] We'll see a more general version of this action in Section \ref{q-Fock-space-section}.

\subsection{Hamiltonians and symmetric functions} \label{Hamiltonian-symmetric-functions-section}

We will work with a set of functions that naturally arise from Hamiltonians and generalize some common symmetric and supersymmetric functions such as power sum, homogeneous, elementary, and Schur polynomials. This approach was first taken by Lam \cite{Lam} and Zinn-Justin \cite{ZinnJustin-six-vertex}, and most of the results in this section were proved by one or both of them. An important reference for symmetric functions computations is \cite[Chapter~I]{Macdonald-symmetric-functions}.

A similar idea is due to Korff \cite{Korff-vicious-osculating} and Gorbounov-Korff \cite{GorbounovKorff}, who used operator analogues of symmetric functions to study quantum cohomology via vicious and osculating walkers.

We will use similar notation for our generalizations as for the classical symmetric functions. To avoid confusion, we will always use parentheses for the generalized functions and square brackets for symmetric and supersymmetric functions.

Fix a set of parameters $\boldsymbol{s}_+ := \{s_k^{(j)}, 1\le j\le N, k\ge 1\}$. We want the negative-index parameters to have a particular relationship with the positive index parameters: \[\boldsymbol{s}_- := \{s_k^{(j)}, 1\le j\le N, k\le -1\}, \hspace{20pt} \text{where} \hspace{20pt} s_{-k}^{(j)} = (-1)^{k-1}s_k^{(j)}.\] We will explore symmetric function analogues in the ring $\C[s_k^{(j)}|k\in\Z\setminus\{0\}, 1\le j\le N]$. In the constructions to follow, specializing $s_k^{(j)} = \frac{1}{k} x_j^k, k>0$ produces the classical symmetric functions, while specializing $s_k^{(j)} = \frac{1}{k} (x_j^k + (-1)^k y_j^k), k>0$ produces the supersymmetric functions.

For a partition $\lambda$, let $z_\lambda = \prod_{i\ge 1} i^{m_i} m_i!$, where $m_i$ is the number of parts of $\lambda$ of size $i$.

Let \[s_{\pm k} := s_{\pm k}(\boldsymbol{s}_\pm) = \sum_{j=1}^N s_{\pm k}^{(j)}, \hspace{20pt} p_{\pm k} := p_{\pm k}(\boldsymbol{s}_\pm) = ks_{\pm k}, \hspace{20pt} k\ne 0,\] \[ s_{\pm\lambda} := s_{\pm\lambda_1}\cdots s_{\pm\lambda_{\ell(\lambda)}}, \hspace{20pt} p_{\pm\lambda} := p_{\pm\lambda_1}\cdots p_{\pm\lambda_{\ell(\lambda)}}, \hspace{20pt} \lambda \text{: partition},\] \[h_{\pm k} := \sum_{\lambda\vdash k} z_\lambda^{-1} p_{\pm\lambda}, \hspace{20pt} e_{\pm k} := \sum_{\lambda\vdash k} (-1)^{|\lambda|-\ell(\lambda)}z_\lambda^{-1} p_{\pm\lambda}, \hspace{20pt} k\ge 1,\] \[s_0=p_0=0, \hspace{20pt} h_0=e_0=1,\] and let $\omega$ be the involution $\omega(s_k^{(j)}) = (-1)^{k-1}s_k^{(j)}$, extended algebraically. In particular, $\omega(h_k) = e_k$.

As above, let $H_{\pm} = H_{\pm}(\boldsymbol{s}_\pm) = \sum_{k\ge 1}\sum_{j=1}^N s_{\pm k}^{(j)}J_{\pm k}$.

\begin{lemma}[Duality] \label{Hamiltonian-duality-lemma}
\begin{align}\omega\left(\langle\mu|e^{H_+}|\lambda\rangle\right) = \langle\lambda|e^{H_-}|\mu\rangle. \label{Hamiltonian-pos-neg-duality}\end{align}
\end{lemma}

\begin{proof}
First, if $k\ne 0$, \begin{align*}\langle \mu|J_k|\lambda\rangle &= \sum_{i\in \Z-1/2}\langle \mu|:\psi^*_{i-k}\psi_i:|\lambda\rangle \\&= \sum_{i\in \Z-1/2}\begin{cases} 1, & \text{there exists a partition $\nu$ such that $\nu\cup (i) = \lambda, \nu\cup (i-k) = \mu$} \\ 0, & \text{otherwise}.\end{cases} \\&= \sum_{i\in \Z-1/2}\begin{cases} 1, & \text{there exists a partition $\nu$ such that $\nu\cup (i+k) = \lambda, \nu\cup (i) = \mu$} \\ 0, & \text{otherwise}.\end{cases} \\&= \sum_{i\in \Z-1/2}\langle \lambda|:\psi^*_{i+k}\psi_i:|\mu\rangle \\&= \langle \lambda|J_{-k}|\mu\rangle,\end{align*} and so \[\omega\left(\langle \mu|s_kJ_k|\lambda\rangle\right) = \omega(s_k)\omega\left(\langle \mu|J_k|\lambda\rangle\right) = s_{-k}\langle \lambda|J_{-k}|\mu\rangle = \langle \lambda|s_{-k}J_{-k}|\mu\rangle.\] Then, (\ref{Hamiltonian-pos-neg-duality}) follows by linearity.
\end{proof}

Let \[S(t) := \sum_{k\ge 1} s_kt^k, \hspace{20pt} P(t) := S'(t) = \sum_{k\ge 1} p_k t^{k-1},\] and also \[H(t) := \sum_{k\ge 0} h_kt^k, \hspace{20pt} E(t) := \sum_{k\ge 0} e_kt^k.\]

\begin{lemma}
\begin{align} S(t) = \log H(t), \label{S-equals-log-H}\end{align} \begin{align} h_k = \langle (0)|e^{H_+}|(k)\rangle. \label{h_k-as-Hamiltonian}\end{align} and similarly \begin{align} -S(-t) = \log E(t), \label{S-equals-log-E}\end{align} \begin{align} e_k = \langle (k)|e^{H_-}|(0)\rangle. \label{e_k-as-Hamiltonian}\end{align}
\end{lemma}

\begin{proof}
We start by proving (\ref{S-equals-log-H}). Note that $z_\lambda$ is the product of the parts of $\lambda$ times the number of permutations on the parts of $\lambda$ that fix $\lambda$. In other words, $\ell(\lambda)! z_\lambda^{-1}$ is the number of compositions of $|\lambda|$ that rearrange to $\lambda$, divided by the product of the parts of $\lambda$. Using this,\[h_k = \sum_{\lambda\vdash k} z_\lambda^{-1} p_\lambda = \sum_{r\ge 0} \frac{1}{r!}\sum_{q_1+\ldots+q_r = k} s_{q_1}\ldots s_{q_r},\] so \[H(t) = \sum_{k\ge 0} h_kt^k = \sum_{r\ge 0} \frac{1}{r!}\left(\sum_{q_1,\ldots,q_r \ge 1} s_{q_1}\ldots s_{q_r}\right) t^{q_1+\ldots+q_r} = \exp(S(t)),\] by the definition of the formal exponential.

(\ref{h_k-as-Hamiltonian}) follows from (\ref{S-equals-log-H}) since \[\langle (k)|e^{H_-}|(0)\rangle = \sum_{k\ge 0}\sum_{r\ge 0} \frac{1}{r!}\left(\sum_{q_1+\ldots+q_r = k} s_{q_1}\ldots s_{q_r}\right).\]

Similarly, \[e_k = \sum_{\lambda\vdash k} z_\lambda^{-1} (-1)^{k-\ell(\lambda)} p_\lambda = \sum_{r\ge 0} \frac{1}{r!} (-1)^{k-r}\sum_{q_1+\ldots+q_r = k} s_{q_1}\ldots s_{q_r},\] so \[E(t) = \sum_{k\ge 0} e_kt^k = \sum_{r\ge 0} \frac{1}{r!} (-1)^r\left(\sum_{q_1,\ldots,q_r \ge 1} s_{q_1}\ldots s_{q_r}\right) (-t)^{q_1+\ldots+q_r} = \exp(-S(-t)).\]

(\ref{e_k-as-Hamiltonian}) follows from (\ref{S-equals-log-E}) since \begin{align*}\langle (0)|e^{H_+}|(k)\rangle &= \sum_{r\ge 0} \frac{1}{r!}\left(\sum_{q_1+\ldots+q_r = k} (-1)^{q_1-1}s_{q_1}\ldots (-1)^{q_r-1}s_{q_r}\right) \\&= \sum_{r\ge 0} \frac{1}{r!} (-1)^{k-r}\left(\sum_{q_1+\ldots+q_r = k} s_{q_1}\ldots s_{q_r}\right).\end{align*}
\end{proof}

Let \[H := (h_{i-j})_{0\le i,j\le n}, \hspace{20pt} E := ((-1)^{i-j} e_{i-j})_{0\le i,j\le n}.\]

\begin{corollary}
\[H(t)E(-t) = 1,\] and \[\sum_{r=0}^n (-1)^r e_r h_{n-r} = 0 \hspace{20pt} \text{for all $n\ge 1$}\]. Furthermore, $E=H^{-1}$. 
\end{corollary}

\begin{proof}
By the previous lemma, \[H(t)E(-t) = \exp(S(t))\exp(-S(t)) = 1,\] and the second equation follows from taking coefficients. The final statement is obtained from a matrix multiplication: \[(H\cdot E)_{i,j} = \sum_{k=0}^n (-1)^{k-j} e_{k-j} h_{i-k} = \sum_{k=j}^i (-1)^{k-j} e_{k-j} h_{i-k} = \sum_{r=0}^{i-j} (-1)^r e_r h_{i-j-r} = \delta_{ij}.\]
\end{proof}

Now let $\lambda$ and $\mu$ be partitions with $\ell(\lambda)=\ell(\mu)$, $\ell(\lambda') = \ell(\mu')$, and $\ell(\lambda)+\ell(\lambda') = n$. (Note that $n$ can be made arbitrarily large, and the partitions can be buffered with trailing zeroes, so this is really no restriction.)

\begin{lemma} We have
\begin{align*}\det\left(h_{\lambda_i-\mu_j-i+j}\right) = \det\left(e_{\lambda_i'-\mu_j'-i+j}\right)\end{align*}
\end{lemma}

\begin{proof}
The proof is exactly the same as in Macdonald \cite[pp.~22-23]{Macdonald-symmetric-functions} 
\end{proof}

Now let $\sigma_{\lambda/\mu} := \langle\mu+\rho|e^{H_+}|\lambda+\rho\rangle$. This is the generalized Schur function, which we denote by $\sigma$ so as to avoid confusion with the $s_k$. Note the $\rho$ shift, since Hamiltonians deal with strict partitions. $\sigma_{\lambda/\mu}$ is 0 unless $\lambda$ and $\mu$ have the same length $\ell$. Let $\sigma_\lambda := \sigma_{\lambda/(0,\ldots,0)} = \langle\rho|e^{H_+}|\lambda+\rho\rangle$. Note in particular that $h_k(\mathbf{s_+}) = s_{(k)/(0)}$. We call $\sigma_{\lambda/\mu}$ the \emph{$\tau$ function} corresponding to $\lambda+\rho$ and $\mu+\rho$.

\begin{proposition}[Jacobi-Trudi, Von N\"{a}gelsbach–Kostka identities] \label{Hamiltonian-Jacobi-Trudi}
\begin{align*}\sigma_{\lambda/\mu} = \det_{1\le i,j\le \ell} h_{\lambda_i-\mu_j-i+j} = \det_{1\le i,j\le \ell} e_{\lambda'_i-\mu'_j-i+j}.\end{align*}
\end{proposition}

\begin{proof}
The second equality is the previous proposition. For the first equality,\[\sigma_{\lambda/\mu} = \langle\mu+\rho|e^{H_+}|\lambda+\rho\rangle = \langle 0|\psi_{\mu_\ell-1/2}\cdots \psi_{\mu_1+\ell-3/2} \; e^{H_+} \; \psi^*_{\lambda_1+\ell-3/2}\cdots \psi^*_{\lambda_\ell-1/2}|0\rangle,\] and by Wick's Theorem this equals \[\det_{1\le i,j\le \ell} \langle 0|\psi_{\mu_j+\ell-j-1/2} \; e^{H_+} \; \psi^*_{\lambda_i+\ell-i-1/2}|0\rangle = \det_{1\le i,j\le \ell} h_{\lambda_i-\mu_j-i+j},\] by (\ref{h_k-as-Hamiltonian}).
\end{proof}

Paired with results about transformations of lattice models, this can be seen as an analogue of the Lindstr\"{o}m-Gessel-Viennot lemma. See Section \ref{free-fermionic-partition-function-section}.

Now we come to our main result of this section, showing that the $\sigma_{\lambda/\mu}$ obey an involutive identity. The proof is now easy.

\begin{proposition}
\[\omega(\sigma_{\lambda/\mu}) = \sigma_{\lambda'/\mu'},\] and thus \[\sigma_{\lambda'/\mu'} = \langle\lambda+\rho|e^{H_-}|\mu+\rho\rangle.\]
\end{proposition}

\begin{proof}
Apply the involution $\omega$ to the previous identities (\ref{Hamiltonian-Jacobi-Trudi}): \[\omega(\sigma_{\lambda/\mu}) = \omega\left(\det_{1\le i,j\le \ell} e_{\lambda'_i-\mu'_j-i+j}\right) = \det_{1\le i,j\le \ell} h_{\lambda'_i-\mu'_j-i+j} = \sigma_{\lambda'/\mu'}.\] The second equation follows from the first equation and (\ref{Hamiltonian-pos-neg-duality}).
\end{proof}

This can alternatively be shown by a particle-hole duality (see \cite{ZinnJustin-six-vertex}).

There are two specializations of the parameters $s_i^{(j)}$ that we care about in particular. If we set \[s_{\pm k}^{(j)} = (\pm 1)^{k-1}\frac{1}{k} x_j^k, \hspace{20pt} k>0,\] we obtain the classical symmetric functions, and $\sigma_{\lambda/\mu}$ is a skew Schur function. If instead we set \[s_{\pm k}^{(j)} = (\pm 1)^{k-1}\frac{1}{k} (x_j^k + (-1)^{k-1} y_j^k), \hspace{20pt} k>0,\] we obtain supersymmetric functions, and $\sigma_{\lambda/\mu}$ is a skew supersymmetric Schur function.

\subsection{Cauchy, Pieri, and branching rules for Hamiltonians} \label{Hamiltonian-identities-section}

For this section, we will use arbitrary sets of parameters. Let $\boldsymbol{s}_+ := \{s_k^{(j)}, 1\le j\le N, k\ge 1\}$ and $\boldsymbol{t}_- := \{t_{-k}^{(j)}, 1\le j\le N, k\ge 1\}$ be two half-infinite sets of parameters, and let $H_+ = H_+(\boldsymbol{s}_+)$, $H_- = H_-(\boldsymbol{t}_-)$. Let $\sigma'_{\lambda/\mu} := \langle\lambda+\rho|e^{H_-}|\mu+\rho\rangle$.

\begin{proposition}[Cauchy identity] \label{Hamiltonian-Cauchy-identity} For any strict partitions $\lambda$ and $\mu$, \begin{align*}\sum_\nu \sigma_{\lambda/\nu}\sigma'_{\mu/\nu} = \prod_{i,j}\exp\left(\sum_{k\ge 1} k\cdot s_k^{(i)}t_{-k}^{(j)}\right) \cdot \sum_\nu \sigma_{\nu/\mu}\sigma'_{\nu/\lambda},\end{align*} where the sums are over all strict partitions $\nu$.
\end{proposition}

\begin{proof}
We evaluate the Hamiltonian $\langle\mu+\rho|e^{H_-}e^{H_+}|\lambda+\rho\rangle$ in two ways. First, \[\langle\mu+\rho|e^{H_-}e^{H_+}|\lambda+\rho\rangle = \sum_\nu \langle\mu+\rho|e^{H_-}|\nu+\rho\rangle \langle \nu+\rho|e^{H_+}|\lambda+\rho\rangle = \sum_\nu \sigma_{\lambda/\nu}\sigma'_{\mu/\nu}.\] Next, we apply the commutation relations between $H_+$ and $H'_-$. \begin{align*}\langle\mu+\rho|e^{H'_-}e^{H_+}|\lambda+\rho\rangle &= \exp\left(\sum_{k\ge 1} k\cdot s_k t_{-k}\right) \cdot \langle\mu+\rho|e^{H_+}e^{H_-}|\lambda+\rho\rangle \\&= \prod_{i,j}\exp\left(\sum_{k\ge 1} k\cdot s_k^{(i)}t_{-k}^{(j)}\right)\cdot \sum_\nu \langle\mu+\rho|e^{H_+}|\nu+\rho\rangle \langle \nu+\rho|e^{H_-}|\lambda+\rho\rangle \\&= \prod_{i,j}\exp\left(\sum_{k\ge 1} k\cdot s_k^{(i)}s_{-k}^{(j)}\right)\cdot\sum_\nu \sigma_{\nu/\mu}\sigma'_{\nu/\lambda}.\end{align*}
\end{proof}

Given a variable set $\boldsymbol{s} = \{s_1^{(1)},\ldots, s_1^{(N)}, s_2^{(1)},\ldots s_2^{(N)},\ldots\}$, and some subset $I$ of $[N]:=\{1,\ldots,N\}$ let $\boldsymbol{s}|_I$ denote the subset $\bigcup_{i\in I}\{s_1^{(i)},s_2^{(i)},\ldots\}$. For example, $\boldsymbol{s}|_{[2,N]} = \boldsymbol{s} \setminus \{s_k^{(1)}\}_{k\ge 1}$.

\begin{proposition}[Branching rule] \label{Hamiltonian-branching-rule}
For all partitions $\lambda,\mu$, \[\sigma_{\lambda/\mu}(\boldsymbol{s}) = \sum_{\nu} \sigma_{\lambda/\nu}(\boldsymbol{s}|_{\{1\}}) \sigma_{\nu/\mu}(\boldsymbol{s}|_{[2,n]}) .\]
\end{proposition}

\begin{proof}
We have \begin{align*}\sigma_{\lambda/\mu} &= \langle\mu+\rho|e^{H_+}|\lambda+\rho\rangle \\&= \sum_{\nu} \langle\mu+\rho|e^{\phi_n}\cdots e^{\phi_2}|\nu+\rho\rangle \langle\nu+\rho|e^{\phi_1}|\lambda+\rho\rangle\\&= \sum_{\nu} \sigma_{\lambda/\nu}(\boldsymbol{s}|_{\{1\}}) \sigma_{\nu/\mu}(\boldsymbol{s}|_{[2,n]}).\end{align*}
\end{proof}

Let $J_{\mu} := J_{\mu_1}\ldots J_{\mu_{\ell(\mu)}}$, and $J_{-\mu} := J_{-\mu_1}\ldots J_{-\mu_{\ell(\mu)}}$. Let \[D_k = \sum_{\mu\vdash k} z_\mu^{-1} J_\mu, \hspace{20pt} U_k = \sum_{\mu\vdash k} z_\mu^{-1} J_{-\mu}.\]

\begin{proposition}[Pieri rule] \label{Hamiltonian-Pieri-rule}
\[h_k \cdot \sigma_\lambda = \sum_\nu \left\langle\nu+\rho\left|U_k\right|\lambda+\rho\right\rangle \sigma_\nu.\]
\end{proposition}

\begin{proof} First note that \[U_k\cdot 1 = \sum_{\mu\vdash k} z_\mu^{-1} J_{-\mu} \cdot 1 = \sum_{\mu\vdash k} z_\mu^{-1} p_\mu = h_k.\] Apply the boson-fermion correspondence (Proposition \ref{boson-fermion-correspondence}) to obtain
\begin{align*} h_k\cdot \sigma_\lambda &= h_k\cdot \langle \rho|e^{H_+}|\lambda+\rho\rangle \\&= \langle \rho| e^{H_+} U_k |\lambda+\rho\rangle \\&= \sum_\nu \langle \rho| e^{H_+} |\nu+\rho\rangle \langle \nu+\rho | U_k |\lambda+\rho\rangle \\&= \sum_\nu \langle \nu+\rho | U_k |\lambda+\rho\rangle \sigma_\nu.\end{align*}
\end{proof}

\section{The six-vertex model} \label{six-vertex-section}

In this section, we will define two related six-vertex models, called $\mathfrak{S}$ and $\mathfrak{S}^*$. Both of them parametrize the space of free fermionic six-vertex models i.e. their Boltzmann weights satisfy the condition \[\vx{a_1^{(i)}}\vx{a_2^{(i)}} + \vx{b_1^{(i)}}\vx{b_2^{(i)}} = \vx{c_1^{(i)}} \vx{c_2^{(i)}}, \hspace{20pt} \text{for all $i$}.\] The Boltzmann weights of these models are a simultaneous generalization of the weights of two pairs of six-vertex models. The first pair are the $\Gamma$ and $\Delta$ models in \cite{BBF-Schur-polynomials}. The second are the \emph{vicious} and \emph{osculating} models that appear in \cite{Korff-vicious-osculating}. There is a duality between the vicious and osculating models that we will generalize. Furthermore, this duality is equivalent to the duality for Hamiltonian operators proven in Lemma \ref{Hamiltonian-duality-lemma}.

Our lattice models are finite rectangular grids of intersecting lines, with vertices at the intersection points of each pair of lines. Each edge is assigned a \emph{spin} from a fixed set. An \emph{admissible vertex} is a vertex around which the spins satisfy one of several admissible configurations. If every vertex in the grid is admissible, we call the resulting configuration an \emph{admissible state}.

Each admissible vertex is assigned a \emph{Boltzmann weight}. These weights often depend a row parameter. In this paper, we will be slightly more general: weights will depend on several row parameters that parametrize certain sets of weights. Our weights will not depend on any column parameters since Hamiltonians do not behave well with respect to column parameters.

By convention, a non-admissible vertex has weight 0. The Boltzmann weight of a state is defined to be the product of all the Boltzmann weights of its constituent vertices, so non-admissible states always have weight 0. The \emph{partition function} $Z$ is the sum of the weights of all admissible states: \[Z = \sum_{\text{state } \mathfrak{s}} \wt(\mathfrak{s}) =  \sum_{\text{state } \mathfrak{s}} \prod_{\text{vertex } v} \wt(v).\]

The partition function $Z$ is a function of its row parameters. A surprising number of special functions occur as partition functions of a lattice model.

For the six-vertex model, our spin set is $\{+,-\}$. We imagine a $-$ spin to indicate the presence of a particle, and a $+$ spin to the indicate lack of a particle. The admissible vertices for $\mathfrak{S}$ and $\mathfrak{S}^*$ are slightly different, but are chosen in a way so that we can draw paths through the $-$ spins that start and end at the boundary. For $\mathfrak{S}$, these paths move up and left, and for $\mathfrak{S}^*$ they move down and left.

Let $x_i,y_i,z_i,w_i,A_i,B_i$ be arbitrary parameters, depending on a row $i$. Then the $\mathfrak{S}$ vertices and Boltzmann weights are given in Figure \ref{Delta-vertex-weights}, and the $\mathfrak{S}^*$ vertices and weights are given in Figure \ref{Gamma-vertex-weights}. We will sometimes refer to the vertex weights using the symbols for the vertices themselves--for instance, writing $\vx{a_1^{(i)}} = A_i$ for the $\mathfrak{S}$ weights below. When there is potential for confusion, we will make clear whether we are talking about the vertex itself or its weight, and which set of weights we are using.

\begin{figure}[h]
\centering
\scalebox{.85}{$
\begin{array}{c@{\hspace{8pt}}c@{\hspace{8pt}}c@{\hspace{5pt}}c@{\hspace{5pt}}c@{\hspace{8pt}}c@{\hspace{8pt}}c@{\hspace{5pt}}c}
\toprule
\vx{a_1^{(i)}} & \vx{a_2^{(i)}} & \vx{b_1^{(i)}} & \vx{b_2^{(i)}} & \vx{c_1^{(i)}} & \vx{c_2^{(i)}}\\
\midrule
\begin{tikzpicture}
\coordinate (a) at (-.75, 0);
\coordinate (b) at (0, .75);
\coordinate (c) at (.75, 0);
\coordinate (d) at (0, -.75);
\coordinate (aa) at (-.75,.5);
\coordinate (cc) at (.75,.5);
\draw (a)--(0,0);
\draw (b)--(0,0);
\draw (c)--(0,0);
\draw (d)--(0,0);
\draw[fill=white] (a) circle (.25);
\draw[fill=white] (b) circle (.25);
\draw[fill=white] (c) circle (.25);
\draw[fill=white] (d) circle (.25);
\node at (0,1) { };
\node at (a) {$+$};
\node at (b) {$+$};
\node at (c) {$+$};
\node at (d) {$+$};
\end{tikzpicture}
&
\begin{tikzpicture}
\coordinate (a) at (-.75, 0);
\coordinate (b) at (0, .75);
\coordinate (c) at (.75, 0);
\coordinate (d) at (0, -.75);
\coordinate (aa) at (-.75,.5);
\coordinate (cc) at (.75,.5);
\draw (a)--(0,0);
\draw (b)--(0,0);
\draw (c)--(0,0);
\draw (d)--(0,0);
\draw[fill=white] (a) circle (.25);
\draw[fill=white] (b) circle (.25);
\draw[fill=white] (c) circle (.25);
\draw[fill=white] (d) circle (.25);
\node at (0,1) { };
\node at (a) {$-$};
\node at (b) {$-$};
\node at (c) {$-$};
\node at (d) {$-$};
\end{tikzpicture}
&
\begin{tikzpicture}
\coordinate (a) at (-.75, 0);
\coordinate (b) at (0, .75);
\coordinate (c) at (.75, 0);
\coordinate (d) at (0, -.75);
\coordinate (aa) at (-.75,.5);
\coordinate (cc) at (.75,.5);
\draw (a)--(0,0);
\draw (b)--(0,0);
\draw (c)--(0,0);
\draw (d)--(0,0);
\draw[fill=white] (a) circle (.25);
\draw[fill=white] (b) circle (.25);
\draw[fill=white] (c) circle (.25);
\draw[fill=white] (d) circle (.25);
\node at (0,1) { };
\node at (a) {$+$};
\node at (b) {$-$};
\node at (c) {$+$};
\node at (d) {$-$};
\end{tikzpicture}
&
\begin{tikzpicture}
\coordinate (a) at (-.75, 0);
\coordinate (b) at (0, .75);
\coordinate (c) at (.75, 0);
\coordinate (d) at (0, -.75);
\coordinate (aa) at (-.75,.5);
\coordinate (cc) at (.75,.5);
\draw (a)--(0,0);
\draw (b)--(0,0);
\draw (c)--(0,0);
\draw (d)--(0,0);
\draw[fill=white] (a) circle (.25);
\draw[fill=white] (b) circle (.25);
\draw[fill=white] (c) circle (.25);
\draw[fill=white] (d) circle (.25);
\node at (0,1) { };
\node at (a) {$-$};
\node at (b) {$+$};
\node at (c) {$-$};
\node at (d) {$+$};
\end{tikzpicture}
&
\begin{tikzpicture}
\coordinate (a) at (-.75, 0);
\coordinate (b) at (0, .75);
\coordinate (c) at (.75, 0);
\coordinate (d) at (0, -.75);
\coordinate (aa) at (-.75,.5);
\coordinate (cc) at (.75,.5);
\draw (a)--(0,0);
\draw (b)--(0,0);
\draw (c)--(0,0);
\draw (d)--(0,0);
\draw[fill=white] (a) circle (.25);
\draw[fill=white] (b) circle (.25);
\draw[fill=white] (c) circle (.25);
\draw[fill=white] (d) circle (.25);
\node at (0,1) { };
\node at (a) {$-$};
\node at (b) {$+$};
\node at (c) {$+$};
\node at (d) {$-$};
\end{tikzpicture}
&
\begin{tikzpicture}
\coordinate (a) at (-.75, 0);
\coordinate (b) at (0, .75);
\coordinate (c) at (.75, 0);
\coordinate (d) at (0, -.75);
\coordinate (aa) at (-.75,.5);
\coordinate (cc) at (.75,.5);
\draw (a)--(0,0);
\draw (b)--(0,0);
\draw (c)--(0,0);
\draw (d)--(0,0);
\draw[fill=white] (a) circle (.25);
\draw[fill=white] (b) circle (.25);
\draw[fill=white] (c) circle (.25);
\draw[fill=white] (d) circle (.25);
\node at (0,1) { };
\node at (a) {$+$};
\node at (b) {$-$};
\node at (c) {$-$};
\node at (d) {$+$};
\end{tikzpicture}
\\
   \midrule
   A_i & y_iA_iB_i & A_iB_i & x_iA_i & (x_i+y_i)A_iB_i & A_i \\
   \bottomrule
\end{array}$}
\caption{The Boltzmann weights for $\mathfrak{S}$. Here, $x_i, y_i, A_i$, and $B_i$ are parameters associated to each row.}
    \label{Delta-vertex-weights}
\end{figure}

\begin{figure}[h]
\centering
\scalebox{.85}{$
\begin{array}{c@{\hspace{8pt}}c@{\hspace{8pt}}c@{\hspace{5pt}}c@{\hspace{5pt}}c@{\hspace{8pt}}c@{\hspace{8pt}}c@{\hspace{5pt}}c}
\toprule
\vx{a_1^{(i)}} & \vx{a_2^{(i)}} & \vx{b_1^{(i)}} & \vx{b_2^{(i)}} & \vx{c_1^{(i)}} & \vx{c_2^{(i)}}\\
\midrule
\begin{tikzpicture}
\coordinate (a) at (-.75, 0);
\coordinate (b) at (0, .75);
\coordinate (c) at (.75, 0);
\coordinate (d) at (0, -.75);
\coordinate (aa) at (-.75,.5);
\coordinate (cc) at (.75,.5);
\draw (a)--(0,0);
\draw (b)--(0,0);
\draw (c)--(0,0);
\draw (d)--(0,0);
\draw[fill=white] (a) circle (.25);
\draw[fill=white] (b) circle (.25);
\draw[fill=white] (c) circle (.25);
\draw[fill=white] (d) circle (.25);
\node at (0,1) { };
\node at (a) {$+$};
\node at (b) {$+$};
\node at (c) {$+$};
\node at (d) {$+$};
\end{tikzpicture}
&
\begin{tikzpicture}
\coordinate (a) at (-.75, 0);
\coordinate (b) at (0, .75);
\coordinate (c) at (.75, 0);
\coordinate (d) at (0, -.75);
\coordinate (aa) at (-.75,.5);
\coordinate (cc) at (.75,.5);
\draw (a)--(0,0);
\draw (b)--(0,0);
\draw (c)--(0,0);
\draw (d)--(0,0);
\draw[fill=white] (a) circle (.25);
\draw[fill=white] (b) circle (.25);
\draw[fill=white] (c) circle (.25);
\draw[fill=white] (d) circle (.25);
\node at (0,1) { };
\node at (a) {$-$};
\node at (b) {$-$};
\node at (c) {$-$};
\node at (d) {$-$};
\end{tikzpicture}
&
\begin{tikzpicture}
\coordinate (a) at (-.75, 0);
\coordinate (b) at (0, .75);
\coordinate (c) at (.75, 0);
\coordinate (d) at (0, -.75);
\coordinate (aa) at (-.75,.5);
\coordinate (cc) at (.75,.5);
\draw (a)--(0,0);
\draw (b)--(0,0);
\draw (c)--(0,0);
\draw (d)--(0,0);
\draw[fill=white] (a) circle (.25);
\draw[fill=white] (b) circle (.25);
\draw[fill=white] (c) circle (.25);
\draw[fill=white] (d) circle (.25);
\node at (0,1) { };
\node at (a) {$+$};
\node at (b) {$-$};
\node at (c) {$+$};
\node at (d) {$-$};
\end{tikzpicture}
&
\begin{tikzpicture}
\coordinate (a) at (-.75, 0);
\coordinate (b) at (0, .75);
\coordinate (c) at (.75, 0);
\coordinate (d) at (0, -.75);
\coordinate (aa) at (-.75,.5);
\coordinate (cc) at (.75,.5);
\draw (a)--(0,0);
\draw (b)--(0,0);
\draw (c)--(0,0);
\draw (d)--(0,0);
\draw[fill=white] (a) circle (.25);
\draw[fill=white] (b) circle (.25);
\draw[fill=white] (c) circle (.25);
\draw[fill=white] (d) circle (.25);
\node at (0,1) { };
\node at (a) {$-$};
\node at (b) {$+$};
\node at (c) {$-$};
\node at (d) {$+$};
\end{tikzpicture}
&
\begin{tikzpicture}
\coordinate (a) at (-.75, 0);
\coordinate (b) at (0, .75);
\coordinate (c) at (.75, 0);
\coordinate (d) at (0, -.75);
\coordinate (aa) at (-.75,.5);
\coordinate (cc) at (.75,.5);
\draw (a)--(0,0);
\draw (b)--(0,0);
\draw (c)--(0,0);
\draw (d)--(0,0);
\draw[fill=white] (a) circle (.25);
\draw[fill=white] (b) circle (.25);
\draw[fill=white] (c) circle (.25);
\draw[fill=white] (d) circle (.25);
\node at (0,1) { };
\node at (a) {$+$};
\node at (b) {$+$};
\node at (c) {$-$};
\node at (d) {$-$};
\end{tikzpicture}
&
\begin{tikzpicture}
\coordinate (a) at (-.75, 0);
\coordinate (b) at (0, .75);
\coordinate (c) at (.75, 0);
\coordinate (d) at (0, -.75);
\coordinate (aa) at (-.75,.5);
\coordinate (cc) at (.75,.5);
\draw (a)--(0,0);
\draw (b)--(0,0);
\draw (c)--(0,0);
\draw (d)--(0,0);
\draw[fill=white] (a) circle (.25);
\draw[fill=white] (b) circle (.25);
\draw[fill=white] (c) circle (.25);
\draw[fill=white] (d) circle (.25);
\node at (0,1) { };
\node at (a) {$-$};
\node at (b) {$-$};
\node at (c) {$+$};
\node at (d) {$+$};
\end{tikzpicture}
\\
   \midrule
   A_i^{-1} & x_iA_i^{-1}B_i^{-1} & A_i^{-1}B_i^{-1} & y_iA_i^{-1} & (x_i+y_i)A_i^{-1} & A_i^{-1}B_i^{-1} \\
   \bottomrule
\end{array}$}
\caption{The Boltzmann weights for $\mathfrak{S}^*$. Here, $x_i, y_i, A_i$, and $B_i$ are parameters associated to each row.}
    \label{Gamma-vertex-weights}
\end{figure}

The $\mathfrak{S}$ weights are given in Figure \ref{Delta-vertex-weights}, and the $\mathfrak{S}^*$ weights are given in Figure \ref{Gamma-vertex-weights}. Both sets of Boltzmann weights parametrize the free fermionic weights (with a small caveat; see Remark \ref{c-weights-remark})).

First, notice that for any values of $x_i,y_i,A_i,B_i$, we have \[\vx{a_1^{(i)}}\vx{a_2^{(i)}} + \vx{b_1^{(i)}}\vx{b_2^{(i)}} - \vx{c_1^{(i)}}\vx{c_2^{(i)}} = y_iA_i^2B_i + x_iA_i^2B_i - (x_i+y_i)A_i^2B_i = 0,\] so the $\mathfrak{S}$ weights satisfy the free fermionic condition. Conversely, given a set of free fermionic weights, we can set \[A_i:= \vx{a_1^{(i)}}, \hspace{10pt} B_i = \frac{\vx{b_1^{(i)}}}{\vx{a_1^{(i)}}}, \hspace{10pt} x_i := \frac{\vx{b_2^{(i)}}}{\vx{a_1^{(i)}}}, \hspace{10pt} y_i := \frac{\vx{a_2^{(i)}}}{\vx{b_1^{(i)}}}.\] Then the free fermionic condition ensures that $\vx{c_1^{(i)}}\vx{c_2^{(i)}} = (x_i+y_i)A_i^2B_i$, as in Figure \ref{Delta-vertex-weights}. The $\mathfrak{S}^*$ weights are similar.

Next, we define the boundary conditions for the two models. Let the lattice model \[\mathfrak{S}_{\lambda/\mu} := \mathfrak{S}_{\lambda/\mu}(\boldsymbol{x}, \boldsymbol{y}, \boldsymbol{A}, \boldsymbol{B}) = \mathfrak{S}_{\lambda/\mu}(x_1, \ldots x_N; y_1, \ldots y_N; A_1, \ldots A_N; B_1, \ldots B_N)\] be defined as follows.

\begin{itemize}
    \item $N$ rows, labelled $1,\ldots, N$ from bottom to top;
    \item $M+1$ columns, where $M\ge \max(\lambda_1,\mu_1)$, labelled $0,\ldots,M$ from left to right;
    \item Left and right boundary edges all $+$;
    \item Bottom boundary edges $-$ on parts of $\lambda$; $+$ otherwise;
    \item Top boundary edges $-$ on parts of $\mu$; $+$ otherwise.
    \item Boltzmann weights from Figure \ref{Delta-vertex-weights} (in row $i$, we assign the weights $\vx{a_1^{(i)}}, \vx{a_2^{(i)}}, \vx{b_1^{(i)}}, \vx{b_2^{(i)}}$, $\vx{c_1^{(i)}}, \vx{c_2^{(i)}}$ from that figure).
\end{itemize}

See Figure \ref{Delta-six-vertex-figure} for an example of these boundary conditions. Despite the model's dependence on $M$ and $N$, we suppress them from our notation. Note that the top and bottom boundaries are arbitrary, while the side boundary conditions are all $+$. These ``empty'' side boundary conditions turn out to be most directly related to Hamiltonian operators. We will consider other side boundary conditions, including domain-wall, in Section \ref{boundary-conditions-section}.

\begin{figure}[h]
\begin{center}
\scalebox{0.8}{
\begin{tikzpicture}
  \coordinate (ab) at (1,0);
  \coordinate (ad) at (3,0);
  \coordinate (af) at (5,0);
  \coordinate (ah) at (7,0);
  \coordinate (aj) at (9,0);
  \coordinate (al) at (11,0);
  \coordinate (ba) at (0,1);
  \coordinate (bc) at (2,1);
  \coordinate (be) at (4,1);
  \coordinate (bg) at (6,1);
  \coordinate (bi) at (8,1);
  \coordinate (bk) at (10,1);
  \coordinate (bm) at (12,1);
  \coordinate (cb) at (1,2);
  \coordinate (cd) at (3,2);
  \coordinate (cf) at (5,2);
  \coordinate (ch) at (7,2);
  \coordinate (cj) at (9,2);
  \coordinate (cl) at (11,2);
  \coordinate (da) at (0,3);
  \coordinate (dc) at (2,3);
  \coordinate (de) at (4,3);
  \coordinate (dg) at (6,3);
  \coordinate (di) at (8,3);
  \coordinate (dk) at (10,3);
  \coordinate (dm) at (12,3);
  \coordinate (eb) at (1,4);
  \coordinate (ed) at (3,4);
  \coordinate (ef) at (5,4);
  \coordinate (eh) at (7,4);
  \coordinate (ej) at (9,4);
  \coordinate (el) at (11,4);
  \coordinate (fa) at (0,5);
  \coordinate (fc) at (2,5);
  \coordinate (fe) at (4,5);
  \coordinate (fg) at (6,5);
  \coordinate (fi) at (8,5);
  \coordinate (fk) at (10,5);
  \coordinate (fm) at (12,5);
  \coordinate (gb) at (1,6);
  \coordinate (gd) at (3,6);
  \coordinate (gf) at (5,6);
  \coordinate (gh) at (7,6);
  \coordinate (gj) at (9,6);
  \coordinate (gl) at (11,6);
  \draw (ab)--(gb);
  \draw (ad)--(gd);
  \draw (af)--(gf);
  \draw (ah)--(gh);
  \draw (aj)--(gj);
  \draw (al)--(gl);
  \draw (ba)--(bm);
  \draw (da)--(dm);
  \draw (fa)--(fm);
  \draw[fill=white] (ab) circle (.25);
  \draw[fill=white] (ad) circle (.25);
  \draw[fill=white] (af) circle (.25);
  \draw[fill=white] (ah) circle (.25);
  \draw[fill=white] (aj) circle (.25);
  \draw[fill=white] (al) circle (.25);
  \draw[fill=white] (ba) circle (.25);
  \draw[fill=white] (bc) circle (.25);
  \draw[fill=white] (be) circle (.25);
  \draw[fill=white] (bg) circle (.25);
  \draw[fill=white] (bi) circle (.25);
  \draw[fill=white] (bk) circle (.25);
  \draw[fill=white] (bm) circle (.25);
  \draw[fill=white] (cb) circle (.25);
  \draw[fill=white] (cd) circle (.25);
  \draw[fill=white] (cf) circle (.25);
  \draw[fill=white] (ch) circle (.25);
  \draw[fill=white] (cj) circle (.25);
  \draw[fill=white] (cl) circle (.25);
  \draw[fill=white] (da) circle (.25);
  \draw[fill=white] (dc) circle (.25);
  \draw[fill=white] (de) circle (.25);
  \draw[fill=white] (dg) circle (.25);
  \draw[fill=white] (di) circle (.25);
  \draw[fill=white] (dk) circle (.25);
  \draw[fill=white] (dm) circle (.25);
  \draw[fill=white] (eb) circle (.25);
  \draw[fill=white] (ed) circle (.25);
  \draw[fill=white] (ef) circle (.25);
  \draw[fill=white] (eh) circle (.25);
  \draw[fill=white] (ej) circle (.25);
  \draw[fill=white] (el) circle (.25);
  \draw[fill=white] (fa) circle (.25);
  \draw[fill=white] (fc) circle (.25);
  \draw[fill=white] (fe) circle (.25);
  \draw[fill=white] (fg) circle (.25);
  \draw[fill=white] (fi) circle (.25);
  \draw[fill=white] (fk) circle (.25);
  \draw[fill=white] (fm) circle (.25);
  \draw[fill=white] (gb) circle (.25);
  \draw[fill=white] (gd) circle (.25);
  \draw[fill=white] (gf) circle (.25);
  \draw[fill=white] (gh) circle (.25);
  \draw[fill=white] (gj) circle (.25);
  \draw[fill=white] (gl) circle (.25);
  \node at (-2,5) {row:};
  \node at (-1.2,5) {3};
  \node at (-1.2,3) {2};
  \node at (-1.2,1) {1};
  \node at (gb) {$-$};
  \node at (gd) {$-$};
  \node at (gf) {$+$};
  \node at (gh) {$-$};
  \node at (gj) {$+$};
  \node at (gl) {$+$};
  \node at (fa) {$+$};
  \node at (fc) {$+$};
  \node at (fe) {$+$};
  \node at (fg) {$+$};
  \node at (fi) {$+$};
  \node at (fk) {$+$};
  \node at (fm) {$+$};
  \node at (eb) {$-$};
  \node at (ed) {$-$};
  \node at (ef) {$+$};
  \node at (eh) {$-$};
  \node at (ej) {$+$};
  \node at (el) {$+$};
  \node at (da) {$+$};
  \node at (dc) {$-$};
  \node at (de) {$-$};
  \node at (dg) {$+$};
  \node at (di) {$-$};
  \node at (dk) {$-$};
  \node at (dm) {$+$};
  \node at (cb) {$+$};
  \node at (cd) {$-$};
  \node at (cf) {$-$};
  \node at (ch) {$+$};
  \node at (cj) {$+$};
  \node at (cl) {$-$};
  \node at (ba) {$+$};
  \node at (bc) {$+$};
  \node at (be) {$+$};
  \node at (bg) {$-$};
  \node at (bi) {$+$};
  \node at (bk) {$+$};
  \node at (bm) {$+$};
  \node at (ab) {$+$};
  \node at (ad) {$-$};
  \node at (af) {$+$};
  \node at (ah) {$-$};
  \node at (aj) {$+$};
  \node at (al) {$-$};
  \node at (0,7) {column:};
  \node at (1,7) {$0$};
  \node at (3,7) {$1$};
  \node at (5,7) {$2$};
  \node at (7,7) {$3$};
  \node at (9,7) {$4$};
  \node at (11,7) {$5$};
\end{tikzpicture}}
\end{center}
\caption{A state of the lattice model $\mathfrak{S}_{\lambda/\mu}$, where $\lambda = (5,3,1)$ and $\mu = (3,1,0)$.}
\label{Delta-six-vertex-figure}
\end{figure}
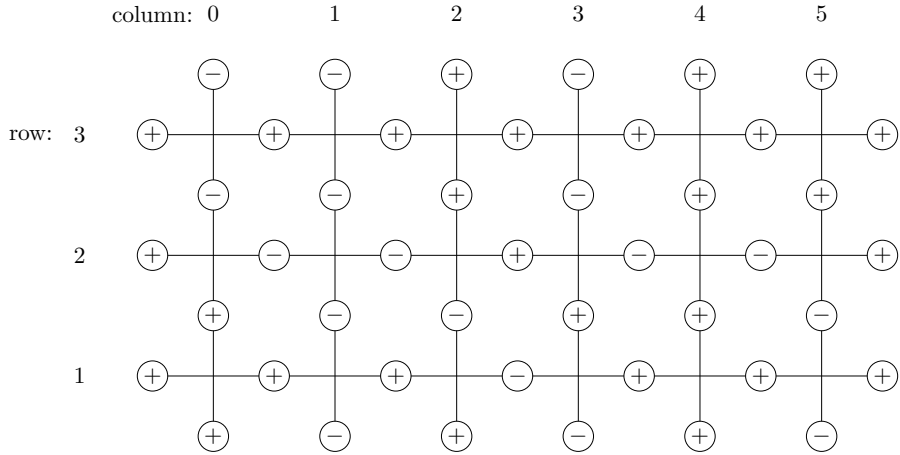

Additionally, we will define $\ba{\mathfrak{S}}_{\lambda/\mu}$ to be the same model as $\mathfrak{S}_{\lambda/\mu}$, but with arbitrary Boltzmann weights. We will use this more general model when we don't want to assume free fermionic weights.

It is often useful to consider the ensemble of lattice models $\{\mathfrak{S}_{\lambda/\mu}\}$ for all (valid) $\lambda,\mu,M,N$ at once. When we want to think of the model in this way, we will simply denote the model $\mathfrak{S}$.

The dual model $\mathfrak{S}^*$ is defined similarly: \[\mathfrak{S}^*_{\lambda/\mu} := \mathfrak{S}^*_{\lambda/\mu}(\boldsymbol{x}, \boldsymbol{y}, \boldsymbol{A}, \boldsymbol{B}) = \mathfrak{S}^*_{\lambda/\mu}(x_1, \ldots x_N; y_1, \ldots y_N; A_1, \ldots A_N; B_1, \ldots B_N)\]  is the following lattice model.

\begin{itemize}
    \item $N$ rows, labelled $1,\ldots, N$ from top to bottom;
    \item $M+1$ columns, where $M\ge \max(\lambda_1,\mu_1)$, labelled $0,\ldots,M$ from left to right;
    \item Left and right boundary edges all $+$;
    \item Bottom boundary edges $-$ on parts of $\mu$; $+$ otherwise;
    \item Top boundary edges $-$ on parts of $\lambda$; $+$ otherwise.
    \item Boltzmann weights from Figure \ref{Gamma-vertex-weights} (in row $i$, we assign the weights $\vx{a_1^{(i)}}, \vx{a_2^{(i)}},\vx{b_1^{(i)}},\vx{b_2^{(i)}}$, $\vx{c_1^{(i)}},\vx{c_2^{(i)}}$ from that figure).
\end{itemize}

\begin{remark} \label{c-weights-remark}
To fully parametrize sets of free fermionic weights, we would need a fifth parameter representing $\vx{c_2}/\vx{c_1}$. However, this turns out not to be necessary. With our boundary conditions, the number of $\vx{c_1^{(i)}}$ vertices and the number of $\vx{c_2^{(i)}}$ vertices in any state are equal. Thus, changing the relative weights of $\vx{c_1^{(i)}}$ and $\vx{c_2^{(i)}}$ without changing their product does not change the partition function.

For general side boundary conditions, the number of $\vx{c_1^{(i)}}$ vertices and the number of $\vx{c_2^{(i)}}$ vertices always differ by a constant that depends only on the boundary conditions. In this case, changing the relative weights of $\vx{c_1^{(i)}}$ and $\vx{c_2^{(i)}}$ multiplies the partition function by a factor which is easily computed.
\end{remark}

\subsection{Relating $\mathfrak{S}$ and $\mathfrak{S}^*$}

We will see in Corollary \ref{Partition-Function-is-Supersymmetric-Schur} that the partition function of both models is a supersymmetric Schur function. For now, note that there are two close relationships between the models. First, let us do the following transformation.

\begin{itemize}
    \item Rotate the model $\mathfrak{S}_{\lambda/\mu}(\boldsymbol{x}, \boldsymbol{y}, \boldsymbol{A}, \boldsymbol{B})$ $180^\circ$.
    \item Flip the vertical spins.
    \item Reverse the ordering on the columns.
    \item Divide the Boltzmann weights by $A_i^2B_i$.
\end{itemize}

This new model has rows labelled from top to bottom, and columns labelled from left to right. If $\lambda$ is the bottom boundary of the original model, then the top boundary of the new model is the partition $\ba{\lambda}$ obtained by reversing $\lambda$ and swapping its parts and non-parts. By Lemma \ref{partition-lemma}, $\ba{\lambda}-\rho = (\lambda-\rho)'$.

The $180^\circ$ rotation along with the flip of vertical spins sends each of the six $\mathfrak{S}$ vertices to its counterpart $\mathfrak{S}^*$ vertex, swapping types $\vx{a}$ and $\vx{b}$ vertices. Dividing by $A_i^2B_i$ sends the $\mathfrak{S}$ Boltzmann weights to the $\mathfrak{S}^*$ Boltzmann weights, again, with vertex types $\vx{a}$ and $\vx{b}$ swapped. Therefore, we have obtained the model $\mathfrak{S}^*_{\ba{\lambda}/\ba{\mu}}$.

Since each of the first three steps is a weight-preserving bijection of states, this allows us to relate the $\mathfrak{S}$ and $\mathfrak{S}^*$ partition functions.

\begin{proposition} \label{first-Gamma-Delta-relationship}
\begin{align*}Z(\mathfrak{S}^*_{\ba{\lambda}/\ba{\mu}}(\boldsymbol{x}, \boldsymbol{y}, \boldsymbol{A}, \boldsymbol{B})) = \prod_{i=1}^N (A_i^{-2M-2} B_i^{-M-1}) \cdot Z(\mathfrak{S}_{\lambda/\mu}(\boldsymbol{x}, \boldsymbol{y}, \boldsymbol{A}, \boldsymbol{B})).\end{align*}
\end{proposition}

For our second relationship, we do the following transformation to $\mathfrak{S}$:

\begin{itemize}
    \item Flip the model vertically (over a horizontal axis).
    \item Swap the $\vx{c_1}$ and $\vx{c_2}$ vertices.
    \item Replace $A_i$ with $A_i^{-1}$ and $B_i$ with $B_i^{-1}$.
    \item Swap $x_i$ and $y_i$.
    \item Rebalance the $\vx{c_1}$ and $\vx{c_2}$ vertices by multiplying the former and dividing the latter by $x_i+y_i$.
\end{itemize}

We have again obtained a $\mathfrak{S}^*$ lattice model, this time simply $\mathfrak{S}^*_{\lambda/\mu}$. Note that only the third and fourth steps change the partition function.

\begin{proposition} \label{second-Gamma-Delta-relationship}
\begin{align}Z(\mathfrak{S}^*_{\lambda/\mu}(\boldsymbol{x}, \boldsymbol{y}, \boldsymbol{A}, \boldsymbol{B})) = Z(\mathfrak{S}_{\lambda/\mu}(\boldsymbol{y}, \boldsymbol{x}, \boldsymbol{A}^{-1}, \boldsymbol{B}^{-1})). \label{Simple-Gamma-Delta-relationship}\end{align}
\end{proposition}

We obtain similar identities by doing similar transformations starting from $\mathfrak{S}^*$. We can now combine these identities to relate the partition functions of $\mathfrak{S}$ and $\mathfrak{S}^*$ to themselves.

\begin{proposition} \label{Functional-equations-proposition}
\leavevmode
\begin{enumerate}
    \item[(a)] \begin{align*} Z(\mathfrak{S}^*_{\ba{\lambda}/\ba{\mu}}(\boldsymbol{x}, \boldsymbol{y}, \boldsymbol{A}, \boldsymbol{B})) = \prod_{i=1}^N (A_i^{-2M-2} B_i^{-M-1}) \cdot Z(\mathfrak{S}^*_{\lambda/\mu}(\boldsymbol{y}, \boldsymbol{x}, \boldsymbol{A}^{-1}, \boldsymbol{B}^{-1})).\end{align*}
    \item[(b)] \begin{align*} Z(\mathfrak{S}_{\ba{\lambda}/\ba{\mu}}(\boldsymbol{x}, \boldsymbol{y}, \boldsymbol{A}, \boldsymbol{B})) = \prod_{i=1}^N (A_i^{2M+2} B_i^{M+1}) \cdot Z(\mathfrak{S}_{\lambda/\mu}(\boldsymbol{y}, \boldsymbol{x}, \boldsymbol{A}^{-1}, \boldsymbol{B}^{-1})).\end{align*}
\end{enumerate}
\end{proposition}

Note that similar identities hold for the general models $\mathfrak{S}$ and $\mathfrak{S}^*$.

The above propositions can be seen as giving us an involution on the partition functions. Let us make that more explicit. Define \[\tilde{\omega}\left(\mathfrak{S}^*_{\lambda/\mu}(\boldsymbol{x}, \boldsymbol{y}, \boldsymbol{A}, \boldsymbol{B})\right) := \mathfrak{S}_{\lambda/\mu}(\boldsymbol{x}, \boldsymbol{y}, \boldsymbol{A}, \boldsymbol{B}), \hspace{20pt} \tilde{\omega}\left(\mathfrak{S}_{\lambda/\mu}(\boldsymbol{x}, \boldsymbol{y}, \boldsymbol{A}, \boldsymbol{B})\right) := \mathfrak{S}^*_{\lambda/\mu}(\boldsymbol{x}, \boldsymbol{y}, \boldsymbol{A}, \boldsymbol{B}),\] and further define \[\tilde{\omega}\left(Z(\mathfrak{S}^*_{\lambda/\mu})\right) := Z\left(\tilde{\omega}(\mathfrak{S}^*_{\lambda/\mu})\right),\hspace{20pt} \tilde{\omega}\left(Z(\mathfrak{S}_{\lambda/\mu})\right) := Z\left(\tilde{\omega}(\mathfrak{S}_{\lambda/\mu})\right).\]

Therefore, by Propositions \ref{first-Gamma-Delta-relationship} and \ref{second-Gamma-Delta-relationship}, we have the following description of the action of $\omega$ on our partition functions:

\begin{corollary} \label{involution-action-partition-function}
For all strict partitions $\lambda,\mu$, \[ \tilde{\omega}\left(Z(\mathfrak{S}^*_{\lambda/\mu}(\boldsymbol{x}, \boldsymbol{y}, \boldsymbol{A}, \boldsymbol{B})))\right) = \prod_{i=1}^N A_i^{2M+2} B_i^{M+1} Z(\mathfrak{S}^*_{\ba{\lambda}/\ba{\mu}}(\boldsymbol{x}, \boldsymbol{y}, \boldsymbol{A}, \boldsymbol{B})) = Z(\mathfrak{S}^*_{\lambda/\mu}(\boldsymbol{y}, \boldsymbol{x}, \boldsymbol{A}^{-1}, \boldsymbol{B}^{-1}))\] and \[ \tilde{\omega}\left(Z(\mathfrak{S}_{\lambda/\mu}(\boldsymbol{x}, \boldsymbol{y}, \boldsymbol{A}, \boldsymbol{B}))\right) = \prod_{i=1}^N A_i^{-2M-2} B_i^{-M-1} Z(\mathfrak{S}_{\ba{\lambda}/\ba{\mu}}(\boldsymbol{x}, \boldsymbol{y}, \boldsymbol{A}, \boldsymbol{B})) = Z(\mathfrak{S}_{\ba{\lambda}/\ba{\mu}}(\boldsymbol{y}, \boldsymbol{x}, \boldsymbol{A}^{-1}, \boldsymbol{B}^{-1})).\]
\end{corollary}

We will see in Section \ref{free-fermionic-partition-function-section} that when we have a Hamiltonian matching a lattice model, the action of $\omega$ is the same as the action of $\tilde{\omega}$.

We conclude this section with a discussion of two specializations. \emph{Vicious walkers} are sets of non-intersecting lattice paths that were introduced by Fisher \cite{Fisher-vicious-walks}. These walkers can be interpreted as states of a solvable lattice model, and there is also a dual model made up of \emph{osculating walkers} \cite{Brak-osculating}. Through tableaux combinatorics or the Yang-Baxter equation, one can show that the partition functions of these models are Schur polynomials. The Lindstr\"{o}m-Gessel-Viennot (LGV) Lemma \cite{Lindstrom-LGV, GesselViennot-LGV} expresses the number of sets of non-intersecting paths on a graph as a determinantal formula, and when applied to vicious walkers, the LGV Lemma gives the Jacobi-Trudi formula for Schur polynomials. In addition, the Von N\"{a}agelsbach–Kostka formula (dual Jacobi-Trudi) follows from a similar argument by particle-hole duality (see \cite{ZinnJustin-six-vertex}).

The specialization of the model $\mathfrak{S}_{\lambda/\mu}(\boldsymbol{x}, \boldsymbol{0}, \boldsymbol{1}, \boldsymbol{1}))$ gives the vicious model from \cite{Korff-vicious-osculating} after reversing the row indices, while $\mathfrak{S}^*_{\lambda/\mu}(\boldsymbol{x}, \boldsymbol{0}, \boldsymbol{1}, \boldsymbol{1}))$ gives the osculating model, after a vertical flip and a rebalancing of the $\vx{c_1}$ and $\vx{c_2}$ weights. On the other hand, $\mathfrak{S}_{\lambda/\mu}(\boldsymbol{0}, \boldsymbol{x}, \boldsymbol{1}, \boldsymbol{1}))$ gives the osculating model, while $\mathfrak{S}^*_{\lambda/\mu}(\boldsymbol{0}, \boldsymbol{x}, \boldsymbol{1}, \boldsymbol{1}))$ gives the vicious model. By Corollary \ref{involution-action-partition-function}, $\tilde{\omega}$ interchanges these models.

A second set of specializations obtains the $\Delta$ and $\Gamma$ models from \cite{Brubaker-Schultz}. In Figure 2 of that paper, interpret left and up arrows as $-$ and interpret right and down arrows as $+$. Then $\mathfrak{S}_{\lambda/\mu}(\boldsymbol{x}, \boldsymbol{xt}, \boldsymbol{1}, \boldsymbol{1}))$ is the $\Delta$ model from that paper. On the other hand, if we flip $\mathfrak{S}^*$ vertically and rebalance the $\vx{c_1}$ and $\vx{c_2}$ vertices, then the $\Gamma$ model in \cite{Brubaker-Schultz} matches $\mathfrak{S}^*_{\lambda/\mu}(\boldsymbol{x}, \boldsymbol{x/t}, \boldsymbol{1}, \boldsymbol{t}))$.

\section{Six-vertex models and Hamiltonians} \label{proof-of-classical-main-result-section}

The purpose of this section is to prove formulas for the partition functions of $\mathfrak{S}$ and $\mathfrak{S}^*$ in terms of a Hamiltonian operator.

Let $T$ (resp. $T^*$) be the \emph{row transfer matrix} for $\mathfrak{S}$ (resp. $\mathfrak{S}^*$): \[\langle\mu|T|\lambda\rangle := Z(\mathfrak{S}_{\lambda/\mu}), \hspace{20pt} \langle\lambda|T^*|\mu\rangle := Z(\mathfrak{S}^*_{\lambda/\mu}),\] where for both lattice models, $N=1$ and $M\ge \max(\lambda_1,\mu_1)$.

Another way to interpret the main result of this section (Theorem \ref{Delta-lattice-Hamiltonian}) is that the action of the operator $e^\phi$ equals (up the a simple factor) the action of $T$.

\subsection{The $\mathfrak{S}$ lattice model}

We will say that the lattice model $\mathfrak{S}$ and the Hamiltonian operator $e^{H_+}$ \emph{match} if the following condition holds: \begin{align}Z(\mathfrak{S}_{\lambda/\mu}) = \prod_{i=1}^{N} A_i^{M+1} B_i^{\ell(\lambda)} \cdot\langle\mu|e^{H_+}|\lambda\rangle \hspace{20pt} \text{for all strict partitions $\lambda,\mu$ and all $M,N$.} \label{Delta-Lattice-Hamiltonian-correspondence-equation}\end{align} More generally, we will say that a lattice model $\mathfrak{S}$ and a Hamiltonian operator $e^H$ match if \begin{align}Z(\mathfrak{S}_{\lambda/\mu}) =  {\Large{*}} \cdot\langle\mu|e^H|\lambda\rangle \hspace{20pt} \text{for all strict partitions $\lambda,\mu$ and all $M,N$,} \label{General-Lattice-Hamiltonian-correspondence-equation}\end{align} where ${\Large{*}}$ represents any easily computable function of the Boltzmann weights of $\mathfrak{S}$.

In a sense, this condition tells us that the Hamiltonian operator $e^{H_+}$ has the \emph{same} structure as the lattice model $\mathfrak{S}$. Although the procedures for calculating the partition function and $\tau$ function are different, the result is always the same.

\begin{theorem} \label{Delta-lattice-Hamiltonian}
\leavevmode
\begin{enumerate}
    \item[(a)] (\ref{Delta-Lattice-Hamiltonian-correspondence-equation}) holds precisely when \begin{align} s_k^{(j)} = \frac{1}{k}\left(x_i^k + (-1)^{k-1}y_i^k\right) \hspace{20pt} \text{for all } k\ge 1, j\in [1,N]. \label{Delta-Hamiltonian-parameter}\end{align}
    \item[(b)] If the Boltzmann weights are not free fermionic, (\ref{Delta-Lattice-Hamiltonian-correspondence-equation}) does not hold for any choice of the $s_k^{(j)}$.
\end{enumerate}
\end{theorem}

There are a few technical reasons that we look for a relationship of the form (\ref{Delta-Lattice-Hamiltonian-correspondence-equation}). First, the $\tau$ function is independent of $M$, and increasing $M$ without changing $\lambda$ or $\mu$ adds more $\vx{a_1}$ vertices to each state. In other words, the Hamiltonian doesn't ``see'' $\vx{a_1}$ vertices. This means that the $\tau$ function is independent of $A_i$.

Similarly, the Hamiltonian doesn't ``see'' $\vx{b_1}$ vertices. For instance, let $M > \max(\lambda_1,\mu_1)$, and let $\tilde{\lambda} = (M,\lambda_1,\lambda_2,\ldots)$, $\tilde{\mu} = (M,\mu_1,\mu_2\ldots)$. Then $Z(\mathfrak{S}_{\tilde{\lambda}/\tilde{\mu}}) = \prod_{i=1}^N B_i\cdot Z(\mathfrak{S}_{\lambda/\mu})$, but the $\tau$ function is unchanged. Thus, the $\tau$ function is independent of $B_i$ as well.

This observation allows us to simplify our analysis. Since $A_i$ appears in every vertex in row $i$ of the lattice, the weight of every state must have the factor $\prod_i A_i^{M+1}$, and therefore so must the partition function. Similarly, note that $B_i$ appears in the weight of precisely the vertices with a $-$ spin on their bottom edge. By particle conservation, there must be the same number of vertical $-$ spins in each row or else both the partition function and $\tau$ function are 0. This means that in each admissible state of the model $B_i$ appears precisely $\ell(\lambda)$ times, and we must have $\ell(\lambda)=\ell(\mu)$.

Therefore, we have reduced Part a of Theorem \ref{Delta-lattice-Hamiltonian} to the following proposition.

\begin{proposition} \label{Reduced-Delta-lattice-Hamiltonian}
Set $A_i=B_i=1$ for all $i$. Then for all strict partitions $\lambda$ and $\mu$ with $\ell(\lambda) = \ell(\mu) = \ell$, \begin{align}Z(\mathfrak{S}_{\lambda/\mu}) = \cdot\langle\mu|e^{H_+}|\lambda\rangle \hspace{20pt} \text{for all strict partitions $\lambda,\mu$ and all $M,N$.} \label{Reduced-Lattice-Hamiltonian-correspondence-equation}\end{align} precisely when the Hamiltonian parameters are defined by (\ref{Delta-Hamiltonian-parameter}).
\end{proposition}

Our proof will involve cases of increasing generality. We will first prove the proposition in the case that $N=1$ and $\ell=1$, and then move on to the case where $N=1$ and $\ell$ is arbitrary. This second step involves Wick's theorem in a way that will make Part b of Theorem \ref{Delta-lattice-Hamiltonian} easy to prove. When $N=1$, we only have one set of parameters $x_1,y_1$, and $s_k = s_k^{(1)}$; we will often leave off the index in this case.

Finally, we will use a simple branching argument to prove Proposition \ref{Reduced-Delta-lattice-Hamiltonian} for arbitrary $N$.

\begin{lemma} \label{one-particle-lemma}
Proposition \ref{Reduced-Delta-lattice-Hamiltonian} is true in the case where $N=1$, $\lambda = (r+p)$, and $\mu = (r)$.
\end{lemma}

\begin{proof}
If $p<0$, both sides of (\ref{Reduced-Lattice-Hamiltonian-correspondence-equation}) are zero. Otherwise, $\mathfrak{S}_{\lambda/\mu}$ has exactly one admissible state. If $p=0$, then column $r$ has a vertex of type $\vx{b_1}$, and all other vertices are type $\vx{a_1}$. If $p>0$, then column $r$ has a vertex of type $\vx{c_1}$, columns $r+1,\ldots,r+p-1$ have vertices of type $\vx{b_2}$, column $r+p$ has a vertex of type $\vx{c_2}$, and all other vertices are type $\vx{a_1}$.

This gives \[Z(\mathfrak{S}_{\lambda/\mu}) = \begin{cases} (x+y)x^{p-1}, & \text{if } p\ge 1 \\ 1, & \text{if } p=0.\end{cases}\]

Now, we show that the Hamiltonian matches the partition function. Let \[H(t) = 1 + \sum_{p\ge 1} (x+y)x^{p-1} t^p.\] On the other hand, \[\langle\mu|e^{H_+}|\lambda\rangle = \langle 0|\psi_r e^{H_+} \psi^*_{p+r}|0\rangle = \langle 0|\psi_0 e^{H_+} \psi^*_p|0\rangle = \langle 0|\psi_0 e^{H_+} \psi^*(t)|0\rangle|_{t^p},\] so the result will follow once we can show that \[H(t) = \tilde{H}(t) := \langle 0|\psi_0 e^{H_+} \psi^*(t)|0\rangle.\]

By Lemma \ref{h_k-as-Hamiltonian}, \begin{align*}h_p = \langle 0|\psi_0 e^{H_+} \psi^*(t)|0\rangle|_{t^p} &= \left\langle 0\left|\psi_0\sum_{k\ge 0} \frac{1}{k!} \sum_{q_1+\ldots+q_k = p} s_{q_1}\ldots s_{q_k} J_{q_1}\ldots J_{q_k} \psi^*_p\right|0\right\rangle \notag\\&= \sum_{k\ge 0} \frac{1}{k!} \sum_{q_1+\ldots+q_k = p} s_{q_1}\ldots s_{q_k},\end{align*} so \begin{align*}\tilde{H}(t) &= \sum_{p\ge 0}\left(\sum_{k\ge 0} \frac{1}{k!} \sum_{q_1+\ldots+q_k = p} s_{q_1}\ldots s_{q_k}\right) t^p = \exp\left(\sum_{m\ge 1} s_m t^m\right).\end{align*}

Now, we can sum $H(t)$ as a geometric series: \begin{align*} H(t) = \frac{1+((x+y)-x)t}{1-xt} = \frac{1+yt}{1-xt},\end{align*} so \[\log H(t) = \log(1+yt) - \log(1-xt),\] and \begin{align*} \left. \left(\frac{d^n}{dt^n} \log H(t)\right)\right|_{t=0} &= \left.\left(-\frac{(n-1)!(-y)^n}{(1+yt)^n} + \frac{(n-1)!x^n}{(1-xt)^n}\right)\right|_{t=0} \\&= (n-1)!(x^n + (-1)^{n-1}y^n) \\&= n! s_n.\end{align*} Therefore, $H(t) = \tilde{H}(t)$.
\end{proof}

\begin{lemma} \label{one-row-lemma}
Proposition \ref{Reduced-Delta-lattice-Hamiltonian} is true when $N=1$, for arbitrary $\lambda,\mu$.
\end{lemma}

\begin{proof}
First assume that $\lambda$ and $\mu$ \emph{interleave} i.e. $\lambda_i\ge \mu_i\ge \mu_{i+1}$ for all $i$. In this case, $\mathfrak{S}_{\lambda/\mu}$ has exactly one admissible state, with vertices determined by the following table.

\begin{center}
\begin{tabular}{ |c|c| } 
 \hline
 Vertex in column $k$ & Condition ($i$ arbitrary)  \\\hline 
 $\vx{a_2}$ & $\mu_i=\lambda_{i+1}=k$  \\\hline
 $\vx{b_1}$ & $\lambda_i=\mu_i=k$  \\\hline
 $\vx{b_2}$ & $\lambda_i>k>\mu_i$  \\\hline
 $\vx{c_1}$ & $\mu_{i-1}>\lambda_i=k>\mu_i$  \\\hline
 $\vx{c_2}$ & $\lambda_i>k=\mu_i>\lambda_{i+1}$  \\\hline
 $\vx{a_1}$ & else  \\
 \hline
\end{tabular}
\end{center}

By the Jacobi-Trudi formula (Proposition \ref{Hamiltonian-Jacobi-Trudi}; note the $\rho$ shift), \[ \langle\mu|e^{H_+}|\lambda\rangle = \det_{1\le i,j\le \ell} h_{\lambda_i-\mu_j}.\] 

Let $\eta_{i,j} = h_{\lambda_i-\mu_j}$. Since $\lambda_i\ge\mu_i\ge\lambda_{i+1}$ for all $i$, and since $\lambda,\mu$ are strict, $\eta_{i,j} = 0$ whenever $j\ge i+1$, and \[\eta_{i,i+1} = \begin{cases} 1, & \mu_i = \lambda_{i+1} \\ 0, & \mu_i>\lambda_{i+1}.\end{cases}\]

On the other hand, by Lemma \ref{one-particle-lemma}, for all $p>0$, \[h_p = Z(\mathfrak{S}_{(p)/(0)}) = (x+y)\cdot x^{p-1}.\]

Note that $\eta_{i,j}$ with $i\le j$ doesn't appear in the determinant $\det (\eta_{i,j})$ unless $\mu_i = \lambda_{i+1}, \mu_{i+1}=\lambda_{i+1},\ldots \mu_{j-1} = \lambda_j$. Therefore, $\det(\eta_{i,j})$ is a product of blocks of the form \[\begin{bmatrix}\eta_{i,i} & 1 & & & \\ & & 1 & & \\ \ldots & & & \ldots & \\ & & & & 1 \\ \eta_{j,i} & & \ldots & & \eta_{j,j}\end{bmatrix}, \hspace{20pt} \text{where } i\le j,  \mu_i = \lambda_{i+1}, \mu_{i+1}=\lambda_{i+1},\ldots \mu_{j-1} = \lambda_j.\]

For all $i\le b\le a\le j$, $\lambda_b>\mu_a$, so $\eta_{x,y} = (x+y)x^{\lambda_b-\mu_a-1}$, and so the determinant of the block is \[\det \begin{bmatrix} (x+y) x^{\lambda_i-\mu_i-1} & 1 & & & \\ & & 1 & & \\ \ldots & & & \ldots & \\ & & & & 1 \\ (x+y) x^{\lambda_i-\mu_j-1} & & \ldots & & (x+y) x^{\lambda_j-\mu_j-1}\end{bmatrix} = \frac{x+y}{x}\cdot\left(\frac{x+y}{x}-1\right)^{j-i-1} x^{\lambda_i-\mu_j}.\]

Taking the product over all such blocks, we obtain \begin{align*}\det(\eta_{i,j}) &= x^{|\lambda|-|\mu|} \left(\frac{x+y}{x}-1\right)^\ell \left(\frac{x+y}{x+y-x}\right)^{\# \mu_i>\lambda_{i+1}}\cdot \left(\frac{x+y}{x}\right)^{-(\#\lambda_i=\mu_i)} \\&= x^{|\lambda|-|\mu|} \left(\frac{y}{x}\right)^\ell \left(\frac{x+y}{y}\right)^{\# \mu_i>\lambda_{i+1}}\cdot \left(\frac{x+y}{x}\right)^{-(\#\lambda_i=\mu_i)} \\&= x^{|\lambda|-|\mu|-\ell+(\#\lambda_i=\mu_i)} (x+y)^{(\#\mu_i>\lambda_{i+1})-(\#\lambda_i=\mu_i)} y^{\ell - (\#\mu_i>\lambda_{i+1})},\end{align*} where $|\lambda|$ is the sum of the parts of $\lambda$, likewise for $\mu$, and expressions like $\# \mu_i>\lambda_{i+1}$ represent the number of indices $i$ such that the statement is true.

Now, from the lattice model side, \[Z(\mathfrak{S}_{\lambda/\mu}) = x^{|\lambda|-|\mu|-(\#\lambda_i>\mu_i)} \cdot (x+y)^{\ell - (\#\lambda_i=\mu_i) - (\#\mu_i = \lambda_{i+1})} \cdot y^{\#\mu_i=\lambda_{i+1}} = \det(\eta_{i,j}).\]

If $\lambda$ and $\mu$ do not interleave, then it is easy to see that $Z(\mathfrak{S}_{\lambda/\mu})=0$. For some $a$ either $\lambda_a<\mu_a$, in which case $\langle\mu|e^{H_+}|\lambda\rangle=0$, or $\mu_a<\lambda_{a+1}$. Notice that $h_{p+r}/h_r$ is independent of $r$ as long as $r\ge 1$. We can use this fact to do column operations on the matrix $(\eta_{i,j})$, and obtain a matrix $(M_{ij})$ where $M_{ij}=0$ whenever $\lambda_{i+1}>\mu_j$. In particular, every nonzero entry in the first $a$ columns will be in the first $a-1$ rows. But this means that $\det(\eta_{ij}) = \det(M_{ij})=0$, so the result still holds.
\end{proof}

The proof of Proposition \ref{Reduced-Delta-lattice-Hamiltonian} is now formal. Both lattice models and Hamiltonians ``branch'' in the same way, so we simply sum over the intermediate partitions. In similar contexts, this is sometimes called a \emph{Miwa transform}. For clarity, we will write $\mathfrak{S}^N_{\lambda/\mu}$ for the usual $\mathfrak{S}$ lattice model with $N$ rows.

\begin{proof}[Proof of Proposition \ref{Reduced-Delta-lattice-Hamiltonian}]
\begin{align*} \langle\mu|e^{H_+}|\lambda\rangle &= \langle\mu|e^{\phi_N}\ldots e^{\phi_1}|\lambda\rangle \\&= \sum_{\nu_1,\ldots,\nu_{N-1}} \langle\mu|e^{\phi_N}|\nu_{N-1}\rangle \langle\nu_{N-1}|e^{\phi_{N-1}}|\nu_{N-2}\rangle\ldots \langle\nu_1|e^{\phi_1}|\lambda\rangle \\&= \sum_{\nu_1,\ldots,\nu_{N-1}} Z(\mathfrak{S}^1_{\nu_{N-1}/\mu})\ldots Z(\mathfrak{S}^1_{\lambda/\nu_1}) \\&= Z(\mathfrak{S}^N_{\lambda/\mu}).\end{align*}
\end{proof}

This completes the proof of Theorem \ref{Delta-lattice-Hamiltonian}(a).

\begin{proof}[Proof of Theorem \ref{Delta-lattice-Hamiltonian}(b)]
Let both the Boltzmann weights and the Hamiltonian parameters be arbitrary. In other words, we are using the general model $\ba{\mathfrak{S}}_{\lambda/\mu}$. By Remark \ref{c-weights-remark} and the discussion after the statement of Theorem \ref{Delta-lattice-Hamiltonian}, we may set $\vx{a_1^{(j)}}=\vx{b_1^{(j)}}=\vx{c_2^{(j)}}=1$, and check when (\ref{Reduced-Lattice-Hamiltonian-correspondence-equation}) holds.

By performing the calculations in Lemma \ref{one-particle-lemma}, and using the same notation, we get \[H(t) = \frac{1 + (\vx{c_1} - \vx{b_2})t}{1+\vx{b_2}t}, \hspace{20pt} \] so \[ \left. \left(\frac{d^n}{dz^n} \log H(z)\right)\right|_{z=0} =  (n-1)!(\vx{b_2}^n + (-1)^{n-1}(\vx{c_1}-\vx{b_2})^n).\] This means we must have $s_n = \frac{1}{n}(\vx{b_2}^n + (-1)^{n-1}(\vx{c_1}-\vx{b_2})^n)$.

Now, performing the calculations in Lemma \ref{one-row-lemma}, if $\lambda$ and $\mu$ interleave, \begin{align*}\langle\mu|e^{H_+}|\lambda\rangle &= \vx{b_2}^{|\lambda|-|\mu|} \cdot \left(\frac{\vx{c_1}}{\vx{b_2}}-1\right)^\ell \cdot \left(\frac{\vx{c_1}}{\vx{c_1}-\vx{b_2}}\right)^{\# \mu_i>\lambda_{i+1}}\cdot \left(\frac{\vx{c_1}}{\vx{b_2}}\right)^{-(\#\lambda_i=\mu_i)} \\&= \vx{b_2}^{|\lambda|-|\mu|-(\#\lambda_i>\mu_i)} \cdot \vx{c_1}^{\ell - (\#\lambda_i=\mu_i) - (\#\mu_i = \lambda_{i+1})} \cdot (\vx{c_1}-\vx{b_2})^{\#\mu_i=\lambda_{i+1}}.\end{align*}

On the other hand, \[Z(\mathfrak{S}_{\lambda/\mu}) = \vx{b_2}^{|\lambda|-|\mu|-(\#\lambda_i>\mu_i)} \cdot \vx{c_1}^{\ell - (\#\lambda_i=\mu_i) - (\#\mu_i = \lambda_{i+1})} \cdot \vx{a_2}^{\#\mu_i=\lambda_{i+1}},\] so equality holds if and only if $\vx{a_2} = \vx{c_1}-\vx{b_2}$, which is the free fermion condition.

\end{proof}

This completes the proof of Theorem \ref{Delta-lattice-Hamiltonian}. The following corollary is just restatement of the Theorem in terms of the Boltzmann weights, for convenience.

\begin{corollary}
The general $\mathfrak{S}$ lattice model $\ba{\mathfrak{S}}_{\lambda/\mu}$ matches the Hamiltonian $e^{H_+}$ (see (\ref{General-Lattice-Hamiltonian-correspondence-equation})) if and only if:
\begin{enumerate}
    \item[(a)] The Boltzmann weights of $\ba{\mathfrak{S}}_{\lambda/\mu}$ are free fermionic.
    \item[(b)] \begin{align*} s_k^{(j)} = \frac{1}{k}\left(\left(\frac{\vx{b_2^{(j)}}}{\vx{a_1^{(j)}}}\right)^k + (-1)^{k-1}\left(\frac{\vx{a_2^{(j)}}}{\vx{b_1^{(j)}}}\right)^k\right) \hspace{20pt} \text{for all } k\ge 1, j\in [1,N].\end{align*}
    \item[(c)] The extra factor is ${\Large{*}} = \prod_{i=1}^N (\vx{a_1^{(i)}})^{M-\ell} (\vx{b_1^{(i)}})^\ell.$ 
\end{enumerate}
\end{corollary}

\subsection{The $\mathfrak{S}^*$ lattice model}

Now we prove a similar identity for the $\mathfrak{S}^*$ lattice model. We will say that the lattice model $\mathfrak{S}$ and the Hamiltonian operator $e^{H_-}$ \emph{match} if the following condition holds: \begin{align}Z(\mathfrak{S}^*_{\lambda/\mu}) = \prod_{i=1}^{N} A_i^{-(M+1)} B_i^{-\ell(\lambda)} \cdot\langle\lambda|e^{H_-}|\mu\rangle \hspace{20pt} \text{for all strict partitions $\lambda,\mu$ and all $M,N$.} \label{Gamma-Lattice-Hamiltonian-correspondence-equation}\end{align}

\begin{theorem} \label{Gamma-lattice-Hamiltonian}
\leavevmode
\begin{enumerate}
    \item[(a)] (\ref{Gamma-Lattice-Hamiltonian-correspondence-equation}) holds precisely when \begin{align} s_{-k}^{(j)} = \frac{1}{k}\left(y_i^k + (-1)^{k-1}x_i^k\right) \hspace{20pt} \text{for all } k\ge 1, j\in [1,N]. \label{Gamma-Hamiltonian-parameter}\end{align}
    \item[(b)] If the Boltzmann weights are not free fermionic, (\ref{Gamma-Lattice-Hamiltonian-correspondence-equation}) does not hold for any choice of the $s_{-k}^{(j)}$.
\end{enumerate}
\end{theorem}

\begin{proof}
Note that taking (\ref{Delta-Hamiltonian-parameter}) and (\ref{Gamma-Hamiltonian-parameter}), $s_{-k}^{(j)} = s_k^{(j)}|_{x_i\leftrightarrow y_i}$, so $\langle \lambda|e^{H_-}|\mu\rangle = \langle \mu|e^{H_+}|\lambda\rangle|_{x_i\leftrightarrow y_i}$. Then, \begin{flalign*} && Z(\mathfrak{S}^*_{\lambda/\mu}(\boldsymbol{x}, \boldsymbol{y}, \boldsymbol{A}, \boldsymbol{B})) &= Z(\mathfrak{S}_{\lambda/\mu}(\boldsymbol{y}, \boldsymbol{x}, \boldsymbol{A}^{-1}, \boldsymbol{B}^{-1})) && \text{by (\ref{Simple-Gamma-Delta-relationship})} \\& &&= \left.\prod_{i=1}^{N} A_i^{M+1} B_i^{\ell(\lambda)} \cdot\langle\mu|e^{H_+}|\lambda\rangle\right|_{x_i\leftrightarrow y_i, A_i\mapsto A_i^{-1}, B_i\mapsto B_i^{-1}} && \text{by Theorem \ref{Delta-lattice-Hamiltonian}} \\& &&= \prod_{i=1}^{N} A_i^{-(M+1)} B_i^{-\ell(\lambda)} \cdot\langle\lambda|e^{H_-}|\mu\rangle\end{flalign*}
\end{proof}

\begin{corollary}
The general $\mathfrak{S}^*$ lattice model $\ba{\mathfrak{S}}^*_{\lambda/\mu}$ matches the Hamiltonian $e^{H_-}$ (see (\ref{General-Lattice-Hamiltonian-correspondence-equation})) if and only if:
\begin{enumerate}
    \item[(a)] The Boltzmann weights of $\ba{\mathfrak{S}}^*_{\lambda/\mu}$ are free fermionic.
    \item[(b)] \begin{align*} s_{-k}^{(j)} = \frac{1}{k}\left(\left(\frac{\vx{b_2^{(j)}}}{\vx{a_1^{(j)}}}\right)^k + (-1)^{k-1}\left(\frac{\vx{a_2^{(j)}}}{\vx{b_1^{(j)}}}\right)^k\right) \hspace{20pt} \text{for all } k\ge 1, j\in [1,N].\end{align*}
    \item[(c)] The extra factor is ${\Large{*}} = \prod_{i=1}^N (\vx{a_1^{(i)}})^{M-\ell} (\vx{b_1^{(i)}})^\ell.$ 
\end{enumerate}
\end{corollary}

\section{The free fermionic partition function} \label{corollaries-section}

The main result of this section is that the free fermionic partition function is a (skew) supersymmetric Schur function up to a simple factor. This fact will give us new proofs of Cauchy, Pieri, Jacobi-Trudi, and branching identities for these functions, as well as a six-vertex free fermionic analogue of the Lindstr\"{o}m-Gessel-Viennot (LGV) Lemma and a mysterious positivity result.

\subsection{Supersymmetric Schur functions}

Supersymmetric functions are a generalization of symmetric functions with two sets of parameters $\boldsymbol{x} = x_1,\ldots,x_n$ and $\boldsymbol{y} = y_1,\ldots,y_n$. This exposition is taken from Macdonald \cite[\S 6]{Macdonald-Schur-functions-variations}.

Let $[x|\boldsymbol{y}]^r = (x+y_1)(x+y_2)\cdots (x+y_r)$. If $\lambda$ is a partition, let \[A_\lambda := \det\left((x_i|\boldsymbol{y})^{\lambda_j}\right).\] Then \[s_\lambda[\boldsymbol{x}|\boldsymbol{y}] := \frac{A_{\lambda+\rho}}{A_\rho}\] is a symmetric polynomial in the $x_i$, and $s_\lambda[\boldsymbol{x}|\boldsymbol{0}] = s_\lambda[\boldsymbol{x}]$, the Schur polynomial associated to $\lambda$.

Then for an integer $r\ge 0$ we define \[h_r[\boldsymbol{x}|\boldsymbol{y}] := s_{(r)}[\boldsymbol{x}|\boldsymbol{y}], \hspace{20pt} e_r[\boldsymbol{x}|\boldsymbol{y}] := s_{(1^r)}[\boldsymbol{x}|\boldsymbol{y}].\] Macdonald then defines skew supersymmetric Schur functions via Jacobi-Trudi formulas: if $\ell(\lambda)=\ell(\mu)=\ell$, \[s_{\lambda/\mu}[\boldsymbol{x}|\boldsymbol{y}] := \det_{1\le i,j\le\ell} h_{\lambda_i-\mu_j-i+j}[\boldsymbol{x}|\boldsymbol{y}] = \det_{1\le i,j\le\ell} e_{\lambda_i'-\mu_j'-i+j}[\boldsymbol{x}|\boldsymbol{y}]\]

\subsection{Evaluating the partition function} \label{free-fermionic-partition-function-section}

Brubaker and Schultz showed in \cite{Brubaker-Schultz} that the partition functions of a large class of free fermionic lattice models are supersymmetric Schur functions.

\begin{proposition}\cite[\S A.3.2]{Brubaker-Schultz} \label{Supersymmetric-Schur-Prop}
\leavevmode
\begin{enumerate}
    \item[(a)] If $s_k^{(j)} = \frac{1}{k}(x_j^k + (-1)^{k-1} y_j^k)$, then \[\langle \mu+\rho|e^{H_+}|\lambda+\rho\rangle = s_{\lambda/\mu}[x_1,\ldots,x_n|y_1,\ldots,y_n].\]
    \item[(b)] If $s_{-k}^{(j)} = \frac{1}{k}(y_j^k + (-1)^{k-1} x_j^k)$, then \[\langle \lambda+\rho|e^{H_-}|\mu+\rho\rangle = s_{\lambda/\mu}[y_1,\ldots,y_n|x_1,\ldots,x_n].\]
\end{enumerate}
\end{proposition}

By scaling these Hamiltonians, or by scaling the lattice models in \cite{Brubaker-Schultz}, we can compute the free fermionic partition function.

\begin{corollary}
\label{Partition-Function-is-Supersymmetric-Schur}
\[Z(\mathfrak{S}_{\lambda+\rho/\mu+\rho}) = \prod_{i=1}^{N} A_i^{M+1}B_i^{\ell(\lambda)} \cdot s_{\lambda/\mu}[x_1,\ldots,x_n|y_1,\ldots,y_n],\] and similarly
\[Z(\mathfrak{S}^*_{\lambda+\rho/\mu+\rho}) = \prod_{i=1}^{N} A_i^{-(M+1)}B_i^{-\ell(\lambda)} \cdot s_{\lambda/\mu}[y_1,\ldots,y_n|x_1,\ldots,x_n].\]
\end{corollary}

\begin{proof}
This follows from Theorems \ref{Delta-lattice-Hamiltonian} and \ref{Gamma-lattice-Hamiltonian}, and Proposition \ref{Supersymmetric-Schur-Prop}.
\end{proof}

For completeness, we will rewrite these formulas in terms of the vertex weights directly. \begin{align}Z(\mathfrak{S}_{\lambda+\rho/\mu+\rho}) = \prod_{i=1}^{N} \vx{a_1^{(i)}}^{M+1-\ell(\lambda)}\vx{b_1^{(i)}}^{\ell(\lambda)} \cdot s_{\lambda/\mu}\left[\left.\frac{\vx{b_2^{(1)}}}{\vx{a_1^{(1)}}},\ldots,\frac{\vx{b_2^{(N)}}}{\vx{a_1^{(N)}}}\right|\frac{\vx{a_2^{(1)}}}{\vx{b_1^{(1)}}},\ldots,\frac{\vx{a_2^{(N)}}}{\vx{b_1^{(N)}}}\right], \label{Delta-partition-function}\end{align} \begin{align}Z(\mathfrak{S}^*_{\lambda+\rho/\mu+\rho}) = \prod_{i=1}^{N} \vx{a_1^{(i)}}^{M+1-\ell(\lambda)}\vx{b_1^{(i)}}^{\ell(\lambda)} \cdot s_{\lambda/\mu}\left[\left.\frac{\vx{b_2^{(1)}}}{\vx{a_1^{(1)}}},\ldots,\frac{\vx{b_2^{(N)}}}{\vx{a_1^{(N)}}}\right|\frac{\vx{a_2^{(1)}}}{\vx{b_1^{(1)}}},\ldots,\frac{\vx{a_2^{(N)}}}{\vx{b_1^{(N)}}}\right], \label{Gamma-partition-function}\end{align}

The weights in (\ref{Delta-partition-function}) are the $\mathfrak{S}$ weights (Figure \ref{Delta-vertex-weights}), while the weights in (\ref{Gamma-partition-function}) are the $\mathfrak{S}^*$ weights (Figure \ref{Gamma-vertex-weights}).

This result, along with Theorem \ref{partition-function-all-boundaries} can be seen as a replacement of \cite[Theorem~9]{BBF-Schur-polynomials}. That result holds in the case of the 5-vertex free fermionic model, but there is an error in that proof when applied to the full 6-vertex free fermionic model.

We now show that our involution on lattice models defined in Section \ref{six-vertex-section} is equivalent to the involution on generalized symmetric functions defined in Section \ref{background-section}. In other words, the supersymmetric involution can be written \emph{diagrammatically} in terms of lattice model manipulations.

Extend the involution $\omega$ from Section \ref{background-section} so that it sends $A_i\mapsto A_i^{-1}$, $B_i\mapsto B_i^{-1}$. Then, the involutions $\omega$ and $\tilde{\omega}$ are the same.

\begin{corollary}
\[\omega(Z(\mathfrak{S}_{\lambda/\mu})) = \tilde{\omega}(Z(\mathfrak{S}_{\lambda/\mu})),\] and \[\omega(Z(\mathfrak{S}^*_{\lambda/\mu})) = \tilde{\omega}(Z(\mathfrak{S}^*_{\lambda/\mu})).\]
\end{corollary}

\begin{proof}
By Theorems \ref{Delta-lattice-Hamiltonian} and \ref{Gamma-lattice-Hamiltonian}, \begin{align*}\omega(Z(\mathfrak{S}_{\lambda/\mu})) &= \omega\left(\prod_i A_i^{M+1} B_i^{\ell(\lambda)} \langle\mu|e^{H_+}|\lambda\rangle\right) \\&= \prod_i A_i^{-(M+1)} B_i^{-\ell(\lambda)} \langle\lambda|e^{H_-}|\mu\rangle \\&= Z(\mathfrak{S}^*_{\lambda/\mu}) \\&= \tilde{\omega}(Z(\mathfrak{S}_{\lambda/\mu})).\end{align*} The proof of the second statement is similar
\end{proof}

In particular, using the model $\mathfrak{S}_{\lambda+\rho/\mu+\rho}(\boldsymbol{x}, \boldsymbol{y}, \boldsymbol{1}, \boldsymbol{1})$, we have the expected involution on supersymmetric Schur functions: \[\omega(s_{\lambda/\mu}[\boldsymbol{x}|\boldsymbol{y}]) = \tilde{\omega}(s_{\lambda/\mu}[\boldsymbol{x}|\boldsymbol{y}]) = s_{\lambda/\mu}[\boldsymbol{y}|\boldsymbol{x}].\]

In addition, the Jacobi-Trudi formula (Proposition \ref{Hamiltonian-Jacobi-Trudi}), applied to the lattice model, is a six-vertex free fermionic analogue of the Lindstr\"{o}m-Gessel-Viennot (LGV) Lemma.

For any $r$, $h_k = Z(\mathfrak{S}_{(k+r)/(r)})|_{A_i=B_i=1}$ and $e_k = Z(\mathfrak{S}^*_{(k+r)/(r)})|_{A_i=B_i=1}$, so we have the following.

\begin{proposition}[Six-vertex LGV Lemma] \label{six-vertex-LGV}
\[Z(\mathfrak{S}_{\lambda/\mu}) = \det_{1\le i,j\le \ell(\lambda)} Z(\mathfrak{S}_{(\lambda_i)/(\mu_j)})\] and \[Z(\mathfrak{S}^*_{\lambda/\mu}) = \det_{1\le i,j\le \ell(\lambda)} Z(\mathfrak{S}^*_{(\lambda_i)/(\mu_j)}),\] where if the models on the left side have $M$ columns, and the models on the right side have $M_{ij}$ columns, we have $M = \sum_i M_{i,\sigma(i)}$ for all permutations $\sigma$.
\end{proposition}

This reduces to the usual case of the LGV lemma on a five-vertex model. If we set $y_i=0, A_i=B_i=1$, the resulting model gives the well-known equivalence between the tableaux and Jacobi-Trudi definitions of Schur functions, and if we additionally set $x_i=1$, it gives a determinantal count of sets of non-intersecting lattice paths. We could do the same thing for sets of ``osculating paths'' by instead setting $x_i=0, y_i=A_i=B_i=1$. However, we cannot get an unweighted count of states Boltzmann for the six-vertex model from this formula since specializing all weights to 1 would violate the free fermion condition. 

\subsection{Identities: Pieri rule, Cauchy identity, and branching rule}

Combining Corollary \ref{Partition-Function-is-Supersymmetric-Schur} with the results in Section \ref{Hamiltonian-identities-section}, we get new proofs of the Pieri rule, Cauchy identity, and branching rule for supersymmetric Schur polynomials. Equivalently, we get those same identities for the free fermionic partition function.

The branching rule is straightforward to prove in both the context of the Hamiltonian and the lattice model. The Cauchy identity can also be proved via both methods. Since our Boltzmann weights are free fermionic, the lattice model is solvable by \cite[Theorem~1]{BBF-Schur-polynomials}. Then one can prove the dual Cauchy identity by repeatedly applying the Yang-Baxter equation to a ``combined lattice model''. See \cite[Theorem~7]{Bump-McNamara-Nakasuji} for an example of this process. On the other hand, the Pieri rule appears to be more easily proved using the Hamiltonian.

\begin{proposition}[Cauchy identity] \label{supersymmetric-Cauchy-identity} For any strict partitions $\lambda$ and $\mu$, \begin{align*}\sum_\nu s_{\lambda/\nu}[\boldsymbol{x}|\boldsymbol{y}] s_{\mu/\nu}[\boldsymbol{z}|\boldsymbol{w}] = \prod_{i,j}\frac{(1-x_iz_j)(1-y_iw_j)}{(1+x_iw_j)(1+y_iz_j)} \cdot \sum_\nu s_{\nu/\mu}[\boldsymbol{x}|\boldsymbol{y}] s_{\nu/\lambda}[\boldsymbol{z}|\boldsymbol{w}],\end{align*} where the sums are over all strict partitions $\nu$.
\end{proposition}

\begin{proof}
From Proposition \ref{Hamiltonian-Cauchy-identity}, and using the notation from that proposition, we have the formula \[\sum_\nu \sigma_{\lambda/\nu}\sigma'_{\mu/\nu} = \prod_{i,j}\exp\left(\sum_{k\ge 1} k\cdot s_k^{(i)}s_{-k}^{(j)}\right) \cdot \sum_\nu \sigma_{\nu/\mu}\sigma'_{\nu/\lambda},\] so we just need to prove that when $s_k^{(i)} = \frac{1}{k}\left(x_i^k + (-1)^{k-1}y_i^k\right)$ and $s_{-k}^{(i)} = \frac{1}{k}\left(z_i^k + (-1)^{k-1}w_i^k\right)$, \[\exp\left(\sum_{k\ge 1} k\cdot s_k^{(i)}s_{-k}^{(j)}\right) = \frac{(1-x_iz_j)(1-y_iw_j)}{(1+x_iw_j)(1+y_iz_j)}.\] This follows by exponentiating the next string of equalities. \begin{align*} \sum_{k\ge 1} k\cdot s_k^{(i)}s_{-k}^{(j)} &= \sum_{k\ge 1} k\cdot \frac{1}{k}\left(x_i^k + (-1)^{k-1}y_i^k\right)\cdot \frac{1}{k}\left(z_j^k + (-1)^{k-1}w_j^k\right) \\&= \sum_{k\ge 1} \frac{1}{k}\left(x_i^kz_j^k + (-1)^{k-1}x_i^kw_j^k + (-1)^{k-1}y_i^kz_j^k + y_i^kw_j^k\right) \\&= \log(1+x_iw_j) + \log(1+y_iz_j) -\log(1-x_iz_j) - \log(1-y_iw_j)  \\&= \log\left(\frac{(1+x_iw_j)(1+y_iz_j)}{(1-x_iz_j)(1-y_iw_j)}\right). \end{align*}
\end{proof}

The branching rule and Pieri rule all follow directly from their Hamiltonian analogues in Section \ref{Hamiltonian-symmetric-functions-section}: Propositions \ref{Hamiltonian-branching-rule} and \ref{Hamiltonian-Pieri-rule}, respectively.

\begin{proposition}[Branching rule] \label{supersymmetric-branching-rule}
For all partitions $\lambda,\mu$, \[s_{\lambda/\mu}[x_1,\ldots,x_n|y_1,\ldots,y_n] = \sum_{\nu} s_{\mu/\nu}[x_1|y_1] s_{\lambda/\mu}[x_2,\ldots,x_n|y_2,\ldots,y_n].\]
\end{proposition}

\begin{proposition}[Pieri rule] \label{supersymmetric-Pieri-rule}
\[h_k[\boldsymbol{x}|\boldsymbol{y}] \cdot s_\lambda[\boldsymbol{x}|\boldsymbol{y}] = \sum_\nu \left\langle\nu+\rho|U_k|\lambda+\rho\right\rangle s_\nu[\boldsymbol{x}|\boldsymbol{y}].\]
\end{proposition}

\subsection{Positivity}

Suppose for this subsection that our lattice model parameters $A_i,B_i,x_i,y_i\in\R$. It turns out that the Hamiltonian interpretation provides a positivity condition for the partition function $Z(\mathfrak{S}_{\lambda/\mu})$.

We will ask when $\mathfrak{S}$ satisfies the following condition: \begin{align} \text{The partition function $Z(\mathfrak{S}_{\lambda/\mu}(\boldsymbol{x}, \boldsymbol{y}, \boldsymbol{1}, \boldsymbol{1}))\ge 0$ for all strict partitions $\lambda,\mu$.}\label{positivity-condition}\end{align}

\begin{proposition}
$\mathfrak{S}$ satisfies (\ref{positivity-condition}) if and only if  \[x_i,y_i\ge 0 \hspace{20pt} \text{for all } 1\le i\le n.\]
\end{proposition}

\begin{proof}
By Theorem \ref{Delta-lattice-Hamiltonian}, the map $\phi:\Lambda\to\R$ defined by \[s_k \mapsto \frac{1}{k}\sum_j \left(x_j^k + (-1)^{k-1}y_j^k\right)\] has the property that $\phi(\sigma_{\lambda/\mu}) = Z(\mathfrak{S}_{\lambda/\mu}(\boldsymbol{x}, \boldsymbol{y}, \boldsymbol{1}, \boldsymbol{1}))$ for all strict partitions $\lambda,\mu$.

By the Edrei-Thoma Theorem (see \cite[Proposition~1.3]{Matveev-Macdonald}), $\phi(\sigma_{\lambda/\mu})$ is positive for all $\lambda,\mu$ precisely when $x_j,y_j\ge 0$ for all $j$.
\end{proof}

In other words, if a free fermionic six-vertex model satisfies \[\frac{\vx{b_2^{(j)}}}{\vx{a_1^{(j)}}}, \frac{\vx{a_2^{(j)}}}{\vx{b_1^{(j)}}}\ge 0 \hspace{20pt} \text{for all } 1\le j\le n,\] then the partition function has a predictable sign that only depends on $M$ and $\ell(\lambda)$. It is unclear whether any similar result holds for non-free-fermionic weights, or whether there is a probabilistic interpretation of this result.

\section{Boundary Conditions} \label{boundary-conditions-section}

We will use the results of the previous sections to compute the free fermionic partition function for a model with modified left and right boundary conditions. For any number of rows and any left and right boundary conditions, we give an operator on Fock space that matches the partition function; however, this operator is ugly in general. In the case where the left and right boundaries are both uniform (e.g. domain-wall), we can use a different method to give a precise formula for the partition function.

\subsection{A Fock space operator for any boundary conditions}

We'll work with a version of $\mathfrak{S}$ with generalized boundary conditions. Similar results hold for $\mathfrak{S}^*$. If $\alpha,\beta,\lambda,\mu$ are strict partitions where all parts of $\alpha$ and $\beta$ are positive, let $\mathfrak{S}_{\lambda/\mu}^{\alpha/\beta} := (\mathfrak{S})_{\lambda/\mu}^{\alpha/\beta}$ be defined as follows.

\begin{itemize}
    \item $N$ rows, where $N\ge \max(\sigma_1,\tau_1)$, labelled $1,\ldots, N$ from bottom to top;
    \item $M+1$ columns, where $M\ge \max(\lambda_1,\mu_1)$, labelled $0,\ldots,M$ from left to right;
    \item Right boundary edges $-$ on parts of $\alpha$; $+$ otherwise;
    \item Left boundary edges $-$ on parts of $\beta$; $+$ otherwise;
    \item Bottom boundary edges $-$ on parts of $\lambda$; $+$ otherwise;
    \item Top boundary edges $-$ on parts of $\mu$; $+$ otherwise;
    \item Boltzmann weights from Figure \ref{Delta-vertex-weights}.
\end{itemize}

Given $1\le i\le n$, consider 4 cases:
\begin{enumerate}
    \item[A)] Neither $\alpha$ nor $\beta$ has a part of size $i$.
    \item[B)] $\alpha$ has a part of size $i$, but $\beta$ does not.
    \item[C)] $\beta$ has a part of size $i$, but $\alpha$ does not.
    \item[D)] Both $\alpha$ and $\beta$ have parts of size $i$.
\end{enumerate}

Define \[e^{\Phi_i} := \begin{cases} e^{\phi_i}, & \text{if case A,} \\ (x_i+y_i)^{-1}\cdot e^{\phi_i} \psi^*_{M+1/2}, & \text{if case B,} \\ \psi^*_{-3/2}\psi_{-3/2} e^{\phi_i} \psi_{-3/2}, & \text{if case C,} \\ (x_i+y_i)^{-1} \psi^*_{-3/2} \psi_{-3/2} e^{\phi_i} \psi^*_{M+1/2}\psi_{-3/2}, & \text{if case D.}\end{cases}\] Here is the general idea behind the operators $e^{\Phi_i}$. In cases B and D, we introduce a particle in column $M+1$, while in cases C and D, we remove a particle from column $-1$. This adjustment corresponds on the lattice model side to creating a ``ghost vertex'' of type $\vx{c_1^{(i)}}$ on the right of the row (resp. type $\vx{c_2^{(i)}}$ on the left of the row. We then divide by the weight of the ghost vertex since it doesn't appear in the lattice model.

The particle removal has another wrinkle. We first remove a particle from column -1 by applying $\psi_{-3/2}$, which introduces a factor of $(-1)^{\ell(\lambda)}$. Then the operator $e^{\phi_i}$ fills that empty spot up. Then the operator $\psi^*_{-3/2}\psi_{-3/2}$ serves as a ``check'', killing the state unless there is a particle in column -1.

Let $S:\mathcal{F}_\ell\to \mathcal{F}_{\ell+1}$ be the shift operator defined by $S|\rho_\ell\rangle = |\rho_{\ell+1}\rangle$. Note that $S$ commutes with current operators, and therefore with $e^{\phi_i}$. Let $\emptyset$ denote the empty partition (no parts).

\begin{proposition}
\begin{align} Z((\mathfrak{S})^{\alpha/\beta}_{\lambda/\mu}) = \prod_{i=0}^N (-1)^{\ell(\lambda) + \delta_i} A_i^{M+1} B_i^{\ell(\lambda) + \delta_i} \langle \mu|e^{\Phi_N} e^{\Phi_{N-1}}\ldots e^{\Phi_1} e^{\Phi_0}|\lambda\rangle, \label{boundary-operator-equation}\end{align} where \[\delta_i = |\{j|\alpha_j < i\}| - |\{j|\beta_j < i\}|\]
\end{proposition}

\begin{proof}
We prove this first for a single row, and drop the subscripts $i=0$. If $N=1$, we want to prove \[Z((\mathfrak{S})^{\alpha/\beta}_{\lambda/\mu}) = A_1^M B_1^{\ell(\lambda)} \langle \mu|e^{\Phi_1}|\lambda\rangle\] for all possible $\alpha,\beta\in \{(1),\emptyset\}$. We have four cases, corresponding to the cases above: A) $\alpha=\beta=\emptyset$; B) $\alpha=(1), \beta=\emptyset$; C) $\alpha=\emptyset, \beta=(1)$; D) $\alpha=\beta=(1)$. We will prove case D, and the others are similar.

Let $\tilde{\lambda} = (M+2,\lambda_1+1,\ldots,\lambda_{\ell(\lambda)}+1), \widetilde{\mu} = (\mu_1+1,\ldots,\mu_{\ell(\mu)}+1,0)$. Note that \[|\tilde{\lambda}\rangle = (-1)^{\ell(\lambda)}S\psi^*_{M+1/2}\psi_{-3/2}|\lambda\rangle \hspace{20pt} \text{and} \hspace{20pt} \langle\tilde{\mu}| = \langle\mu|\psi^*_{-3/2}\psi_{-3/2}S^{-1},\] so \begin{align*} \langle \mu|e^{\Phi}|\lambda\rangle &= (x+y)^{-1} \langle\mu|\psi^*_{-3/2} \psi_{-3/2} e^{\phi} \psi^*_{M+1/2}\psi_{-3/2}|\lambda\rangle \\&= (-1)^{\ell(\lambda)} (x+y)^{-1}\langle\tilde{\mu}|Se^{\phi}S^{-1}|\tilde{\lambda}\rangle \\&= (-1)^{\ell(\lambda)} (x+y)^{-1}\langle\tilde{\mu}|e^{\phi}|\tilde{\lambda}\rangle \\&= (-1)^{\ell(\lambda)} (x+y)^{-1} A^{-(M+3)} B^{-\ell(\lambda)-1} Z(\mathfrak{S}_{\tilde{\lambda}/\tilde{\mu}}) \\&= (-1)^{\ell(\lambda)} A^{-(M+1)} B^{-\ell(\lambda)} Z(\mathfrak{S}^{(1)/(1)}_{\tilde{\lambda}/\tilde{\mu}}).\end{align*}

Here, the last two equalities are because $\mathfrak{S}_{\tilde{\lambda}/\tilde{\mu}}$ is the following lattice model,

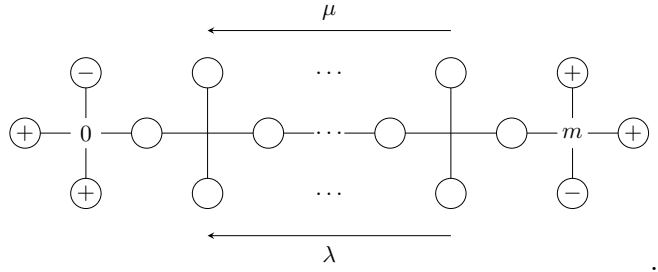
\[\scalebox{0.8}{
\begin{tikzpicture}
  \coordinate (aa) at (1,1);
  \coordinate (ab) at (1,2);
  \coordinate (ac) at (1,3);
  \coordinate (ba) at (2,1);
  \coordinate (bb) at (2,2);
  \coordinate (bc) at (2,3);
  \coordinate (ca) at (3,1);
  \coordinate (cb) at (3,2);
  \coordinate (cc) at (3,3);
  \coordinate (da) at (4,1);
  \coordinate (db) at (4,2);
  \coordinate (dc) at (4,3);
  \coordinate (ea) at (5,1);
  \coordinate (eb) at (5,2);
  \coordinate (ec) at (5,3);
  \coordinate (fa) at (6,1);
  \coordinate (fb) at (6,2);
  \coordinate (fc) at (6,3);
  \coordinate (ga) at (7,1);
  \coordinate (gb) at (7,2);
  \coordinate (gc) at (7,3);
  \coordinate (ha) at (8,1);
  \coordinate (hb) at (8,2);
  \coordinate (hc) at (8,3);
  \coordinate (ia) at (9,1);
  \coordinate (ib) at (9,2);
  \coordinate (ic) at (9,3);
  \coordinate (ja) at (10,1);
  \coordinate (jb) at (10,2);
  \coordinate (jc) at (10,3);
  \coordinate (ka) at (11,1);
  \coordinate (kb) at (11,2);
  \coordinate (kc) at (11,3);
  \draw (ab)--(kb);
  \draw (ba)--(bc);
  \draw (da)--(dc);
  \draw (ha)--(hc);
  \draw (ja)--(jc);
  \draw[fill=white] (ba) circle (.25);
  \draw[fill=white] (da) circle (.25);
  \draw[fill=white] (ha) circle (.25);
  \draw[fill=white] (ja) circle (.25);
  \draw[fill=white] (bc) circle (.25);
  \draw[fill=white] (dc) circle (.25);
  \draw[fill=white] (hc) circle (.25);
  \draw[fill=white] (jc) circle (.25);
  \draw[fill=white] (ab) circle (.25);
  \draw[fill=white] (cb) circle (.25);
  \draw[fill=white] (eb) circle (.25);
  \draw[fill=white] (gb) circle (.25);
  \draw[fill=white] (ib) circle (.25);
  \draw[fill=white] (kb) circle (.25);
  \path[fill=white] (bb) circle (.25);
  \path[fill=white] (fb) circle (.25);
  \path[fill=white] (jb) circle (.25);
  \node at (fc) {$\ldots$};
  \node at (fb) {$\ldots$};
  \node at (fa) {$\ldots$};
  \node at (kb) {$+$};
  \node at (jb) {$m$};
  \node at (bb) {$0$};
  \node at (jc) {$+$};
  \node at (bc) {$-$};
  \node at (ja) {$-$};
  \node at (ba) {$+$};
  \node at (ab) {$+$};
  \draw [->,>=stealth] (8,3.7)--(4,3.7);
  \node at (6,4) {$\mu$};
  \draw [->,>=stealth] (8,0.3)--(4,0.3);
  \node at (6,0) {$\lambda$};
\end{tikzpicture}}.\]

and the rightmost vertex of $\mathfrak{S}_{\tilde{\lambda}/\tilde{\mu}}$ must be type $\vx{c_1}$, and the leftmost vertex  must be type $\vx{c_2}$, so \begin{align*}Z(\mathfrak{S}^{(1)/(1)}_{\lambda/\mu}) = \frac{Z(\mathfrak{S}_{\tilde{\lambda}/\tilde{\mu}})}{(x+y)A^2B}.\end{align*} Thus, (\ref{boundary-operator-equation}) holds for case D, and the other cases are similar.

Now, for multiple rows, the result follows by induction from the single row case, using the branching rules for both lattice models and Hamiltonians. $\delta_i$ is the number of parts of $\lambda_i$ where $|\lambda_i\rangle$ appears in $|e^{\Phi_{i-1}}\cdots e^{\Phi_0|}\lambda\rangle$.
\end{proof}

\subsection{Computation of the partition function with uniform side boundary conditions} \label{side-boundary-part-fun-section}

Recall that $\rho_\ell = (\ell-1,\ldots,1,0)$. Let $\rho^+_\ell = (\ell,\ldots,2,1)$.

We make the definitions \[L(\beta) := Z(\mathfrak{S}^{\beta/\emptyset}_{\emptyset/\rho_{\ell(\beta)}}), \hspace{20pt} R(\alpha) := Z(\mathfrak{S}^{\emptyset/\alpha}_{\rho_{\ell(\alpha)}/\emptyset}).\]

Given integers $s,t$, let $\lambda,\mu$ be strict partitions with $\ell(\lambda)=\ell,\ell(\mu)=k$, $\ell+s=k+t$. Let $\ell(\alpha)=s,\ell(\beta)=t$, and as usual let $M\ge\max(\lambda_1,\mu_1),N\ge\max(\alpha_1,\beta_1)$. We want to evaluate the partition function $Z(\mathfrak{S}^{\alpha/\beta}_{\lambda/\mu})$.

Let \[\tilde{\lambda} = (M+s+t,M+s+t-1,\ldots,M+t+1,\lambda_1+t,\lambda_2+t,\ldots,\lambda_\ell+t),\] \[\tilde{\mu} = (\mu_1+t,\mu_2+t,\ldots,\mu_k+t,t-1,t-2,\ldots,0).\] Note that if $t=0$, $\tilde{\mu}=\mu$ and if $s=t=0$, then $\tilde{\lambda}=\lambda$.

\begin{proposition} \label{extendedlatticemodelproposition}
\[\prod_{i=1}^N A_i^{M+1+s+t} B_i^{\ell+s}\langle \tilde{\mu}|e^{H_+}|\tilde{\lambda}\rangle = \sum_{\alpha,\beta}  L(\beta) R(\alpha) Z(\mathfrak{S}^{\alpha/\beta}_{\lambda/\mu}),\] where the sum is over all $\alpha,\beta$ where $\alpha_1,\beta_1\le N$ and $\ell(\alpha)=s,\ell(\beta)=t$.
\end{proposition}

\begin{proof}
By Theorem \ref{Delta-lattice-Hamiltonian}, \[\prod_{i=1}^N A_i^{M+1+s+t} B_i^{\ell+s}\langle \tilde{\mu}|e^{H_+}|\tilde{\lambda}\rangle = Z(\mathfrak{S}_{\tilde{\lambda}/\tilde{\mu}}).\] The result follows from breaking the lattice model into three parts: the left $t$ columns, the middle $M$ columns, and the right $s$ columns.
\end{proof}

We will look more closely at the four special cases where $s$ and $t$ are each either 0 or $N$. In these cases, there is only one choice for $\alpha$ and $\beta$, so the sum from Proposition \ref{extendedlatticemodelproposition} disappears. Let $L(N) := L(\rho^+_{N+1}), R(N) := R(\rho^+_{N+1})$. Then, the previous proposition has the following immediate corollary.

\begin{corollary} Let $\ell = \ell(\lambda)$.
\begin{enumerate}
    \item[A)] \[Z(\mathfrak{S}^{\emptyset/\emptyset}_{\lambda/\mu}) = \prod_{i=1}^N A_i^{M+1} B_i^\ell \cdot \langle \mu|e^{H_+}|\lambda\rangle,\]
    \item[B)] \[Z(\mathfrak{S}^{\rho^+_N/\emptyset}_{\lambda/\mu}) = \frac{1}{R(N)}\cdot \prod_{i=1}^N A_i^{M+1+N} B_i^{\ell + N} \cdot \langle\mu|e^{H_+}|\tilde{\lambda}\rangle,\]
    \item[C)] \[Z(\mathfrak{S}^{\emptyset/\rho^+_N}_{\lambda/\mu}) = \frac{1}{L(N)}\cdot \prod_{i=1}^N A_i^{M+1+N} B_i^\ell \cdot \langle \tilde{\mu}|e^{H_+}|\tilde{\lambda}\rangle,\]
    \item[D)] \[Z(\mathfrak{S}^{\rho^+_N/\rho^+_N}_{\lambda/\mu}) = \frac{1}{L(N)R(N)}\cdot \prod_{i=1}^N A_i^{M+1+2N} B_i^{\ell+N} \cdot \langle \tilde{\mu}|e^{H_+}|\tilde{\lambda}\rangle.\]
\end{enumerate}
\end{corollary}

Case A is just Theorem \ref{Delta-lattice-Hamiltonian} restated. To evaluate the other three partition functions, we need to evaluate $L(N)$ and $R(N)$. For the rest of the section, we will use the vertices and the weights from Figure \ref{Delta-vertex-weights} interchangeably.

\begin{lemma}\cite[Lemma~10]{BBF-Schur-polynomials}
\[L(N)\cdot \left[\prod_{i<j}(\vx{a_1^{(j)}}\vx{a_2^{(i)}} + \vx{b_1^{(i)}}\vx{b_2^{(j)}})\right]^{-1}\] is symmetric with respect to the row parameters, and expressible as a polynomial in the variables $\vx{a_1^{(i)}}, \vx{a_2^{(i)}}, \vx{b_1^{(i)}}, \vx{b_2^{(i)}}$ with integer coefficients.
\end{lemma}

\begin{proposition}
\[L(N) = \prod_{k=1}^N \vx{c_2^{(k)}}\cdot \prod_{i<j} \left(\vx{a_1^{(j)}}\vx{a_2^{(i)}} + \vx{b_1^{(i)}}\vx{b_2^{(j)}}\right) = \prod_{k=1}^N A_k^N B_k^{N-k} \cdot \prod_{i<j} (y_i + x_j),\] and \[R(N) = \prod_{k=1}^N \vx{c_1^{(k)}}\cdot \prod_{i<j} \left(\vx{a_1^{(i)}}\vx{a_2^{(j)}} + \vx{b_1^{(j)}}\vx{b_2^{(i)}}\right) = \prod_{k=1}^N A_k^N B_k^{k-1} (x_k+y_k)\cdot \prod_{i<j} (x_i+y_j).\]
\end{proposition}

\begin{proof}
For the first equation, by the previous lemma, $L(N)$ is a multiple of \[\left[\prod_{i<j}(\vx{a_1^{(j)}}\vx{a_2^{(i)}} + \vx{b_1^{(i)}}\vx{b_2^{(j)}})\right]\] as a polynomial in the Boltzmann weights. In addition, each state of $\mathfrak{S}^{\rho^+_N/\emptyset}_{\emptyset/\rho_N}$ must have precisely 1 more $\vx{c_2}$ vertex than $\vx{c_1}$ vertex in each row, so $L(N)$ is also divisible by $\prod_{k=1}^N \vx{c_2^{(k)}}$. Each state of $\mathfrak{S}^{\rho^+_N/\emptyset}_{\emptyset/\rho_N}$ has $N^2$ vertices, so $L(N)$ must have degree $N^2$. The product of the factors we have already determined also has degree $N$, so \[L(N) = \prod_{k=1}^N \vx{c_2^{(k)}}\cdot \prod_{i<j} \left(\vx{a_1^{(j)}}\vx{a_2^{(i)}} + \vx{b_1^{(i)}}\vx{b_2^{(j)}}\right)\] as desired. The second equality in the first equation follows from plugging in the weights from Figure \ref{Delta-vertex-weights}. The second equation follows by reversing the edges in $\mathfrak{S}^{\rho^+_N/\emptyset}_{\emptyset/\rho_N}$ to obtain $\mathfrak{S}^{\emptyset/\rho^+_N}_{\rho_N/\emptyset}$.
\end{proof}

All together, we have the following result, where we have also applied Theorem \ref{Delta-lattice-Hamiltonian}. Case B is the important case of domain-wall boundary conditions.

\begin{theorem} \label{partition-function-all-boundaries} Let $\ell = \ell(\lambda)$.
\begin{enumerate}
    \item[A)] \[Z(\mathfrak{S}^{\emptyset/\emptyset}_{\lambda+\rho/\mu+\rho}) = \prod_{i=1}^{N} (\vx{a_1^{(i)}})^{M+1-\ell}(\vx{b_1^{(i)}})^{\ell} \cdot s_{\lambda/\mu}\left[\left.\frac{\vx{b_2^{(1)}}}{\vx{a_1^{(1)}}},\ldots,\frac{\vx{b_2^{(N)}}}{\vx{a_1^{(N)}}}\right|\frac{\vx{a_2^{(1)}}}{\vx{b_1^{(1)}}},\ldots,\frac{\vx{a_2^{(N)}}}{\vx{b_1^{(N)}}}\right],\]
    \item[B)] \[Z(\mathfrak{S}^{\rho^+_N/\emptyset}_{\lambda+\rho/\mu+\rho}) = \frac{\prod_{i=1}^{N} (\vx{a_1^{(i)}})^{M+1-\ell}(\vx{b_1^{(i)}})^{\ell+N} \cdot s_{\tilde{\lambda}/\mu}\left[\left.\frac{\vx{b_2^{(1)}}}{\vx{a_1^{(1)}}},\ldots,\frac{\vx{b_2^{(N)}}}{\vx{a_1^{(N)}}}\right|\frac{\vx{a_2^{(1)}}}{\vx{b_1^{(1)}}},\ldots,\frac{\vx{a_2^{(N)}}}{\vx{b_1^{(N)}}}\right]}{\prod_{k=1}^N \vx{c_1^{(k)}}\cdot \prod_{i<j} \left(\vx{a_1^{(i)}}\vx{a_2^{(j)}} + \vx{b_1^{(j)}}\vx{b_2^{(i)}}\right)},\]
    \item[C)] \[Z(\mathfrak{S}^{\emptyset/\rho^+_N}_{\lambda+\rho/\mu+\rho}) = \frac{\prod_{i=1}^{N} (\vx{a_1^{(i)}})^{M+1+N-\ell}(\vx{b_1^{(i)}})^{\ell} \cdot s_{\tilde{\lambda}/\tilde{\mu}}\left[\left.\frac{\vx{b_2^{(1)}}}{\vx{a_1^{(1)}}},\ldots,\frac{\vx{b_2^{(N)}}}{\vx{a_1^{(N)}}}\right|\frac{\vx{a_2^{(1)}}}{\vx{b_1^{(1)}}},\ldots,\frac{\vx{a_2^{(N)}}}{\vx{b_1^{(N)}}}\right]}{\prod_{k=1}^N \vx{c_2^{(k)}}\cdot \prod_{i<j} \left(\vx{a_1^{(j)}}\vx{a_2^{(i)}} + \vx{b_1^{(i)}}\vx{b_2^{(j)}}\right)},\]
    \item[D)] \[Z(\mathfrak{S}^{\rho^+_N/\rho^+_N}_{\lambda+\rho/\mu+\rho}) = \frac{\prod_{i=1}^{N} (\vx{a_1^{(i)}})^{M+1+N-\ell}(\vx{b_1^{(i)}})^{\ell+N} \cdot s_{\tilde{\lambda}/\tilde{\mu}}\left[\left.\frac{\vx{b_2^{(1)}}}{\vx{a_1^{(1)}}},\ldots,\frac{\vx{b_2^{(N)}}}{\vx{a_1^{(N)}}}\right|\frac{\vx{a_2^{(1)}}}{\vx{b_1^{(1)}}},\ldots,\frac{\vx{a_2^{(N)}}}{\vx{b_1^{(N)}}}\right]}{\prod_{k=1}^N \vx{c_1^{(k)}}\vx{c_2^{(k)}}\cdot \prod_{i\ne j} \left(\vx{a_1^{(j)}}\vx{a_2^{(i)}} + \vx{b_1^{(i)}}\vx{b_2^{(j)}}\right)}.\]
\end{enumerate}
\end{theorem}

\subsection{Berele-Regev formula and Tokuyama's formula} \label{berele-regev-section}

For this subsection, let $Z_\lambda = Z(\mathfrak{S}_{\emptyset/\lambda+\rho}^{\rho^+_N/\emptyset})$. We will use the Berele-Regev formula to show that $Z_\lambda$ has another expression as a Schur function times a deformed denominator. This corrects the result \cite[Theorem~9]{BBF-Schur-polynomials}, whose proof is circular.

We will use the \emph{supertableau} (or \emph{bitableau}) formula for supersymmetric Schur functions \cite[\S~I.5,~Exercise~23]{Macdonald-symmetric-functions}.

A supertableau of shape $\lambda/\mu$ is a filling of $\lambda/\mu$ with the entries $1,\ldots,N, 1',\ldots, N'$ such that

\begin{itemize}
    \item[(i)] The entries weakly increase across rows and columns under the ordering $1<\ldots<N<1'<\ldots,N'$.
    \item[(ii)] There is at most one $j'$ in every row, and at most one $i$ in every column.
\end{itemize}

Then, \begin{align} s_{\lambda/\mu}[\boldsymbol{x}|\boldsymbol{y}] = \sum_T [\boldsymbol{x}|\boldsymbol{y}]^T, \hspace{30pt} [\boldsymbol{x}|\boldsymbol{y}]^T = \prod_{i=1}^N x_i^{m_i} \prod_{j=1}^N y_j^{n_j}, \label{supertableau-formula}\end{align} where the sum is over all supertableaux $T$ of shape $\lambda/\mu$, and $m_i$ (resp. $n_j$) is the number of entries $i$ (resp. $j'$) in $T$.

\begin{lemma} \label{tableau-rotation-lemma}
Let $\lambda/\mu$ be a skew shape, and let $\gamma/\sigma$ be the skew shape obtained from rotating $\lambda/\mu$ by $180^\circ$. Then $s_{\lambda/\mu}[\boldsymbol{x}|\boldsymbol{y}] = s_{\gamma/\sigma}[\boldsymbol{x}|\boldsymbol{y}]$.
\end{lemma}

\begin{proof}
By \cite[\S~I.5,~Exercise~23(c)]{Macdonald-symmetric-functions}, (\ref{supertableau-formula}) holds if the ordering on $\{1,\ldots,N,1',\ldots,N'\}$ is replaced by any total ordering. It is easy to see that a $180^\circ$ rotation gives a weight-preserving bijection between supertableaux for $\lambda/\mu$ and supertableaux for $\gamma/\sigma$ with the ordering $N'<\ldots<1'<N<\ldots<1$.
\end{proof}

\begin{corollary} \label{domain-wall-corollary}
Let $\tau$ be the partition obtained by taking the complement of $\lambda$ in a $N\times (M+1-N)$ box and rotating $180^\circ$. Then, \begin{align} Z_\lambda = \prod_{i=1}^{N} (\vx{a_1^{(i)}})^{M+1-N} \vx{c_2^{(i)}} \cdot \prod_{i<j} \left(\vx{a_1^{(j)}}\vx{a_2^{(i)}} + \vx{b_1^{(i)}}\vx{b_2^{(j)}}\right)\cdot s_\tau\left[\frac{\vx{b_2^{(1)}}}{\vx{a_1^{(1)}}},\ldots, \frac{\vx{b_2^{(1)}}}{\vx{a_1^{(1)}}}\right] \label{domain-wall-partition-function}\end{align}
\end{corollary}

\begin{proof}

By Theorem \ref{partition-function-all-boundaries}B, plugging in $\emptyset$ for $\lambda$, and $\lambda+\rho$ for $\mu$, \[Z_\lambda = \frac{\prod_{i=1}^N (\vx{a_1^{(i)}})^{M+1}(\vx{b_1^{(i)}})^{N} \cdot s_{\nu/\lambda}\left[\left.\frac{\vx{b_2^{(1)}}}{\vx{a_1^{(1)}}},\ldots,\frac{\vx{b_2^{(N)}}}{\vx{a_1^{(N)}}}\right|\frac{\vx{a_2^{(1)}}}{\vx{b_1^{(1)}}},\ldots,\frac{\vx{a_2^{(N)}}}{\vx{b_1^{(N)}}}\right]}{\prod_{k=1}^N \vx{c_1^{(k)}}\cdot \prod_{i<j} \left(\vx{a_1^{(i)}}\vx{a_2^{(j)}} + \vx{b_1^{(j)}}\vx{b_2^{(i)}}\right)},\] where $\nu$ is the $N\times (M+1)$ block partition: $(M+1,\ldots,M+1)$.

The $180^\circ$ rotation of $\nu/\lambda$ is the (nonskew) partition $\mu := (M+1-\lambda_N, M+1-\lambda_{N-1},\ldots M+1-\lambda_1)$. Notice that all parts of $\mu$ have size $\ge N$ since we must have $M\ge \lambda_1+N-1$ for $\lambda+\rho$ to fit on the top boundary. By Lemma \ref{tableau-rotation-lemma},

\begin{align} Z_\lambda = \frac{\prod_{i=1}^{N} (\vx{a_1^{(i)}})^{M+1}(\vx{b_1^{(i)}})^{N} \cdot s_{\mu}\left[\left.\frac{\vx{b_2^{(1)}}}{\vx{a_1^{(1)}}},\ldots,\frac{\vx{b_2^{(N)}}}{\vx{a_1^{(N)}}}\right|\frac{\vx{a_2^{(1)}}}{\vx{b_1^{(1)}}},\ldots,\frac{\vx{a_2^{(N)}}}{\vx{b_1^{(N)}}}\right]}{\prod_{k=1}^N \vx{c_1^{(k)}}\cdot \prod_{i<j} \left(\vx{a_1^{(i)}}\vx{a_2^{(j)}} + \vx{b_1^{(j)}}\vx{b_2^{(i)}}\right)}. \label{domain-wall-partition-function-intermediate}\end{align}

Next we apply the Berele-Regev formula \cite[Theorem~6.20]{BereleRegev}. $\tau$ is also the partition obtained by subtracting $N$ from every part of $\mu$. Then the Berele-Regev formula says that \[s_\mu\left[x_1,\ldots, x_N\left|y_1,\ldots, y_N\right.\right]  = \prod_{i,j} (x_j + y_i)\cdot s_\tau[x_1,\ldots, x_n].\]

Applying this to (\ref{domain-wall-partition-function-intermediate}) gives \begin{align*} Z_\lambda &= \frac{\prod_{i=1}^{N} (\vx{a_1^{(i)}})^{M+1}(\vx{b_1^{(i)}})^{N} \cdot \prod_{i,j} \left(\frac{\vx{b_2^{(j)}}}{\vx{a_1^{(j)}}} + \frac{\vx{a_2^{(i)}}}{\vx{b_1^{(i)}}}\right)\cdot s_\tau\left[\frac{\vx{b_2^{(1)}}}{\vx{a_1^{(1)}}},\ldots, \frac{\vx{b_2^{(1)}}}{\vx{a_1^{(1)}}}\right]}{\prod_{k=1}^N \vx{c_1^{(k)}}\cdot \prod_{i<j} \left(\vx{a_1^{(i)}}\vx{a_2^{(j)}} + \vx{b_1^{(j)}}\vx{b_2^{(i)}}\right)} \\&= \frac{\prod_{i=1}^{N} (\vx{a_1^{(i)}})^{M+1-N} \cdot \prod_{i,j} \left(\vx{a_1^{(j)}}\vx{a_2^{(i)}} + \vx{b_1^{(i)}}\vx{b_2^{(j)}}\right)\cdot s_\tau\left[\frac{\vx{b_2^{(1)}}}{\vx{a_1^{(1)}}},\ldots, \frac{\vx{b_2^{(1)}}}{\vx{a_1^{(1)}}}\right]}{\prod_{k=1}^N \vx{c_1^{(k)}}\cdot \prod_{i<j} \left(\vx{a_1^{(i)}}\vx{a_2^{(j)}} + \vx{b_1^{(j)}}\vx{b_2^{(i)}}\right)} \\&=  \prod_{i=1}^{N} (\vx{a_1^{(i)}})^{M+1-N} \vx{c_2^{(i)}} \cdot \prod_{i<j} \left(\vx{a_1^{(j)}}\vx{a_2^{(i)}} + \vx{b_1^{(i)}}\vx{b_2^{(j)}}\right)\cdot s_\tau\left[\frac{\vx{b_2^{(1)}}}{\vx{a_1^{(1)}}},\ldots, \frac{\vx{b_2^{(1)}}}{\vx{a_1^{(1)}}}\right], \end{align*} where the last equality uses the free fermion condition. Note that this is a polynomial in the Boltzmann weights of degree $(M+1)\cdot N$.
\end{proof}

A very similar result was stated by Brubaker, Bump, and Friedberg \cite[Theorem 9]{BBF-Schur-polynomials} as a generalization of Tokuyama's formula \cite{Tokuyama}; they give exactly the formula (\ref{domain-wall-partition-function}) except that their power of $\vx{a_1^{(i)}}$ is different. The proof in \cite{BBF-Schur-polynomials} is circular: they observe that after normalizing $\vx{c_2^{(i)}}=1$, using the free fermion condition, $Z_\lambda$ can be expressed as a polynomial in $\vx{a_1^{(i)}}, \vx{a_2^{(i)}}, \vx{b_1^{(i)}}, \vx{b_2^{(i)}}$. However, they then consider only the states without vertices of type $\vx{c_1^{(i)}}$, rather than all admissible states. Despite the fact that $\vx{c_1^{(i)}}$ can be removed algebraically from an expression of the partition function, it is still necessary to consider states involving type $\vx{c_1^{(i)}}$ vertices. Corollary \ref{domain-wall-corollary} is therefore a correction of that proof. The fact that the formula given in \cite{BBF-Schur-polynomials} is so close to correct suggests that maybe their proof technique can be salvaged.

In \cite{BBF-Schur-polynomials-arxiv}, the ArXiv version of \cite{BBF-Schur-polynomials}, Brubaker, Bump, and Friedberg give a correct proof of Tokuyama's Theorem by two different evaluations of two free fermionic six-vertex models they call $\Gamma$ and $\Delta$ ice. One side of Tokuyama's formula is given as a sum over Gelfand-Tsetlin patterns, while the other is a Schur function times a deformed denominator.

It is easy to see that is the case of these weights, Corollary \ref{domain-wall-corollary} specializes to Tokuyama's formula. In the case of the $\Gamma$ weights in \cite[Table~1]{BBF-Schur-polynomials-arxiv}, Corollary \ref{domain-wall-corollary} yields \[Z_\lambda = \prod_{i<j} (t_iz_j + z_i) s_\tau[z_1,\ldots,z_n],\] which is \cite[Theorem~5]{BBF-Schur-polynomials-arxiv}.

Bump, McNamara, and Nakasuji \cite{Bump-McNamara-Nakasuji} give a Tokuyama formula for factorial Schur functions. Motegi \cite{Motegi-free-fermion} then generalizes this, and considers a version of this problem using Izergin-Korepin analysis. Motegi's result benefits from consideration of colulmn parameters, while ours is slightly more general when considering only row parameters (thereby allowing our corrective to the formula in \cite{BBF-Schur-polynomials}). In addition, our proof of Tokuyama's formula using Hamiltonian operators is very different from other known proofs.

Aggarwal, Borodin, Petrov, and Wheeler consider the Berele-Regev formula from the opposite direction. Using their lattice models and \cite[Theorem~9]{BBF-Schur-polynomials}, they prove a generalized Berele-Regev formula \cite[Corollary~4.13]{AggarwalBorodinPetrovWheeler}.

This completes our analysis of the classical six-vertex model and classical Fock space.

\section{Metaplectic Fock spaces} \label{q-Fock-space-section}

We now turn our attention to the six-vertex model with charge, which will turn out to match Hamiltonian operators arising from Drinfeld twists of $q$-Fock space.

$q$-Fock space \cite{Kashiwara-Miwa-Stern} is a quantum analogue of the classical Fock space defined in Section \ref{background-section}. It is formed as a quotient by a Hecke algebra action of the infinite tensor power of the standard evaluation module of the quantum group $U_q(\widehat{\mathfrak{sl}_n})$.

We will work with a family of related spaces defined by Brubaker, Buciumas, Bump, and Gustafsson \cite{BBBG-Hamiltonian}. Instead of being a module for $U_q(\widehat{\mathfrak{sl}_n})$, these spaces are modules for Drinfeld twists of $U_q(\widehat{\mathfrak{sl}_n})$. Reshetikhin \cite{Reshetikhin-twist} defined a large class of Drinfeld twists of quantum groups, and Brubaker, Buciumas, Bump, and Gustafsson applied these twists to $q$-Fock space. Drinfeld twisting doesn't affect the algebra structure of a quantum group, but it does affect the coalgebra structure, so the result is a set of genuinely distinct modules, which we will now describe.

Choose an integer $n\ge 1$. Let \[F = \exp\left(\sum_{1\le i<j\le n} a_{ij}(H_i\otimes H_j - H_j\otimes H_i)\right),\] where the $H_i$ are certain generators of the topological Hopf algebra associated to $U_q(\widehat{\mathfrak{sl}_n})$ (see \cite{Reshetikhin-twist}). If $1\le i,j\le n$, set \[\alpha_{ij} = \begin{cases} \exp(2a_{i,j} - 2a_{i-1,j} - 2a_{i,j-1} + 2a_{i-1,j-1}), & i\ne j \\ 1, & i=j,\end{cases}\] and let $\alpha_{ij}$ be defined modulo $n$ so that $\alpha_{i+kn,j+mn} = \alpha_{ij}$.

These $\alpha_{ij}$ determine the relations in our twists of $q$-Fock space. There is no Clifford algebra structure associated to $q$-Fock space, so we will define it using wedges. The twisted $q$-Fock space $\mathcal{F} := \mathcal{F}(F)$ is the space with basis \[u_{m_1}\wedge u_{m_2}\wedge u_{m_3}\wedge \ldots, \hspace{30pt} (m_1>m_2>\ldots).\] Wedges with decreasing index like these are called \emph{normally ordered}.

Additionally, $\mathcal{F}$ has the following \emph{normal ordering relations}. \[u_l\wedge u_m = \begin{cases} -u_m\wedge u_l, & l\equiv m\mod n\\ -q\alpha_{lm} u_m\wedge u_l + (q^2-1)(u_{m-i}\wedge u_{l+i} + -q\alpha_{lm} u_{m-n}\wedge u_{l+n} + \\\hspace{40pt}q^2 u_{m-n-i}\wedge u_{l+n+i} - q^3 \alpha_{lm} u_{m-2n}\wedge u_{l+2n} + \ldots), & \text{otherwise,}\end{cases}\] where $l\le m$. In particular, $u_m\wedge u_m = 0$. Here, $i$ satisfies $0<i<n, m-i\equiv l \mod n$. The sum in the second case continues while the wedges are normally ordered. Notice that these relations depend only on the numbers $\alpha_{ij}$.

We will say a Drinfeld twist of $q$-Fock space is \emph{shift invariant} if $\alpha_{ij}$ depends only on $i-j$. In this case, we define a function $g$ on integers modulo $n$ let \[g(i-j) := \begin{cases} -q\alpha_{ij}, & i\ne j \mod n \\ -q^2, & i=j \mod n.\end{cases}\] We call the integers modulo $n$ \emph{charge}, and will often take as representatives the integers $0,1,\ldots,n-1$. For convenience, let $v=q^2$. The relation $\alpha_{ij}\alpha_{ji}=1$ becomes \[g(a)g(-a) = -g(0), \hspace{40pt} \text{for all } a\not\equiv 0\mod n.\]

Under the assumption of shift invariance, the wedge relations become: \[u_l\wedge u_m = \begin{cases} -u_m\wedge u_l, & l\equiv m\mod n\\ g(l-m) u_m\wedge u_l + (q^2-1)(u_{m-i}\wedge u_{l+i} + g(l-m) u_{m-n}\wedge u_{l+n} + \\\hspace{40pt}q^2 u_{m-i}\wedge u_{l+i} + q^2 g(l-m) u_{m-n}\wedge u_{l+n} + \ldots), & \text{otherwise}.\end{cases}\] Notice that these relations depend only on the function $g$. As a result, we will often view $g$ as being synonymous with $\mathcal{F}$.

The shift invariant Drinfeld-Reshetikhin twists of $q$-Fock space were shown in \cite{BBBG-Hamiltonian} to be related to lattice models for metaplectic Whittaker functions, so we will call these spaces \emph{metaplectic Fock spaces}. They will also turn out to be closely related to the solvability of charged models. The shift invariance property is a natural one to require; as we will see in the next section, charged models have a sort of shift invariance themselves, so it doesn't make sense to compare non-shift-invariant spaces with charged lattice models. 

For a strict partition $\lambda$, we define \[|\lambda\rangle := v_{\lambda_1}\wedge v_{\lambda_2}\wedge\ldots\wedge v_{\lambda_{\ell(\lambda)}}\wedge v_{-1}\wedge v_{-2}\wedge\ldots.\] Define the dual basis $\{\langle\lambda|\}$ of $\mathcal{F}^*$ by the pairing \[\langle\mu|\lambda\rangle = \begin{cases} 1, & \text{if $\lambda=\mu$} \\ 0, &\text{otherwise}.\end{cases}\]

There is a Heisenberg action \[J_k\cdot (u_{m_1}\wedge u_{m_2}\wedge\ldots) = \sum_{i\ge 0} (u_{m_1}\wedge\ldots \wedge u_{m_{i-1}} \wedge u_{m_i-kn} \wedge u_{m_{i+1}}\wedge\ldots).\]

It has been shown \cite{Kashiwara-Miwa-Stern, BBBG-Hamiltonian} that \[[J_k,J_l] = k\cdot \frac{1-v^{n|k|}}{1-v^{|k|}} \delta_{k,-l}.\]

Once again, we fix parameters $s_k^{(j)}$, $k\ne 0, 1\le j\le N$ and set \[s_k := \sum_{j=1}^N s_k^{(j)},\hspace{20pt} H_\pm = \sum_{k\ge 1} s_{\pm k} J_{\pm k}, \hspace{40pt} e^{H_\pm} = \sum_{m\ge 0} \frac{1}{m!} H^m.\]

\section{Six-vertex models with charge} \label{charged-models-section}

The lattice models in this section are similar to the six-vertex model, but use an extra statistic called \emph{charge}, an integer modulo $n$ associated to each horizontal $-$ spin. We call the resulting combination of spin and charge (or simply a $+$ spin) a \emph{decorated spin}. As first seen in \cite{BBBG-Hamiltonian}, the charge statistic turns out to be the right way to represent the $q$-Fock space for $U_q(\widehat{\mathfrak{sl}}_n)$, along with the metaplectic data from its Drinfeld twists.

Informally, the use of charge forces particles to travel a multiple of $n$ columns in each row, mirroring the action of the current operators in the previous section. Therefore, the $\vx{a_2}$ and $\vx{b_2}$ vertices in charged models depend on the charges on their horizontal edges, while the other four vertices either have $+$ spins on their horizontal edges or are restricted to specified charges on their horizontal $-$ edges. This results in a total of $2n+4$ vertices, which makes the analysis both by the Yang-Baxter equation and by Hamiltonians more challenging than for the classical six-vertex model. Note that if $n=1$, this model reduces to that one.

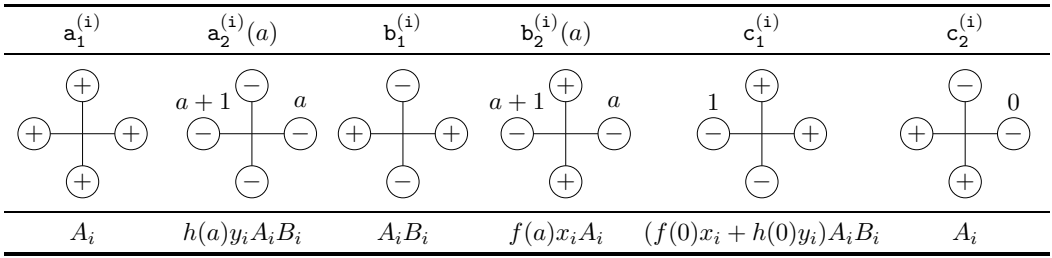
\begin{figure}[h]
\begin{center}
\scalebox{.85}{$
\begin{array}{c@{\hspace{8pt}}c@{\hspace{8pt}}c@{\hspace{5pt}}c@{\hspace{5pt}}c@{\hspace{8pt}}c@{\hspace{8pt}}c@{\hspace{5pt}}c}
\toprule
\vx{a_1^{(i)}}&\vx{a_2^{(i)}}(a)&\vx{b_1^{(i)}} & \vx{b_2^{(i)}}(a) & \vx{c_1^{(i)}}&\vx{c_2^{(i)}}\\
\midrule
\begin{tikzpicture}
\coordinate (a) at (-.75, 0);
\coordinate (b) at (0, .75);
\coordinate (c) at (.75, 0);
\coordinate (d) at (0, -.75);
\coordinate (aa) at (-.75,.5);
\coordinate (cc) at (.75,.5);
\draw (a)--(0,0);
\draw (b)--(0,0);
\draw (c)--(0,0);
\draw (d)--(0,0);
\draw[fill=white] (a) circle (.25);
\draw[fill=white] (b) circle (.25);
\draw[fill=white] (c) circle (.25);
\draw[fill=white] (d) circle (.25);
\node at (0,1) { };
\node at (a) {$+$};
\node at (b) {$+$};
\node at (c) {$+$};
\node at (d) {$+$};
\end{tikzpicture}
&
\begin{tikzpicture}
\coordinate (a) at (-.75, 0);
\coordinate (b) at (0, .75);
\coordinate (c) at (.75, 0);
\coordinate (d) at (0, -.75);
\coordinate (aa) at (-.75,.5);
\coordinate (cc) at (.75,.5);
\draw (a)--(0,0);
\draw (b)--(0,0);
\draw (c)--(0,0);
\draw (d)--(0,0);
\draw[fill=white] (a) circle (.25);
\draw[fill=white] (b) circle (.25);
\draw[fill=white] (c) circle (.25);
\draw[fill=white] (d) circle (.25);
\node at (0,1) { };
\node at (a) {$-$};
\node at (b) {$-$};
\node at (c) {$-$};
\node at (d) {$-$};
\node at (aa) {$a+1$};
\node at (cc) {$a$};
\end{tikzpicture}
&
\begin{tikzpicture}
\coordinate (a) at (-.75, 0);
\coordinate (b) at (0, .75);
\coordinate (c) at (.75, 0);
\coordinate (d) at (0, -.75);
\coordinate (aa) at (-.75,.5);
\coordinate (cc) at (.75,.5);
\draw (a)--(0,0);
\draw (b)--(0,0);
\draw (c)--(0,0);
\draw (d)--(0,0);
\draw[fill=white] (a) circle (.25);
\draw[fill=white] (b) circle (.25);
\draw[fill=white] (c) circle (.25);
\draw[fill=white] (d) circle (.25);
\node at (0,1) { };
\node at (a) {$+$};
\node at (b) {$-$};
\node at (c) {$+$};
\node at (d) {$-$};
\end{tikzpicture}
&
\begin{tikzpicture}
\coordinate (a) at (-.75, 0);
\coordinate (b) at (0, .75);
\coordinate (c) at (.75, 0);
\coordinate (d) at (0, -.75);
\coordinate (aa) at (-.75,.5);
\coordinate (cc) at (.75,.5);
\draw (a)--(0,0);
\draw (b)--(0,0);
\draw (c)--(0,0);
\draw (d)--(0,0);
\draw[fill=white] (a) circle (.25);
\draw[fill=white] (b) circle (.25);
\draw[fill=white] (c) circle (.25);
\draw[fill=white] (d) circle (.25);
\node at (0,1) { };
\node at (a) {$-$};
\node at (b) {$+$};
\node at (c) {$-$};
\node at (d) {$+$};
\node at (aa) {$a+1$};
\node at (cc) {$a$};
\end{tikzpicture}
&
\begin{tikzpicture}
\coordinate (a) at (-.75, 0);
\coordinate (b) at (0, .75);
\coordinate (c) at (.75, 0);
\coordinate (d) at (0, -.75);
\coordinate (aa) at (-.75,.5);
\coordinate (cc) at (.75,.5);
\draw (a)--(0,0);
\draw (b)--(0,0);
\draw (c)--(0,0);
\draw (d)--(0,0);
\draw[fill=white] (a) circle (.25);
\draw[fill=white] (b) circle (.25);
\draw[fill=white] (c) circle (.25);
\draw[fill=white] (d) circle (.25);
\node at (0,1) { };
\node at (a) {$-$};
\node at (b) {$+$};
\node at (c) {$+$};
\node at (d) {$-$};
\node at (aa) {$1$};
\end{tikzpicture}
&
\begin{tikzpicture}
\coordinate (a) at (-.75, 0);
\coordinate (b) at (0, .75);
\coordinate (c) at (.75, 0);
\coordinate (d) at (0, -.75);
\coordinate (aa) at (-.75,.5);
\coordinate (cc) at (.75,.5);
\draw (a)--(0,0);
\draw (b)--(0,0);
\draw (c)--(0,0);
\draw (d)--(0,0);
\draw[fill=white] (a) circle (.25);
\draw[fill=white] (b) circle (.25);
\draw[fill=white] (c) circle (.25);
\draw[fill=white] (d) circle (.25);
\node at (0,1) { };
\node at (a) {$+$};
\node at (b) {$-$};
\node at (c) {$-$};
\node at (d) {$+$};
\node at (cc) {$0$};
\end{tikzpicture}
\\
   \midrule
   A_i & h(a)y_iA_iB_i & A_iB_i & f(a)x_iA_i & (f(0)x_i+h(0)y_i)A_iB_i & A_i \\
   \bottomrule
\end{array}$}
\end{center}
\caption{The Boltzmann weights for $\mathfrak{S}$ with charge. Here, $x_i, y_i, A_i$, and $B_i$ are parameters associated to each row, while $f(a)$ and $h(a)$ depend only on the change $a$.}
    \label{Delta-charged-vertex-weights}
\end{figure}

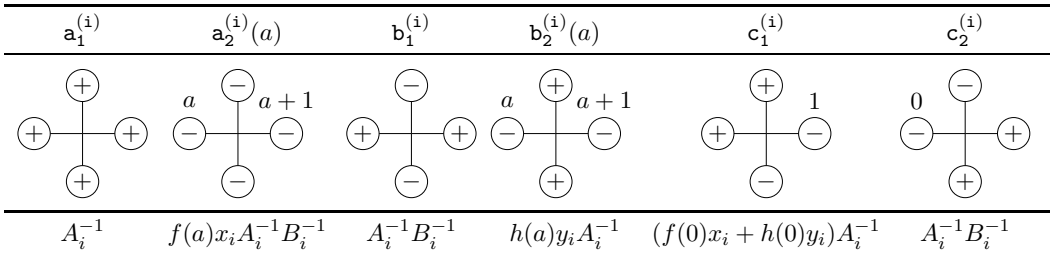
\begin{figure}[h]
\centering
\scalebox{.85}{$
\begin{array}{c@{\hspace{8pt}}c@{\hspace{8pt}}c@{\hspace{5pt}}c@{\hspace{5pt}}c@{\hspace{8pt}}c@{\hspace{8pt}}c@{\hspace{5pt}}c}
\toprule
\vx{a_1^{(i)}}&\vx{a_2^{(i)}}(a)&\vx{b_1^{(i)}} & \vx{b_2^{(i)}}(a) & \vx{c_1^{(i)}}&\vx{c_2^{(i)}}\\
\midrule
\begin{tikzpicture}
\coordinate (a) at (-.75, 0);
\coordinate (b) at (0, .75);
\coordinate (c) at (.75, 0);
\coordinate (d) at (0, -.75);
\coordinate (aa) at (-.75,.5);
\coordinate (cc) at (.75,.5);
\draw (a)--(0,0);
\draw (b)--(0,0);
\draw (c)--(0,0);
\draw (d)--(0,0);
\draw[fill=white] (a) circle (.25);
\draw[fill=white] (b) circle (.25);
\draw[fill=white] (c) circle (.25);
\draw[fill=white] (d) circle (.25);
\node at (0,1) { };
\node at (a) {$+$};
\node at (b) {$+$};
\node at (c) {$+$};
\node at (d) {$+$};
\end{tikzpicture}
&
\begin{tikzpicture}
\coordinate (a) at (-.75, 0);
\coordinate (b) at (0, .75);
\coordinate (c) at (.75, 0);
\coordinate (d) at (0, -.75);
\coordinate (aa) at (-.75,.5);
\coordinate (cc) at (.75,.5);
\draw (a)--(0,0);
\draw (b)--(0,0);
\draw (c)--(0,0);
\draw (d)--(0,0);
\draw[fill=white] (a) circle (.25);
\draw[fill=white] (b) circle (.25);
\draw[fill=white] (c) circle (.25);
\draw[fill=white] (d) circle (.25);
\node at (0,1) { };
\node at (a) {$-$};
\node at (b) {$-$};
\node at (c) {$-$};
\node at (d) {$-$};
\node at (aa) {$a$};
\node at (cc) {$a+1$};
\end{tikzpicture}
&
\begin{tikzpicture}
\coordinate (a) at (-.75, 0);
\coordinate (b) at (0, .75);
\coordinate (c) at (.75, 0);
\coordinate (d) at (0, -.75);
\coordinate (aa) at (-.75,.5);
\coordinate (cc) at (.75,.5);
\draw (a)--(0,0);
\draw (b)--(0,0);
\draw (c)--(0,0);
\draw (d)--(0,0);
\draw[fill=white] (a) circle (.25);
\draw[fill=white] (b) circle (.25);
\draw[fill=white] (c) circle (.25);
\draw[fill=white] (d) circle (.25);
\node at (0,1) { };
\node at (a) {$+$};
\node at (b) {$-$};
\node at (c) {$+$};
\node at (d) {$-$};
\end{tikzpicture}
&
\begin{tikzpicture}
\coordinate (a) at (-.75, 0);
\coordinate (b) at (0, .75);
\coordinate (c) at (.75, 0);
\coordinate (d) at (0, -.75);
\coordinate (aa) at (-.75,.5);
\coordinate (cc) at (.75,.5);
\draw (a)--(0,0);
\draw (b)--(0,0);
\draw (c)--(0,0);
\draw (d)--(0,0);
\draw[fill=white] (a) circle (.25);
\draw[fill=white] (b) circle (.25);
\draw[fill=white] (c) circle (.25);
\draw[fill=white] (d) circle (.25);
\node at (0,1) { };
\node at (a) {$-$};
\node at (b) {$+$};
\node at (c) {$-$};
\node at (d) {$+$};
\node at (aa) {$a$};
\node at (cc) {$a+1$};
\end{tikzpicture}
&
\begin{tikzpicture}
\coordinate (a) at (-.75, 0);
\coordinate (b) at (0, .75);
\coordinate (c) at (.75, 0);
\coordinate (d) at (0, -.75);
\coordinate (aa) at (-.75,.5);
\coordinate (cc) at (.75,.5);
\draw (a)--(0,0);
\draw (b)--(0,0);
\draw (c)--(0,0);
\draw (d)--(0,0);
\draw[fill=white] (a) circle (.25);
\draw[fill=white] (b) circle (.25);
\draw[fill=white] (c) circle (.25);
\draw[fill=white] (d) circle (.25);
\node at (0,1) { };
\node at (a) {$+$};
\node at (b) {$+$};
\node at (c) {$-$};
\node at (d) {$-$};
\node at (cc) {$1$};
\end{tikzpicture}
&
\begin{tikzpicture}
\coordinate (a) at (-.75, 0);
\coordinate (b) at (0, .75);
\coordinate (c) at (.75, 0);
\coordinate (d) at (0, -.75);
\coordinate (aa) at (-.75,.5);
\coordinate (cc) at (.75,.5);
\draw (a)--(0,0);
\draw (b)--(0,0);
\draw (c)--(0,0);
\draw (d)--(0,0);
\draw[fill=white] (a) circle (.25);
\draw[fill=white] (b) circle (.25);
\draw[fill=white] (c) circle (.25);
\draw[fill=white] (d) circle (.25);
\node at (0,1) { };
\node at (a) {$-$};
\node at (b) {$-$};
\node at (c) {$+$};
\node at (d) {$+$};
\node at (aa) {$0$};
\end{tikzpicture}
\\
   \midrule
   A_i^{-1} & f(a)x_iA_i^{-1}B_i^{-1} & A_i^{-1}B_i^{-1} & h(a)y_iA_i^{-1} & (f(0)x_i+h(0)y_i)A_i^{-1} & A_i^{-1}B_i^{-1} \\
   \bottomrule
\end{array}$}
\caption{The Boltzmann weights for $\mathfrak{S}^*$ with charge. Here, $z_i, w_i, A_i$, and $B_i$ are parameters associated to each row, while $f(a)$ and $h(a)$ depend only on the change $a$.}
    \label{Gamma-charged-vertex-weights}
\end{figure}

Define the charged model $\mathfrak{S}^q_{\lambda/\mu} := \mathfrak{S}^q_{\lambda/\mu}(\boldsymbol{x}, \boldsymbol{y}, \boldsymbol{A}, \boldsymbol{B}, f, h)$ to be as follows:
\begin{itemize}
    \item $N$ rows, labelled $1,\ldots, N$ from bottom to top;
    \item $M+1$ columns, where $M\ge \max(\lambda_1,\mu_1)$, labelled $0,\ldots,M$ from left to right;
    \item Left and right boundary edges all $+$;
    \item Bottom boundary edges $-$ on parts of $\lambda$; $+$ otherwise;
    \item Top boundary edges $-$ on parts of $\mu$; $+$ otherwise.
    \item Boltzmann weights from Figure \ref{Delta-charged-vertex-weights},
\end{itemize}

and the charged model $\mathfrak{S}^{*,q}_{\lambda/\mu} := \mathfrak{S}^q_{\lambda/\mu}(\boldsymbol{x}, \boldsymbol{y}, \boldsymbol{A}, \boldsymbol{B}, f, h)$ is defined similarly:
\begin{itemize}
    \item $N$ rows, labelled $1,\ldots, N$ from top to bottom;
    \item $M+1$ columns, where $M\ge \max(\lambda_1,\mu_1)$, labelled $0,\ldots,M$ from left to right;
    \item Left and right boundary edges all $+$;
    \item Bottom boundary edges $-$ on parts of $\mu$; $+$ otherwise;
    \item Top boundary edges $-$ on parts of $\lambda$; $+$ otherwise.
    \item Boltzmann weights from Figure \ref{Gamma-charged-vertex-weights}.
\end{itemize}

Notice that in the case $n=1$, these models become $\mathfrak{S}$ and $\mathfrak{S}^*$ from Section \ref{six-vertex-section}, up to a scaling of the parameters $x_i,y_i$. As in that case, there are two ways to transform $\mathfrak{S}^q_{\lambda/\mu}$ into $\mathfrak{S}^{*,q}_{\lambda/\mu}$ by manipulating the model. First,
\begin{itemize}
    \item Rotate the model $\mathfrak{S}^q_{\lambda/\mu}(\boldsymbol{x}, \boldsymbol{y}, \boldsymbol{A}, \boldsymbol{B}, f, h)$ $180^\circ$.
    \item Flip the vertical spins.
    \item Reverse the ordering on the columns.
    \item Divide the Boltzmann weights by $A_i^2B_i$.
\end{itemize}

This results in the model $\mathfrak{S}^{*,q}_{\ba{\lambda}/\ba{\mu}}(\boldsymbol{x}, \boldsymbol{y}, \boldsymbol{A}, \boldsymbol{B}, f, h)$,so we have the following relationship between partition functions.

\begin{proposition} \label{charged-first-Gamma-Delta-relationship}
\begin{align*}Z(\mathfrak{S}^{*,q}_{\ba{\lambda}/\ba{\mu}} (\boldsymbol{x}, \boldsymbol{y}, \boldsymbol{A}, \boldsymbol{B}, f, h)) = \prod_{i=1}^N (A_i^{-2M-2} B_i^{-M-1}) \cdot Z( \mathfrak{S}^q_{\lambda/\mu}(\boldsymbol{x}, \boldsymbol{y}, \boldsymbol{A}, \boldsymbol{B}, f, h)).\end{align*}
\end{proposition}

Second,
\begin{itemize}
    \item Flip $\mathfrak{S}^q_{\lambda/\mu}(\boldsymbol{x}, \boldsymbol{y}, \boldsymbol{A}, \boldsymbol{B}, f, h)$ vertically (over a horizontal axis).
    \item Replace each charge $a$ with the representative modulo $n$ of $n-a+1$ in the range $[0,n-1]$.
    \item Swap the $\vx{c_1}$ and $\vx{c_2}$ vertices.
    \item Replace $A_i$ with $A_i^{-1}$ and $B_i$ with $B_i^{-1}$.
    \item Swap $x_i$ and $y_i$.
    \item Replace $f(a)$ with $h(-a)$ and $h(a)$ with $f(-a)$;
    \item Rebalance the $\vx{c_1}$ and $\vx{c_2}$ vertices by multiplying the former and dividing the latter by $f(0)z_i+h(0)w_i$.
\end{itemize}

The resulting model is $\mathfrak{S}^{*,q}_{\lambda/\mu}(\boldsymbol{x}, \boldsymbol{y}, \boldsymbol{A}, \boldsymbol{B}, f, h)$. Let $\ba{f}(a) = f(a)$, and similarly for $\ba{h}$. Then we have

\begin{proposition} \label{charged-second-Gamma-Delta-relationship}
\begin{align*}Z(\mathfrak{S}^{*,q}_{\lambda/\mu}(\boldsymbol{x}, \boldsymbol{y}, \boldsymbol{A}, \boldsymbol{B}, f, h)) = Z(\mathfrak{S}^q_{\lambda/\mu}(\boldsymbol{y}, \boldsymbol{x}, \boldsymbol{A}^{-1}, \boldsymbol{B}^{-1}, \ba{h}, \ba{f})).\end{align*}
\end{proposition}

We can now combine these identities to relate the partition functions of $\mathfrak{S}^{*,q}$ and $\mathfrak{S}^q$ to themselves.

\begin{proposition} \label{charged-Functional-equations-proposition}
\leavevmode
\begin{enumerate}
    \item[(a)] \begin{align*} Z(\mathfrak{S}^{*,q}_{\ba{\lambda}/\ba{\mu}}(\boldsymbol{x}, \boldsymbol{y}, \boldsymbol{A}, \boldsymbol{B}, f, h)) = \prod_{i=1}^N (A_i^{-2M-2} B_i^{-M-1}) \cdot Z(\mathfrak{S}^q_{\lambda/\mu}(\boldsymbol{y}, \boldsymbol{x}, \boldsymbol{A}^{-1}, \boldsymbol{B}^{-1}, \ba{h}, \ba{f})).\end{align*}
    \item[(b)] \begin{align*} Z(\mathfrak{S}^q_{\ba{\lambda}/\ba{\mu}}(\boldsymbol{x}, \boldsymbol{y}, \boldsymbol{A}, \boldsymbol{B}, f, h)) = \prod_{i=1}^N (A_i^{2M+2} B_i^{M+1}) \cdot Z(\mathfrak{S}^{*,q}_{\lambda/\mu}(\boldsymbol{y}, \boldsymbol{x}, \boldsymbol{A}^{-1}, \boldsymbol{B}^{-1}, \ba{h}, \ba{f})).\end{align*}
\end{enumerate}
\end{proposition}

Similar identities hold when the weights are left arbitrary.

We'll also need the row transfer matrices for both models. Define: \[\langle\mu|T^q|\lambda\rangle := Z(\mathfrak{S}^q_{\lambda/\mu}), \hspace{20pt} \langle\lambda|T^{*,q}|\mu\rangle := Z(\mathfrak{S}^{*,q}_{\lambda/\mu}),\] where for both lattice models, $N=1$ and $M\ge \max(\lambda_1,\mu_1)$.

Let $\mathcal{F}$ be a metaplectic Fock space as in the previous subsection, and let $g$ be the associated function modulo $n$. We say the Boltzmann weights for a lattice model with charge (either $\mathfrak{S}^q$ or $\mathfrak{S}^{*,q}$ above) satisfy the \emph{$\mathcal{F}$-free fermion condition} if the following two conditions hold:

\begin{itemize}
    \item Zero-charge free fermion condition: \begin{align}\vx{a_1}\vx{a_2}(0) + \vx{b_1}\vx{b_2}(0) - \vx{c_1}\vx{c_2} = 0, \label{zero-charge-free-fermion-condition}\end{align}
    \item $\mathcal{F}$-charge condition: for any $0\le a\le n-1$, \begin{align}\frac{\vx{a_1^{(i)}}\vx{a_2^{(i)}}(a)}{\vx{b_1^{(i)}}\vx{b_2^{(i)}}(a)} = g(a).\label{F-charge-condition}\end{align}
\end{itemize}

We will say that a set of Boltzmann weights satisfies the \emph{generalized free fermion condition} if it satisfies the \emph{$\mathcal{F}$-free fermion condition} for any metaplectic Fock space $\mathcal{F}$. The $\mathcal{F}$-charge condition involves values of $g$, which depend on $\mathcal{F}$, so we need an equivalent expression that is independent of $g$. To get this, we use the relation $g(a)g(-a) = -g(0)$, and use the $\mathcal{F}$-charge condition to replace each factor with a ratio of weights. Doing this, we arrive with the following conditions.

\begin{itemize}
    \item Zero charge free fermion condition: \[\vx{a_1}\vx{a_2}(0) + \vx{b_1}\vx{b_2}(0) - \vx{c_1}\vx{c_2} = 0,\]
    \item Charge condition: for any $1\le a\le n-1$, \begin{align}\frac{\vx{a_1^{(i)}}\vx{a_2^{(i)}}(a)\vx{a_2^{(i)}}(-a)}{\vx{b_1^{(i)}}\vx{b_2^{(i)}}(a)\vx{b_2^{(i)}}(-a)} = -\frac{\vx{a_2^{(i)}}(0)}{\vx{b_2^{(i)}}(0)}.\label{charge-condition}\end{align}
\end{itemize}

If our weights satisfy the generalized free fermion condition, we can determine $g$ (and therefore $\mathcal{F}$) from (\ref{F-charge-condition}). If $\mathfrak{S}$ is a lattice model with charge, we will denote the corresponding metaplectic Fock space $\mathcal{F}(\mathfrak{S})$. Note that if $n=1$, the charge condition goes away, and the generalized free fermion condition reduces to the usual free fermion condition. In this case, $\mathcal{F}(\mathfrak{S})$ is the classical Fock space of Section \ref{background-section}.

By construction, both the $\mathfrak{S}$ and $\mathfrak{S}^*$ weights above satisfy the free fermion condition. The left side of the charge condition must be independent of $a$. Using the parameters in the $\mathfrak{S}$ weights above, it can be expressed more simply: \[\frac{h(a)h(-a)}{f(a)f(-a)} = -\frac{h(0)}{f(0)}.\]

Notice that the $\mathfrak{S}$ and $\mathfrak{S}^*$ weights above are almost completely general, aside from satisfying the zero-charge free fermion condition. Similarly to the $n=1$ case, we could change the relative weights of $\vx{c_1}$ and $\vx{c_2}$, which would create a model that is only trivially different from the original. The last piece is that we could have let the functions $f(a)$ and $h(a)$ depend on their row. However, the ratio $\frac{h(a)}{f(a)}$ for any $a$ needs to be independent of the row in order for the model to match a Hamiltonian of a single metaplectic Fock space, so we simplify notation by having both $f$ and $h$ be independent of $i$.

\begin{remark}
Six-vertex models with charge can be considered as a subset of the colored models studied in \cite{BBBG-Iwahori}. In that context $\vx{a_2}$ vertices depend on two parameters, one for each color. The charge in our models corresponds to the difference between the colors. Our models have a shift invariance property: adding an integer uniformly to every column index does not change any of the Boltzmann weights in Figures \ref{Delta-charged-vertex-weights} and \ref{Gamma-charged-vertex-weights}, whereas the same operation would change the Boltzmann weights in \cite[Figure~7]{BBBG-Iwahori}. In this light, shift invariance is a natural and necessary condition for our Fock spaces. The more general colored models may not need this condition, and will be the subject of future work.
\end{remark}

The following is our main result relating Hamiltonian operators to charged lattice models. We will prove it in Section \ref{proof-of-quantum-main-result-section}.

\begin{align} Z(\mathfrak{S}^q_{\lambda/\mu}) = \prod_{i=1}^{N} A_i^{M+1} B_i^{\ell(\lambda)} \cdot\langle\mu|e^{H_+}|\lambda\rangle \hspace{20pt} \text{for all strict partitions $\lambda,\mu$ and all $M,N$.} \label{charged-Delta-Lattice-Hamiltonian-correspondence-equation}\end{align}

\begin{theorem} \label{charged-Delta-lattice-Hamiltonian}
\leavevmode
\begin{enumerate}
    \item[(a)] (\ref{charged-Delta-Lattice-Hamiltonian-correspondence-equation}) holds precisely when the weights of $\mathfrak{S}^q$ satisfy the generalized free fermion condition and for all $k\ge 1, j\in [1,N]$, \begin{align} s_k^{(j)} = \frac{1}{k} \left(x_i^{nk} \left(\prod_{a=0}^{n-1} f(a)\right)^k  + (-1)^{k-1}(g(0))^ky_i^k x_i^{(n-1)k}\left(\prod_{a=0}^{n-1} f(a)\right)^k\right). \label{charged-Delta-Hamiltonian-parameter}\end{align}
    \item[(b)] If the Boltzmann weights are not generalized free fermionic, (\ref{charged-Delta-Lattice-Hamiltonian-correspondence-equation}) does not hold for any choice of the $s_k^{(j)}$.
\end{enumerate}
\end{theorem}

\section{Solvability} \label{solvability-section}

In this section, we demonstrate that the generalized free fermion condition is important for solvability of the six-vertex model with charge. Consider the vertices in Figure \ref{general-charged-rectangular-vertices}. We will describe conditions on their Boltzmann weights that are necessary and sufficient for solvability, and in the case where the model is solvable, compute the weights of the $R$-vertices.

A (rectangular) lattice model is called solvable if there exist a set of \emph{$R$-vertices} that satisfy the \emph{Yang-Baxter equation}. Let $T_i$ denote a vertex with row index $i$. The Yang-Baxter equation is satisfied if there exists a set of vertex weights $R_{ij}$ such that for all possible choices of decorated spins $\alpha,\beta,\gamma,\delta,\epsilon,\eta$, we have equality of partition functions:

\begin{align}\label{YBE-diagram}
\begin{array}{c}
\scalebox{.85}{\begin{tikzpicture}
  \draw (0,0)--(2,0);
  \draw (0,2)--(2,2);
  \draw (1,-1)--(1,3);
  \coordinate (a1) at (-2,0);
  \coordinate (c1) at (0,2);
  \coordinate (a2) at (-2,2);
  \coordinate (c2) at (0,0);
  \draw (a1) to [out=0,in=180] (c1);
  \draw (a2) to [out=0,in=180] (c2);
  \draw[fill=white] (-2,0) circle (.3);
  \draw[fill=white] (-2,2) circle (.3);
  \draw[fill=white](1,3) circle(.3);
  \draw[fill=white] (2,2) circle (.3);
  \draw[fill=white] (2,0) circle (.3);
  \draw[fill=white](1,-1) circle (.3);
  \draw[fill=white] (0,2) circle (.3);
  \draw[fill=white] (0,0) circle (.3);
  \draw[fill=white](1,1) circle (.3);
  \path[fill=white] (-1,1) circle (.3);
  \path[fill=white] (1,2) circle (.3);
  \path[fill=white] (1,0) circle (.3);
  \node at (-2,0){$\alpha$};
  \node at (-2,2){$\beta$};
  \node at (1,3){$\gamma$};
  \node at (2,2) {$\delta$};
  \node at (2,0) {$\epsilon$};
  \node at (1,-1){$\eta$};
  \node at (1,2) {$T_i$};
  \node at (1,0) {$T_j$};
  \node at (-1,1){$R_{i,j}$};
  \node at (3,1) {{\Large $=$}};
  \end{tikzpicture}}
\hspace{1em}
\scalebox{.85}{\begin{tikzpicture}
  \draw (0,0)--(-2,0); 
  \draw (0,2)--(-2,2);
  \draw (-1,-1)--(-1,3);
  \coordinate (a1) at (2,0);
  \coordinate (c1) at (0,2);
  \coordinate (a2) at (2,2);
  \coordinate (c2) at (0,0);
  \draw (a1) to [out=180,in=0] (c1);
  \draw (a2) to [out=180,in=0] (c2);
  \draw[fill=white] (-2,0) circle (.3);
  \draw[fill=white] (-2,2) circle (.3);
  \draw[fill=white](-1,3) circle(.3);
  \draw[fill=white] (2,2) circle (.3);
  \draw[fill=white] (2,0) circle (.3);
  \draw[fill=white](-1,-1) circle (.3);
  \draw[fill=white] (0,2) circle (.3);
  \draw[fill=white] (0,0) circle (.3);
  \draw[fill=white](-1,1) circle (.3);
  \path[fill=white] (1,1) circle (.3);
  \path[fill=white] (-1,2) circle (.3);
  \path[fill=white] (-1,0) circle (.3);
  \node at (-2,0) {$\alpha$};
  \node at (-2,2) {$\beta$};
  \node at (-1,3) {$\gamma$};
  \node at (2,2) {$\delta$};
  \node at (2,0) {$\epsilon$};
  \node at (-1,-1) {$\eta$};
  \node at (-1,2) {$T_j$};
  \node at (-1,0) {$T_i$};
  \node at (1,1) {$R_{i,j}$};
  \end{tikzpicture}}\end{array}.
\end{align}

In particular, the edge labels on the $R$-vertices must also be decorated spins. We assume a further condition: \emph{conservation of decorated spin}. Namely, the pair of decorated spins entering $R_{ij}$ from the right equals the pair of decorated spins exiting to the left, in some order. The produces the vertices in Figure \ref{general-charged-R-vertices}. Each $R$-vertex depends on a pair of row indices $(i,j)$, where $i$ is the index of the top right and bottom left edges, and $j$ is the index of the bottom right and top left edges. We suppress this notation for readability.

\begin{figure}[h]
\begin{center}
\scalebox{.85}{$
\begin{array}{c@{\hspace{8pt}}c@{\hspace{8pt}}c@{\hspace{5pt}}c@{\hspace{5pt}}c@{\hspace{8pt}}c@{\hspace{8pt}}c@{\hspace{5pt}}c}
\toprule
\vx{a_1^{(i)}}&\vx{a_2^{(i)}}(a)&\vx{b_1^{(i)}} & \vx{b_2^{(i)}}(a) & \vx{c_1^{(i)}}&\vx{c_2^{(i)}}\\
\midrule
\begin{tikzpicture}
\coordinate (a) at (-.75, 0);
\coordinate (b) at (0, .75);
\coordinate (c) at (.75, 0);
\coordinate (d) at (0, -.75);
\coordinate (aa) at (-.75,.5);
\coordinate (cc) at (.75,.5);
\draw (a)--(0,0);
\draw (b)--(0,0);
\draw (c)--(0,0);
\draw (d)--(0,0);
\draw[fill=white] (a) circle (.25);
\draw[fill=white] (b) circle (.25);
\draw[fill=white] (c) circle (.25);
\draw[fill=white] (d) circle (.25);
\node at (0,1) { };
\node at (a) {$+$};
\node at (b) {$+$};
\node at (c) {$+$};
\node at (d) {$+$};
\end{tikzpicture}
&
\begin{tikzpicture}
\coordinate (a) at (-.75, 0);
\coordinate (b) at (0, .75);
\coordinate (c) at (.75, 0);
\coordinate (d) at (0, -.75);
\coordinate (aa) at (-.75,.5);
\coordinate (cc) at (.75,.5);
\draw (a)--(0,0);
\draw (b)--(0,0);
\draw (c)--(0,0);
\draw (d)--(0,0);
\draw[fill=white] (a) circle (.25);
\draw[fill=white] (b) circle (.25);
\draw[fill=white] (c) circle (.25);
\draw[fill=white] (d) circle (.25);
\node at (0,1) { };
\node at (a) {$-$};
\node at (b) {$-$};
\node at (c) {$-$};
\node at (d) {$-$};
\node at (aa) {$a+1$};
\node at (cc) {$a$};
\end{tikzpicture}
&
\begin{tikzpicture}
\coordinate (a) at (-.75, 0);
\coordinate (b) at (0, .75);
\coordinate (c) at (.75, 0);
\coordinate (d) at (0, -.75);
\coordinate (aa) at (-.75,.5);
\coordinate (cc) at (.75,.5);
\draw (a)--(0,0);
\draw (b)--(0,0);
\draw (c)--(0,0);
\draw (d)--(0,0);
\draw[fill=white] (a) circle (.25);
\draw[fill=white] (b) circle (.25);
\draw[fill=white] (c) circle (.25);
\draw[fill=white] (d) circle (.25);
\node at (0,1) { };
\node at (a) {$+$};
\node at (b) {$-$};
\node at (c) {$+$};
\node at (d) {$-$};
\end{tikzpicture}
&
\begin{tikzpicture}
\coordinate (a) at (-.75, 0);
\coordinate (b) at (0, .75);
\coordinate (c) at (.75, 0);
\coordinate (d) at (0, -.75);
\coordinate (aa) at (-.75,.5);
\coordinate (cc) at (.75,.5);
\draw (a)--(0,0);
\draw (b)--(0,0);
\draw (c)--(0,0);
\draw (d)--(0,0);
\draw[fill=white] (a) circle (.25);
\draw[fill=white] (b) circle (.25);
\draw[fill=white] (c) circle (.25);
\draw[fill=white] (d) circle (.25);
\node at (0,1) { };
\node at (a) {$-$};
\node at (b) {$+$};
\node at (c) {$-$};
\node at (d) {$+$};
\node at (aa) {$a+1$};
\node at (cc) {$a$};
\end{tikzpicture}
&
\begin{tikzpicture}
\coordinate (a) at (-.75, 0);
\coordinate (b) at (0, .75);
\coordinate (c) at (.75, 0);
\coordinate (d) at (0, -.75);
\coordinate (aa) at (-.75,.5);
\coordinate (cc) at (.75,.5);
\draw (a)--(0,0);
\draw (b)--(0,0);
\draw (c)--(0,0);
\draw (d)--(0,0);
\draw[fill=white] (a) circle (.25);
\draw[fill=white] (b) circle (.25);
\draw[fill=white] (c) circle (.25);
\draw[fill=white] (d) circle (.25);
\node at (0,1) { };
\node at (a) {$-$};
\node at (b) {$+$};
\node at (c) {$+$};
\node at (d) {$-$};
\node at (aa) {$1$};
\end{tikzpicture}
&
\begin{tikzpicture}
\coordinate (a) at (-.75, 0);
\coordinate (b) at (0, .75);
\coordinate (c) at (.75, 0);
\coordinate (d) at (0, -.75);
\coordinate (aa) at (-.75,.5);
\coordinate (cc) at (.75,.5);
\draw (a)--(0,0);
\draw (b)--(0,0);
\draw (c)--(0,0);
\draw (d)--(0,0);
\draw[fill=white] (a) circle (.25);
\draw[fill=white] (b) circle (.25);
\draw[fill=white] (c) circle (.25);
\draw[fill=white] (d) circle (.25);
\node at (0,1) { };
\node at (a) {$+$};
\node at (b) {$-$};
\node at (c) {$-$};
\node at (d) {$+$};
\node at (cc) {$0$};
\end{tikzpicture}
\\
   \bottomrule
\end{array}$}
\end{center}
\caption{A set of admissible vertices for the six-vertex model with charge.}
    \label{general-charged-rectangular-vertices}
\end{figure}
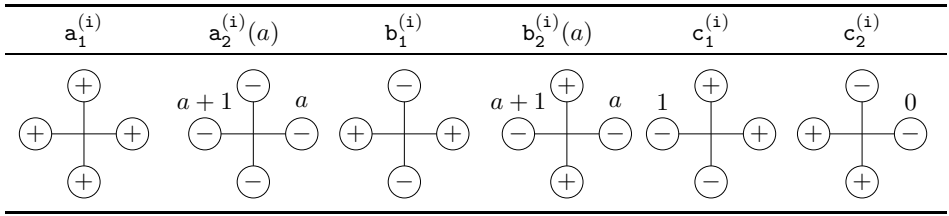

\begin{figure}[h]
\centering
\scalebox{.95}{$\begin{array}{c@{\hspace{10pt}}c@{\hspace{10pt}}c@{\hspace{15pt}}c@{\hspace{10pt}}c@{\hspace{10pt}}c@{\hspace{10pt}}c@{\hspace{10pt}}}
\toprule
\vx{A}_1 & \vx{A}_2(k,m) & \vx{A}^\times_2(k,m) & \vx{B}_1(k) & \vx{B}_2(k) & \vx{C}_1(k) & \vx{C}_2(k) \\
\midrule
\begin{tikzpicture}[scale=0.7]
\draw (0,0) to [out = 0, in = 180] (2,2);
\draw (0,2) to [out = 0, in = 180] (2,0);
\draw[fill=white] (0,0) circle (.35);
\draw[fill=white] (0,2) circle (.35);
\draw[fill=white] (2,2) circle (.35);
\draw[fill=white] (2,0) circle (.35);
\node at (0,0) {$+$};
\node at (0,2) {$+$};
\node at (2,2) {$+$};
\node at (2,0) {$+$};
\end{tikzpicture}
&
\begin{tikzpicture}[scale=0.7]
\draw (0,0) to [out = 0, in = 180] (2,2);
\draw (0,2) to [out = 0, in = 180] (2,0);
\draw[fill=white] (0,0) circle (.35);
\draw[fill=white] (0,2) circle (.35);
\draw[fill=white] (2,2) circle (.35);
\draw[fill=white] (2,0) circle (.35);
\node at (0,0) {$-$};
\node at (0,2) {$-$};
\node at (2,2) {$-$};
\node at (2,0) {$-$};
\node at (0,0.7) {$m$};
\node at (2,0.7) {$m$};
\node at (0,2.7) {$k$};
\node at (2,2.7) {$k$};
\end{tikzpicture}
&
\begin{tikzpicture}[scale=0.7]
\draw (0,0) to [out = 0, in = 180] (2,2);
\draw (0,2) to [out = 0, in = 180] (2,0);
\draw[fill=white] (0,0) circle (.35);
\draw[fill=white] (0,2) circle (.35);
\draw[fill=white] (2,2) circle (.35);
\draw[fill=white] (2,0) circle (.35);
\node at (0,0) {$-$};
\node at (0,2) {$-$};
\node at (2,2) {$-$};
\node at (2,0) {$-$};
\node at (0,0.7) {$k$};
\node at (2,0.7) {$m$};
\node at (0,2.7) {$m$};
\node at (2,2.7) {$k$};
\end{tikzpicture}
&
\begin{tikzpicture}[scale=0.7]
\draw (0,0) to [out = 0, in = 180] (2,2);
\draw (0,2) to [out = 0, in = 180] (2,0);
\draw[fill=white] (0,0) circle (.35);
\draw[fill=white] (0,2) circle (.35);
\draw[fill=white] (2,2) circle (.35);
\draw[fill=white] (2,0) circle (.35);
\node at (0,0) {$+$};
\node at (0,2) {$-$};
\node at (2,2) {$+$};
\node at (2,0) {$-$};
\node at (2,0.7) {$k$};
\node at (0,2.7) {$k$};
\end{tikzpicture}
&
\begin{tikzpicture}[scale=0.7]
\draw (0,0) to [out = 0, in = 180] (2,2);
\draw (0,2) to [out = 0, in = 180] (2,0);
\draw[fill=white] (0,0) circle (.35);
\draw[fill=white] (0,2) circle (.35);
\draw[fill=white] (2,2) circle (.35);
\draw[fill=white] (2,0) circle (.35);
\node at (0,0) {$-$};
\node at (0,2) {$+$};
\node at (2,2) {$-$};
\node at (2,0) {$+$};
\node at (0,0.7) {$k$};
\node at (2,2.7) {$k$};
\end{tikzpicture}
&
\begin{tikzpicture}[scale=0.7]
\draw (0,0) to [out = 0, in = 180] (2,2);
\draw (0,2) to [out = 0, in = 180] (2,0);
\draw[fill=white] (0,0) circle (.35);
\draw[fill=white] (0,2) circle (.35);
\draw[fill=white] (2,2) circle (.35);
\draw[fill=white] (2,0) circle (.35);
\node at (0,0) {$-$};
\node at (0,2) {$+$};
\node at (2,2) {$+$};
\node at (2,0) {$-$};
\node at (0,0.7) {$k$};
\node at (2,0.7) {$k$};
\end{tikzpicture}
&
\begin{tikzpicture}[scale=0.7]
\draw (0,0) to [out = 0, in = 180] (2,2);
\draw (0,2) to [out = 0, in = 180] (2,0);
\draw[fill=white] (0,0) circle (.35);
\draw[fill=white] (0,2) circle (.35);
\draw[fill=white] (2,2) circle (.35);
\draw[fill=white] (2,0) circle (.35);
\node at (0,0) {$+$};
\node at (0,2) {$-$};
\node at (2,2) {$-$};
\node at (2,0) {$+$};
\node at (0,2.7) {$k$};
\node at (2,2.7) {$k$};
\end{tikzpicture}
\\
   \bottomrule
\end{array}$}
    \caption{A set of $R$-vertices for the six-vertex model with charge.}
    \label{general-charged-R-vertices}
\end{figure}
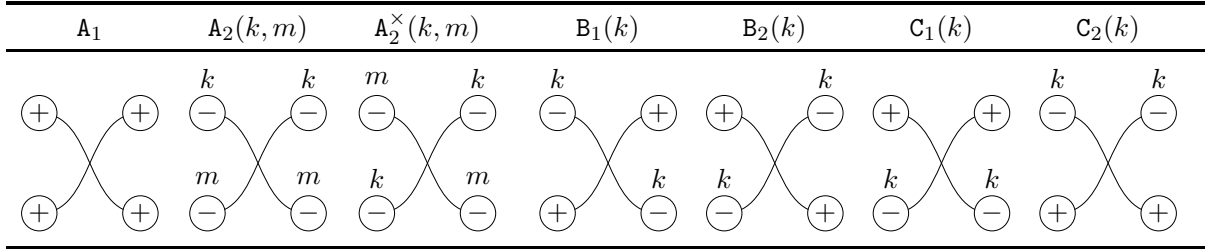

The solvability computation (\ref{YBE-diagram}) was done by hand. It consists of 20 cases, all choices of boundary spins such that $\alpha,\beta$, and $\gamma$ have the same number of $-$ spins as $\delta, \epsilon$, and $\eta$, and many of these cases break up into subcases based on charge. From the resulting equations, we can solve the model and determine conditions on its solvability.

The equations are given in Appendix \ref{appendix-equations}. Here, we present the results. The interested reader can check that the stated conditions are indeed necessary to satisfy the equations in Appendix \ref{appendix-equations}, and that the solutions given do indeed satisfy those conditions.

Let $n>1$. First, we require that \begin{align}\frac{\vx{a_1^{(i)}}\vx{a_2^{(i)}}(k)}{\vx{b_1^{(i)}}\vx{b_2^{(i)}}(k)}, \hspace{20pt} \frac{\vx{a_1^{(i)}}\vx{a_2^{(i)}}(0)}{\vx{c_1^{(i)}}\vx{c_2^{(i)}}}, \hspace{20pt} \text{and} \hspace{20pt} \frac{\vx{b_1^{(i)}}\vx{b_2^{(i)}}(0)}{\vx{c_1^{(i)}}\vx{c_2^{(i)}}} \label{independence-condition}\end{align} are independent of $i$. The first of these conditions along with one of the other two imply the third. Additionally, they imply that \[\Delta := \frac{\vx{a_1^{(i)}}\vx{a_2^{(i)}}(0) + \vx{b_1^{(i)}}\vx{b_2^{(i)}}(0) - \vx{c_1^{(i)}}\vx{c_2^{(i)}}}{2\sqrt{\vx{a_1^{(i)}}\vx{a_2^{(i)}}(0)\vx{b_1^{(i)}}\vx{b_2^{(i)}}(0)}}\] is independent of $i$ i.e. constant.

There are now two cases. If $\Delta\ne 0$, then the model is solvable if and only if \begin{align} \prod_{p=0}^{n-1} \vx{a_1^{(i)}}\vx{b_2^{(j)}}(p) = \prod_{p=0}^{n-1} \vx{a_1^{(j)}}\vx{b_2^{(i)}}(p) \label{non-free-fermion-charge-equation}\end{align} holds, and the $R$-vertex weights are given in Table \ref{R-vertex-weights-other}. Note that while the model is solvable in this case, it is not a very interesting solution. In particular, the weights of $\vx{B}_1(k), \vx{B}_2(k)$, and $\vx{A}_2^\times(k,m), k\ne m$ are all 0.

If $\Delta=0$, the zero charge free fermion condition holds, and the following additional condition is equivalent to solvability \[\frac{\vx{a_1^{(i)}}\vx{a_2^{(i)}}(k)\vx{a_1^{(i)}}\vx{a_2^{(i)}}(-k)}{\vx{b_1^{(i)}}\vx{b_2^{(i)}}(k)\vx{b_1^{(i)}}\vx{b_2^{(i)}}(-k)} = -\frac{\vx{a_1^{(i)}}\vx{a_2^{(i)}}(0)}{\vx{b_1^{(i)}}\vx{b_2^{(i)}}(0)}.\] Both sides of this equation are independent of $i$, and when simplified, it reduces to the charge condition (\ref{charge-condition}). The $R$-vertex weights for this case are given in Table \ref{R-vertex-weights-free-fermion}.

In fact, the conditions $\Delta=0$ (for all $i$), (\ref{charge-condition}), and the first condition in (\ref{independence-condition}) imply the rest of (\ref{independence-condition}). This is equivalent to each row of the lattice model satisfying the $\mathcal{F}$-free-fermion condition for \emph{the same} metaplectic Fock space $\mathcal{F}$.

The following theorem summarizes these results.

\begin{theorem} \label{charge-solvability-theorem}
The six-vertex model with charge is solvable in precisely two cases: \begin{enumerate}
    \item The rectangular weights satisfy the $\mathcal{F}$-free-fermion condition for some fixed metaplectic Fock space $\mathcal{F}$. The $R$-vertex weights are given in Table \ref{R-vertex-weights-free-fermion}.
    \item (\ref{independence-condition}) and (\ref{non-free-fermion-charge-equation}) hold and $\Delta\ne 0$. The $R$-vertex weights are given in Table \ref{R-vertex-weights-other}.
\end{enumerate}
\end{theorem}

\begin{table}
\begin{center}
\begin{tabular}{ |c|c| }
 \hline
 Vertex & Boltzmann weight  \\\hline 
 $\vx{A}_1$ & $\vx{a_1^{(i)}}\vx{a_2^{(j)}}(0) \cdot \prod_{p=1}^{n-1} \vx{a_1^{(i)}}\vx{b_2^{(j)}}(p) + \vx{b_1^{(j)}}\vx{b_2^{(i)}}(0) \cdot \prod_{p=1}^{n-1} \vx{a_1^{(j)}}\vx{b_2^{(i)}}(p)$ \\\hline
 $\vx{A}_2(k,k)$ & $\vx{a_1^{(j)}}\vx{a_2^{(i)}}(0) \prod_{p=1}^{n-1} \vx{a_1^{(j)}}\vx{b_2^{(i)}}(p) + \vx{b_1^{(i)}}\vx{b_2^{(j)}}(0) \prod_{p=1}^{n-1} \vx{a_1^{(i)}}\vx{b_2^{(j)}}(p)$ \\\hline
 $\vx{A}_2(k,m), k<m$ & $\frac{\vx{c_1^{(j)}}\vx{c_2^{(j)}}\vx{a_2^{(i)}}(0)}{\vx{a_2^{(j)}}(0)} \prod_{p=1}^{m-1} \vx{a_1^{(j)}}\vx{b_2^{(i)}}(p) \prod_{p=m}^{n-1} \vx{a_1^{(i)}}\vx{b_2^{(j)}}(p) \prod_{p=0}^{k-1} \frac{\vx{a_1^{(i)}}\vx{b_2^{(j)}}(p)}{\vx{a_1^{(j)}}\vx{b_2^{(i)}}(p)}$ \\\hline $\vx{A}_2(k,m), k>m$ & $\frac{\vx{c_1^{(j)}}\vx{c_2^{(j)}}\vx{b_1^{(i)}}}{\vx{b_1^{(j)}}} \prod_{p=1}^{k-1} \vx{a_1^{(i)}}\vx{b_2^{(j)}}(p) \prod_{p=k}^{n-1} \vx{a_1^{(j)}}\vx{b_2^{(i)}}(p) \prod_{p=0}^{m-1} \frac{\vx{a_1^{(j)}}\vx{b_2^{(i)}}(p)}{\vx{a_1^{(i)}}\vx{b_2^{(j)}}(p)}$ \\\hline $\vx{A}^\times_2(k,m), k\ne m$ & \begin{tabular}{@{}c@{}} $\frac{\vx{a_2^{(i)}}(k-m)}{\vx{b_2^{(i)}}(k-m)} \left(\prod_{p=0}^{n-1} \vx{a_1^{(j)}}\vx{b_2^{(i)}}(p) - \prod_{j=0}^{n-1} \vx{a_1^{(i)}}\vx{b_2^{(j)}}(p)\right)$ \\ $= \frac{\vx{b_2^{(j)}}(m-k)}{\vx{a_1^{(j)}}(m-k)} \left(\vx{a_2^{(j)}}(0)\vx{b_1^{(i)}}\prod_{p=1}^{n-1} \vx{a_1^{(i)}}\vx{b_2^{(j)}}(p) - \vx{a_2^{(i)}}(0)\vx{b_1^{(j)}}\prod_{p=1}^{n-1} \vx{a_1^{(j)}}\vx{b_2^{(i)}}(p)\right)$\end{tabular} \\\hline
 $\vx{B}_1(k)$ & $\vx{a_2^{(j)}}(0)\vx{b_1^{(i)}} \cdot \prod_{p=1}^{n-1} \vx{a_1^{(i)}}\vx{b_2^{(j)}}(p) - \vx{a_2^{(i)}}(0)\vx{b_1^{(j)}} \cdot \prod_{p=1}^{n-1} \vx{a_1^{(j)}}\vx{b_2^{(i)}}(p)$  \\\hline
 $\vx{B}_2(k)$ & $\prod_{p=0}^{n-1} \vx{a_1^{(j)}}\vx{b_2^{(i)}}(p) - \prod_{p=0}^{n-1} \vx{a_1^{(i)}}\vx{b_2^{(j)}}(p)$  \\\hline
 $\vx{C}_1(k)$ & $\vx{c_1^{(i)}}\vx{c_2^{(j)}} \cdot \prod_{p=1}^{k-1} \vx{a_1^{(j)}}\vx{b_2^{(i)}}(p) \cdot \prod_{p=k}^{n-1} \vx{a_1^{(i)}}\vx{b_2^{(j)}}(p)$  \\\hline
 $\vx{C}_2(k)$ & $\vx{c_1^{(j)}}\vx{c_2^{(i)}} \cdot \prod_{p=1}^{k-1} \vx{a_1^{(i)}}\vx{b_2^{(j)}}(p) \cdot \prod_{p=k}^{n-1} \vx{a_1^{(j)}}\vx{b_2^{(i)}}(p)$  \\
 \hline
\end{tabular}
\end{center}
\caption{A set of $R$-vertex weights for the generalized free fermion case (Theorem \ref{charge-solvability-theorem}(a)). For vertices $\vx{C}_1(k)$ and $\vx{C}_2(k)$, $k$ is taken to be $1\le k\le n$, while for vertices $\vx{A}_2(k,m)$, $k$ and $m$ are taken to be $0\le k,m\le n$. The formulas for vertices $\vx{A}_2(k,m)$ and $\vx{A}^\times_2(k,m)$ hold when $k\ne m$ modulo $n$. In the case where $k=m$, both vertices equal $\vx{A}_2(k,k)$. Note that $\vx{B}_1(k), \vx{B}_2(k)$, and $\vx{A}_2(k,k)$ are independent of $k$.}
\label{R-vertex-weights-free-fermion}
\end{table}

\begin{table}
\begin{center}
\begin{tabular}{ |c|c| }
 \hline
 Vertex & Boltzmann weight  \\\hline 
 $\vx{A}_1$ & $\frac{\vx{a_1^{(j)}}\vx{a_2^{(j)}}(0) + \vx{b_1^{(j)}}\vx{b_2^{(j)}}(0)}{\vx{a_1^{(j)}}\vx{b_2^{(j)}}(0)} \cdot \prod_{p=0}^{n-1} \vx{a_1^{(i)}}\vx{b_2^{(j)}}(p)$ \\\hline
 $\vx{A}_2(k,k)$ & $\frac{\vx{a_1^{(i)}}\vx{a_2^{(i)}}(0) + \vx{b_1^{(i)}}\vx{b_2^{(i)}}(0)}{\vx{a_1^{(i)}}\vx{b_2^{(i)}}(0)} \cdot \prod_{p=0}^{n-1} \vx{a_1^{(i)}}\vx{b_2^{(j)}}(p)$ \\\hline
 $\vx{A}_2(k,m), k<m$ & $\frac{\vx{c_1^{(j)}}\vx{c_2^{(j)}}\vx{a_2^{(i)}}(0)}{\vx{a_2^{(j)}}(0)} \prod_{p=1}^{m-1} \vx{a_1^{(j)}}\vx{b_2^{(i)}}(p) \prod_{p=m}^{n-1} \vx{a_1^{(i)}}\vx{b_2^{(j)}}(p) \prod_{p=0}^{k-1} \frac{\vx{a_1^{(i)}}\vx{b_2^{(j)}}(p)}{\vx{a_1^{(j)}}\vx{b_2^{(i)}}(p)}$ \\\hline $\vx{A}_2(k,m), k>m$ & $\frac{\vx{c_1^{(j)}}\vx{c_2^{(j)}}\vx{b_1^{(i)}}}{\vx{b_1^{(j)}}} \prod_{p=1}^{k-1} \vx{a_1^{(i)}}\vx{b_2^{(j)}}(p) \prod_{p=k}^{n-1} \vx{a_1^{(j)}}\vx{b_2^{(i)}}(p) \prod_{p=0}^{m-1} \frac{\vx{a_1^{(j)}}\vx{b_2^{(i)}}(p)}{\vx{a_1^{(i)}}\vx{b_2^{(j)}}(p)}$ \\\hline $\vx{A}^\times_2(k,m), k\ne m$ & $0$ \\\hline
 $\vx{B}_1(k)$ & $0$  \\\hline
 $\vx{B}_2(k)$ & $0$  \\\hline
 $\vx{C}_1(k)$ & $\vx{c_1^{(i)}}\vx{c_2^{(j)}} \cdot \prod_{p=1}^{k-1} \vx{a_1^{(j)}}\vx{b_2^{(i)}}(p) \cdot \prod_{p=k}^{n-1} \vx{a_1^{(i)}}\vx{b_2^{(j)}}(p)$  \\\hline
 $\vx{C}_2(k)$ & $\vx{c_1^{(j)}}\vx{c_2^{(i)}} \cdot \prod_{p=1}^{k-1} \vx{a_1^{(i)}}\vx{b_2^{(j)}}(p) \cdot \prod_{p=k}^{n-1} \vx{a_1^{(j)}}\vx{b_2^{(i)}}(p)$  \\
 \hline
\end{tabular}
\end{center}
\caption{A set of $R$-vertex weights for the non-free-fermion case (Theorem \ref{charge-solvability-theorem}(b)). The same charge conventions are used as in Table \ref{R-vertex-weights-free-fermion}. These vertex weights are precisely the same as those in Table \ref{R-vertex-weights-free-fermion} when we impose the additional condition (\ref{non-free-fermion-charge-equation}).}
\label{R-vertex-weights-other}
\end{table}

\begin{remark}
A comparison of Theorem \ref{charge-solvability-theorem} and the solvability criterion for the six-vertex model given in \cite[Theorem~1]{BBF-Schur-polynomials-arxiv} gives evidence of the utility of the Hamiltonian perspective. The equations in Appendix \ref{appendix-equations} and the process of solving them depend on the condition $n>1$ so that we have at least one non-zero charge, and indeed the criteria given by Theorem \ref{charge-solvability-theorem}(b) do not descend to the six-vertex ($n=1$) case.

However, for generalized free fermion models (Theorem \ref{charge-solvability-theorem}(a)), our criteria do in fact still hold when setting $n=1$. The Hamiltonian perspective explains why this happens. Theorem \ref{charged-Delta-lattice-Hamiltonian} and its proof in Section \ref{proof-of-quantum-main-result-section} do still hold when $n=1$, and so the connection of solvability to Hamiltonian operators explains why the solution does descend in this case.
\end{remark}

\begin{remark}
See \cite{n-color-ice} for subsequent work with a similar solvability criterion for the more general setting of lattice models with \emph{color}. It would be interesting to see whether the solvability conditions in that paper have any relation to (some generalization of) Hamiltonian operators.
\end{remark}

\section{Proof of Theorem \ref{charged-Delta-lattice-Hamiltonian}} \label{proof-of-quantum-main-result-section}

In this section, we prove Theorem \ref{charged-Delta-lattice-Hamiltonian}. We cannot use the same approach to prove this theorem that we used in the $n=1$ case, as Wick's theorem is not available to us in this context, so we will use an induction argument due to Brubaker, Buciumas, Bump, and Gustafsson \cite[\S~4]{BBBG-Hamiltonian}. Our proof closely mirrors theirs, with generalizations in certain places, and so we will sometimes refer to their proof for steps that are identical.

We will prove the one-row case first, and as a simple scaling gets us the proper powers of $A_i$ and $B_i$, we set them equal to 1 for convenience.

\begin{proposition}
\label{charged-Delta-one-row}
Theorem \ref{charged-Delta-lattice-Hamiltonian}(a) holds in the case $N=1$, $A_1=B_1=1$.
\end{proposition}

Let $x=x_1, y=y_1$.  Suppose $\mathfrak{S}^q$ satisfies the generalized free fermionic condition. and let $\mathcal{F} = \mathcal{F}(\mathfrak{S}^q)$ and $g$ be the corresponding function $g(a) = h(a)/f(a)$. 

If $k-n\le s\le k$, define \[\zeta := x^n \prod_{a=0}^{n-1} f(a), \hspace{20pt} \tau := -\frac{g(0)\zeta y}{x}, \hspace{20pt} \zeta_s = x^{s-k+n}\prod_{k-n\le a\le s-1} f(a).\]

The following lemma is the base case of Proposition \ref{charged-Delta-one-row}.

\begin{lemma}
If $M<n$, then Proposition \ref{charged-Delta-one-row} holds.
\end{lemma}

\begin{proof}
Since there are at most $n$ vertices in the lattice model, there is not enough room for a particle to travel from one column to another. Thus, $Z(\mathfrak{S}^q_{\lambda/\mu}) = 1$ if $\lambda=\mu$, and otherwise is zero. Similarly, any current operator $J_k, k\ge 1$ will send a particle at least $n$ spaces to the left. Therefore, $\langle\mu|e^{H_+}|\lambda\rangle = 1$ if $\lambda=\mu$, and otherwise is zero, so the two sides are equal.
\end{proof}

For the inductive step, we will introduce an operator $\rho^*_j(t)$ that has a similar relationship to both the partition function and the Hamiltonian.

Let $\psi_j^*$ be the creation operator, $\psi_j^*\cdot \underline{u} := u_j\wedge \underline{u}$. A deletion operator is not straightforward to define in the $q\ne -1$ case, so instead we will compute the relationship between the actions of $\psi_j^*$ and $e^{H_+}$.

Let $\psi^*(t) = \sum_{j\in\Z} \psi_j^* t^j$. Recall from (\ref{S-equals-log-H}) that \[S(t^n) = \sum_{k\ge 1} s_k t^{kn} = \log\left(\sum_{m\ge 0} h_m t^{mn}\right) = \log\left(H(t^n)\right).\]

\begin{lemma}
If (\ref{charged-Delta-Hamiltonian-parameter}) holds and $N=1$, \begin{align} H(t) = 1 + \sum_{k\ge 1} (1 - \tau) \zeta^{k-1} t^k. \label{H(t)-charge-equation}\end{align}
\end{lemma}

\begin{proof}
We start with (\ref{H(t)-charge-equation}), and derive (\ref{charged-Delta-Hamiltonian-parameter}).

Assuming (\ref{H(t)-charge-equation}), we have \[H(t) = 1 + \frac{(1 - \tau)t}{1-\zeta t} = \frac{1 - \tau t}{1-\zeta t},\] so \[\log H(t) = \log(1-\tau t) - \log(1-\zeta t)\] and \[\left.\frac{d^k}{dt^k}(\log H(t))\right|_{t=0} = (k-1)!(\zeta^k - \tau^k),\] so \[s_k = \frac{1}{k} (\zeta^k - \tau^k),\] as in (\ref{charged-Delta-Hamiltonian-parameter}). All the steps are reversible, so (\ref{charged-Delta-Hamiltonian-parameter}) implies (\ref{H(t)-charge-equation}).
\end{proof}

\begin{lemma} \label{rho-H-commutation-lemma}
\begin{enumerate}
\item[(a)] \[e^{H_+} \psi^*(t) e^{-H_+} =  H(t^n)\psi^*(t).\]
\item[(b)] Let $\rho_k^*(t) = \psi_k^* - t\psi_{j-n}^*$. Then, $e^{H_+} \rho_k^*(\zeta)e^{-H} = \rho_k^*(\tau)$.
\end{enumerate}
\end{lemma}

\begin{proof}

\begin{enumerate}
\item[(a)] Observe that $[J_m,\psi_j^*] = \psi^*_{j-mn}$. Thus, \[[H, \psi^*(t)] = \sum_{m\ge 1}\sum_{j\in\Z} s_m t^j \psi_{j-mn}^* = \sum_{m\ge 1} s_mt^{mn} \psi^*(t) = S(t^n) \psi^*(t) = \log\left(H(t^n)\right) \psi^*(t).\] Then, (a) follows from the Baker-Campbell-Hausdoff Theorem.

\item[(b)] Recall that $h_0 = 1$, and for $j\ge 1, h_j = s_1\zeta^{j-1}$. By matching coefficients in part (a), \[e^{H_+} \psi_k^* e^{-H_+} = \psi_k^* + \sum_{j\ge 1} \psi_{k-jn}^* h_j,\] so \begin{align*} e^{H_+}\rho_k^*(\zeta)e^{-H_+} &= \psi_k^* + \sum_{j\ge 1} \psi_{k-jn}^* h_j - \zeta\psi_{k-n}^* - \sum_{j\ge 1} \psi_{k-n-jn}^* \zeta h_j \\&= \psi_k^* + \sum_{j\ge 1} \psi_{k-jn}^* h_j - \zeta\psi_{k-n}^* + \psi^*_{k-n}h_1 - \sum_{j\ge 1} \psi_{k-jn}^* h_j \\&= \psi_k^* + (h_1-\zeta)\psi_{k-n}^*.\end{align*} Finally, note that $\zeta-h_1 = \tau$.
\end{enumerate}
\end{proof}

We'll do a similar conjugation on the lattice model side. To do this, we'll need a column-restricted version $\hat{T}_k$ of the row transfer matrix $T^q$. We let $\hat{T}_k$ be the operator such that $\langle\mu|\hat{T}_k|\lambda\rangle$ is the  partition function of the following lattice model \[\scalebox{1.0}
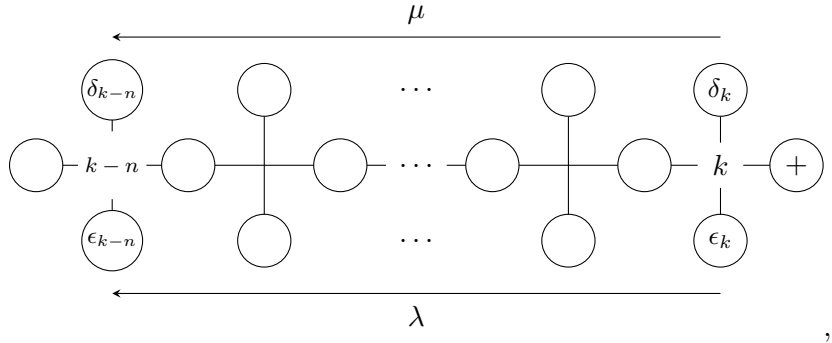
{
\begin{tikzpicture}
  \coordinate (aa) at (1,1);
  \coordinate (ab) at (1,2);
  \coordinate (ac) at (1,3);
  \coordinate (ba) at (2,1);
  \coordinate (bb) at (2,2);
  \coordinate (bc) at (2,3);
  \coordinate (ca) at (3,1);
  \coordinate (cb) at (3,2);
  \coordinate (cc) at (3,3);
  \coordinate (da) at (4,1);
  \coordinate (db) at (4,2);
  \coordinate (dc) at (4,3);
  \coordinate (ea) at (5,1);
  \coordinate (eb) at (5,2);
  \coordinate (ec) at (5,3);
  \coordinate (fa) at (6,1);
  \coordinate (fb) at (6,2);
  \coordinate (fc) at (6,3);
  \coordinate (ga) at (7,1);
  \coordinate (gb) at (7,2);
  \coordinate (gc) at (7,3);
  \coordinate (ha) at (8,1);
  \coordinate (hb) at (8,2);
  \coordinate (hc) at (8,3);
  \coordinate (ia) at (9,1);
  \coordinate (ib) at (9,2);
  \coordinate (ic) at (9,3);
  \coordinate (ja) at (10,1);
  \coordinate (jb) at (10,2);
  \coordinate (jc) at (10,3);
  \coordinate (ka) at (11,1);
  \coordinate (kb) at (11,2);
  \coordinate (kc) at (11,3);
  \draw (ab)--(kb);
  \draw (ba)--(bc);
  \draw (da)--(dc);
  \draw (ha)--(hc);
  \draw (ja)--(jc);
  \draw[fill=white] (ba) circle (.4);
  \draw[fill=white] (da) circle (.35);
  \draw[fill=white] (ha) circle (.35);
  \draw[fill=white] (ja) circle (.35);
  \draw[fill=white] (bc) circle (.4);
  \draw[fill=white] (dc) circle (.35);
  \draw[fill=white] (hc) circle (.35);
  \draw[fill=white] (jc) circle (.35);
  \draw[fill=white] (ab) circle (.35);
  \draw[fill=white] (cb) circle (.35);
  \draw[fill=white] (eb) circle (.35);
  \draw[fill=white] (gb) circle (.35);
  \draw[fill=white] (ib) circle (.35);
  \draw[fill=white] (kb) circle (.35);
  \path[fill=white] (bb) circle (.45);
  \path[fill=white] (fb) circle (.4);
  \path[fill=white] (jb) circle (.3);
  \node at (fc) {$\ldots$};
  \node at (fb) {$\ldots$};
  \node at (fa) {$\ldots$};
  \node at (kb) {$+$};
  \node at (jb) {$k$};
  \node at (bb) {{\scriptsize $k-n$}};
  \node at (jc) {{\small $\delta_k$}};
  \node at (bc) {{\footnotesize $\delta_{k-n}$}};
  \node at (ja) {{\small $\epsilon_k$}};
  \node at (ba) {{\footnotesize $\epsilon_{k-n}$}};
  \draw [->,>=stealth] (10,3.7)--(2,3.7);
  \node at (6,4) {$\mu$};
  \draw [->,>=stealth] (10,0.3)--(2,0.3);
  \node at (6,0) {$\lambda$};
\end{tikzpicture}},\]

where \[\epsilon_i = \begin{cases} -, &  \text{if $\lambda$ has a part of size $i$},\\ +, & \text{otherwise,}\end{cases} \hspace{20pt} \delta_i = \begin{cases} -, &  \text{if $\mu$ has a part of size $i$},\\ +, & \text{otherwise.}\end{cases}\] The leftmost spin is undetermined, and has at most a unique possibility. Colloquially, $\hat{T}_k$ is the partition function of columns $k-n$ through $k$ of the one row lattice model. $\mathfrak{S}^q_{\lambda/\mu}$, 

Let us define several quantities that will be useful in the proof of the next lemma, some of which depend on the spins $\epsilon_{k-n},\ldots,\epsilon_k$ associated to $\lambda$. Let $\gamma = \frac{f(0)x + h(0)y}{f(0)x} = 1 + g(0)\frac{y}{x}$. Given integers $k$ and $s$ where $k\ge s > k-n$ let \[G := \prod_{\substack{k-n+1\le i\le k-1 \\ \epsilon_i=-}} \frac{g(k-i)\cdot y}{x}, \hspace{20pt} G_s := \prod_{\substack{k-n+1\le i\le s-1 \\ \epsilon_i=-}} \frac{g(s-i)\cdot y}{x}, \hspace{20pt} G'_s := \prod_{\substack{k-n+1\le i\le k-1 \\ \epsilon_i=- \\ i\ne s}} g(k-i).\]

\begin{lemma} \label{rho-T-hat-commutation-lemma}
For all strict partitions $\lambda,\mu$, \[\langle\mu|\hat{T}_k\rho_k^*(\zeta)|\lambda\rangle = \langle\mu|\rho_k^*(\tau)\hat{T}_k|\lambda\rangle.\]
\end{lemma}

\begin{proof}
We prove this by cases. The result follows because in every row of the following two tables, the sum of the first two columns is equal to the sum of the last two columns.

First, suppose that for all $k>i>k-n$, $\epsilon_i=\delta_i$. The following table gives the relevant expressions for all choices of the spins in columns $k$ and $k-n$.

\begin{center}
\begin{tabular}{ |c||c||c|c|c|c| } 
 \hline
 $(\epsilon_k, \epsilon_{k-n}, \delta_k, \delta_{k-n})$ & $\langle\eta|\hat{T}|\xi\rangle$ & $\langle\eta|\hat{T}\psi_k^*|\xi\rangle$ & $-\zeta\langle\eta|\hat{T}\psi_{k-n}^*|\xi\rangle$ & $\langle\eta|\psi_k^*\hat{T}|\xi\rangle$ & $-\tau\langle\eta|\psi_{k-n}^*\hat{T}|\xi\rangle$  \\\hline 
 $(+,+,+,+)$&1&$\gamma\zeta G$&$-\gamma\zeta G$&0&0\\\hline
 $(+,+,+,-)$&0&$\gamma\zeta G$&$-\zeta G$&0&$-\tau G$\\\hline
 $(+,+,-,+)$&0&1&0&1&0\\\hline
 $(+,-,+,+)$&$\gamma$&0&0&0&0\\\hline
 $(-,+,+,+)$&$\gamma\zeta G$&0&0&0&0\\\hline
 $(+,+,-,-)$&0&0&0&0&0\\\hline
 $(+,-,+,-)$&1&$-\gamma\tau G$&0&0&$-\gamma\tau G$\\\hline
 $(+,-,-,+)$&0&$\gamma$&0&$\gamma$&0\\\hline
 $(-,+,+,-)$&$\gamma\zeta G$&0&$-\gamma\tau\zeta G^2 $&0&$-\gamma\tau\zeta G^2$\\\hline
 $(-,+,-,+)$&1&0&$\gamma\zeta G$&$\gamma\zeta G$&0\\\hline
 $(-,-,+,+)$&0&0&0&0&0\\\hline
 $(+,-,-,-)$&0&1&0&1&0\\\hline
 $(-,+,-,-)$&0&0&$\zeta G$&$\gamma\zeta G$&$\tau G$\\\hline
 $(-,-,+,-)$&$-\gamma\tau G$&0&0&0&0\\\hline
 $(-,-,-,+)$&$\gamma$&0&0&0&0\\\hline
 $(-,-,-,-)$&1&0&0&$-\gamma\tau G$&$\gamma\tau G$\\\hline
\end{tabular}
\end{center}

Now, suppose that for a unique $k>s>k-n$, $\epsilon_s=-, \delta_s = +$. The next table enumerates these cases.
\begin{center}
\scalebox{0.8}{
\begin{tabular}{ |c||c||c|c|c|c| } 
 \hline
 $(\epsilon_k, \epsilon_{k-n}, \delta_k, \delta_{k-n})$ & $\langle\eta|\hat{T}|\xi\rangle$ & $\langle\eta|\hat{T}\psi_k^*|\xi\rangle$ & $-\zeta\langle\eta|\hat{T}\psi_{k-n}^*|\xi\rangle$ & $\langle\eta|\psi_k^*\hat{T}|\xi\rangle$ & $-\tau\langle\eta|\psi_{k-n}^*\hat{T}|\xi\rangle$  \\\hline 
 $(+,+,+,+)$&$\gamma G_s\zeta_s$&0&0&0&0\\\hline
 $(+,+,+,-)$&0&0&$-\gamma\tau G_sG_s'\zeta_s$&0&$-\gamma \tau G_sG_s'\zeta_s$\\\hline
 $(+,+,-,+)$&0&$\gamma G_s\zeta_s$&0&$\gamma G_s\zeta_s$&0\\\hline
 $(+,-,+,+)$&0&0&0&0&0\\\hline
 $(-,+,+,+)$&0&0&0&0&0\\\hline
 $(+,+,-,-)$&0&0&0&0&0\\\hline
 $(+,-,+,-)$&$\gamma G_s\zeta_s g(s-k) yx^{-1}$&0&0&0&0\\\hline
 $(+,-,-,+)$&0&0&0&0&0\\\hline
 $(-,+,+,-)$&0&0&0&0&0\\\hline
 $(-,+,-,+)$&$\gamma G_s\zeta_s$&0&0&0&0\\\hline
 $(-,-,+,+)$&0&0&0&0&0\\\hline
 $(+,-,-,-)$&0&$\gamma G_s\zeta_sg(s-k) yx^{-1}$&0&$\gamma G_s\zeta_s g(s-k) yx^{-1}$&0\\\hline
 $(-,+,-,-)$&0&0&$\gamma\tau G_sG_s'\zeta_s$&0&$\gamma \tau G_sG_s'\zeta_s$\\\hline
 $(-,-,+,-)$&0&0&0&0&0\\\hline
 $(-,-,-,+)$&0&0&0&0&0\\\hline
 $(-,-,-,-)$&$\gamma G_s\zeta_s g(s-k) yx^{-1}$&0&0&0&0\\\hline
\end{tabular}}
\end{center}

\end{proof}

\begin{lemma} \label{rho-T-commutation-lemma}
$\langle\eta|T\rho_k^*(\zeta)|\xi\rangle = \langle\eta|\rho_k^*(\tau)T|\xi\rangle$.
\end{lemma}

\begin{proof}
This follows from Lemma \ref{rho-T-hat-commutation-lemma} by the same argument as \cite[Proposition~4.4]{BBBG-Hamiltonian}.
\end{proof}

\begin{proof}[Proof of Proposition \ref{charged-Delta-one-row}]
This follows from Propositions \ref{rho-H-commutation-lemma} and \ref{rho-T-commutation-lemma} by the same argument as the proof of Theorem A in \cite{BBBG-Hamiltonian}.
\end{proof}

The rest of the proof is similar to the proof of Proposition \ref{Reduced-Delta-lattice-Hamiltonian}. Let $\mathfrak{S}^{q,N}_{\lambda/\mu}$ denote the usual $\mathfrak{S}$ lattice model with $N$ rows. Recall that $H_+ = \sum_{j=1}^N \phi_j$.

\begin{proof}[Proof of Theorem \ref{charged-Delta-lattice-Hamiltonian}]
For part a, \begin{align*} \prod_{i=1}^N A_i^{M+1} B_i^{\ell(\lambda)} \langle\mu|e^{H_+}|\lambda\rangle &= \prod_{i=1}^N A_i^{M+1} B_i^{\ell(\lambda)}\langle\mu|e^{\phi_N}\ldots e^{\phi_1}|\lambda\rangle \\&= \sum_{\nu_1,\ldots,\nu_{N-1}} \langle\mu|e^{\phi_N}|\nu_{N-1}\rangle \langle\nu_{N-1}|e^{\phi_{N-1}}|\nu_{N-2}\rangle\ldots \langle\nu_1|e^{\phi_1}|\lambda\rangle \\&= \sum_{\nu_1,\ldots,\nu_{N-1}} Z(\mathfrak{S}^{q,1}_{\nu_{N-1}/\mu})\ldots Z(\mathfrak{S}^{q,1}_{\lambda/\nu_1}) \\&= Z(\mathfrak{S}^{q,N}_{\lambda/\mu}).\end{align*}

To see that the solution (\ref{charged-Delta-Hamiltonian-parameter}) for $s_k^{(j)}$ is unique, note that the partition function $Z(\mathfrak{S}^q_{(mn)/(0)})$ depends only on $s_k^{(j)}$ for $1\le k\le m$.

For part b, we need to show that both the free fermion condition and the charge condition (assuming $n>1$) are required in order to satisfy (\ref{charged-Delta-Lattice-Hamiltonian-correspondence-equation}). Consider the model $\ba{\mathfrak{S}}^q$, which is equal to $\mathfrak{S}^q$ but with unspecified Boltzmann weights $\vx{a_1^{(i)}}, \vx{a_2^{(i)}}(a), \vx{b_1^{(i)}}, \vx{b_2^{(i)}}(a), \vx{c_1^{(i)}}, \vx{c_2^{(i)}}$. By rescaling the weights as we do above, we can assume without loss of generality that $\vx{a_1^{(i)}} = \vx{b_1^{(i)}} = \vx{c_2^{(i)}} = 1$ for all $i$. Then we assume there exist Hamiltonian parameters $s_k^{(j)}, k\ge 1, 1\le j \le N$ such that \begin{align}Z(\ba{\mathfrak{S}}^q_{\lambda/\mu}) = {\Large{*}} \cdot \langle\mu|e^{H_+}|\lambda\rangle \hspace{20pt} \text{for all strict partitions $\lambda,\mu$ and all $M,N$.} \label{general-charged-correspondence-equation}\end{align}

For the zero charge free fermion condition, we observe that if $M$ and all the parts of $\lambda$ and $\mu$ are multiples of $n$, then there are no vertices of type $\vx{a_2^{(i)}}(a)$ where $a\ne 0$. This means that both the lattice model and Hamiltonian reduce to the $n=1$ case. Namely, let $\mathfrak{S}_{red}$ be the lattice model with the following weights and boundary conditions \[(\vx{a_1^{(i)}})_{red} := (\vx{a_1^{(i)}})^n, \hspace{10pt} (\vx{a_2^{(i)}})_{red} := \vx{a_2^{(i)}}(0)\prod_{a=1}^N \vx{b_2^{(i)}}(a), \hspace{10pt} (\vx{b_1^{(i)}})_{red} = \vx{b_1^{(i)}}(\vx{a_1^{(i)}})^{n-1},\] \[(\vx{b_2^{(i)}})_{red} = \vx{b_2^{(i)}}(0)\prod_{a=1}^N \vx{b_2^{(i)}}(a), \hspace{10pt} (\vx{c_1^{(i)}})_{red} = \vx{c_1^{(i)}}\prod_{a=1}^N \vx{b_2^{(i)}}(a), \hspace{10pt} (\vx{c_2^{(i)}})_{red} = \vx{c_2^{(i)}}(\vx{a_1^{(i)}})^{n-1}.\] \[M_{red} := \frac{M}{n}, \hspace{10pt} \lambda_{red} := \left(\frac{\lambda_1}{n}, \hspace{10pt} \frac{\lambda_2}{n} \ldots\right), \hspace{10pt} \mu_{red} := \left(\frac{\mu_1}{n}, \frac{\mu_2}{n} \ldots\right).\] In this case, (\ref{general-charged-correspondence-equation}) is satisfied precisely when $(\mathfrak{S}_{red})_{\lambda_{red}/\mu_{red}}$ and $H_{red} : =  \sum_{k\ge 1} s_kJ_k$ satisfy (\ref{General-Lattice-Hamiltonian-correspondence-equation}), so $\mathfrak{S}_{red}$ must satisfy the free fermion condition. The reduction hasn't changed the partition function or $\tau$ function, so thus $\mathfrak{S}^q$ must also satisfy the free fermion condition.

For the charge condition, if $0<p<n$, then \[\langle (p,0)|e^{H_+}|(n,p)\rangle = \langle v_p\wedge v_0\ldots|e^{H_+}|v_n\wedge v_p\wedge\ldots\rangle = g(n-p)\langle v_p\wedge v_0\ldots|e^{H_+}|v_p\wedge v_n\wedge\ldots\rangle = g(n-p)s_1,\] while \[Z(\ba{\mathfrak{S}}^q_{(n,p)/(p,0)}) = \frac{\vx{a_2^{(1)}}(n-p)\vx{c_2^{(1)}}}{\vx{b_2^{(1)}}(n-p)\vx{b_2^{(1)}}(0)}\left(\prod_{i=1}^n \vx{b_2^{(1)}}(a)\right).\] Meanwhile, \[\langle (0)|e^{H_+}|(n)\rangle = s_1 \hspace{20pt} \text{and} \hspace{20pt} Z(\ba{\mathfrak{S}}^q_{(n)/(0)}) = \frac{\vx{c_2^{(1)}}}{\vx{b_2^{(1)}}(0)}\left(\prod_{i=1}^n \vx{b_2^{(1)}}(a)\right),\] so we must have $\frac{\vx{a_2^{(1)}}(a)}{\vx{b_2^{(1)}}(a)} = g(a)$ for all $0<a<n$. Combining this with the condition $g(a)g(-a) = -g(0)$ gives the charge condition for the first row. For the other weights, use the branching rules \[Z(\ba{\mathfrak{S}}^q_{\lambda/\mu}) = \sum_\nu Z(\ba{\mathfrak{S}}^{q,N-1}_{\lambda/\nu}) Z(\ba{\mathfrak{S}}^{q,1}_{\nu/\mu}) \] \[\langle\mu|e^{H_+}|\lambda\rangle = \sum_\nu \langle\mu|e^{\phi_N}|\nu\rangle\langle\nu|e^{\phi_{N-1}}\cdots e^{\phi_1}|\lambda\rangle\] and induction.
\end{proof}

Let $\mathcal{F}$ be a Fock space, and let $\mathcal{F}_*$ equal $\mathcal{F}$ as a vector space, but with wedge relations defined by $g_*(a) := (g(-a))^{-1}$. In particular, we are sending $q\mapsto q^{-1}$. We'll write $|\lambda\rangle_*$ and $\langle\lambda|_*$ to refer to basis vector in $\mathcal{F}_*$ and its dual, and any current operator $J_k$ acting on $\mathcal{F}_*$ or its dual is assumed to be the relevant current operator. Then we have

\begin{theorem} \label{charged-Gamma-lattice-Hamiltonian}
\leavevmode
\begin{enumerate}
    \item[(a)] The equation \begin{align}Z(\mathfrak{S}^{*,q}_{\lambda/\mu}) = \prod_{i=1}^{N} A_i^{-(M+1)} B_i^{-\ell(\lambda)} \cdot \langle\lambda|_*e^{H_-}|\mu\rangle_* \hspace{20pt} \text{for all strict partitions $\lambda,\mu$ and all $M,N$.} \label{charged-Gamma-Lattice-Hamiltonian-correspondence-equation}\end{align} holds precisely when the weights of $\mathfrak{S}^{*,q}$ satisfy the generalized free fermion condition and for all $k\ge 1, j\in [1,N]$, \begin{align} s_{-k}^{(j)} = \frac{1}{k} \left(y_i^{nk} \left(\prod_{a=0}^{n-1} h(a)\right)^k  + (-1)^{k-1}(g(0))^{-k}x_i^k y_i^{(n-1)k}\left(\prod_{a=0}^{n-1} h(a)\right)^k\right). \label{charged-Gamma-Hamiltonian-parameter}\end{align}
    \item[(b)] If the Boltzmann weights are not generalized free fermionic, (\ref{charged-Gamma-Lattice-Hamiltonian-correspondence-equation}) does not hold for any choice of the $s_k^{(j)}$.
\end{enumerate}
\end{theorem}

\begin{proof}
This follows from Theorem \ref{charged-Delta-lattice-Hamiltonian} and Proposition \ref{charged-second-Gamma-Delta-relationship}.
\end{proof}

\section{Supersymmetric LLT Polynomials} \label{charged-partition-function-section}

In this final section, we prove two main results. The first result (Theorem \ref{Partition-Function-is-Supersymmetric-LLT}) is that our charged partition function is a supersymmetric LLT polynomials, and the second (Theorem \ref{super-LLT-Cauchy}) is a Cauchy identity for supersymmetric LLT polynomials. Although our partition functions do not give all specializations of supersymmetric LLT polynomials, the Cauchy identity only uses Hamiltonian operators and is therefore fully general.

\subsection{The partition function of the charged models}

Let \[L_+(\boldsymbol{x}) = \sum_{k=1}^\infty \sum_{j=1}^N \frac{1}{k}x_j^k J_k.\] Brubaker, Buciumas, Bump, and Gustafsson showed \cite[Theorem~5.8]{BBBG-Hamiltonian} that the supersymmetric LLT polynomial can be expressed as \begin{align} \mathcal{G}_{\lambda/\mu}[\boldsymbol{x}|\boldsymbol{y}] := \mathcal{G}_{\lambda/\mu}[\boldsymbol{x}|\boldsymbol{y}; q] = \langle\mu+\rho|e^{L_+(\boldsymbol{x}^n)-L_+(\boldsymbol{y}^n)}|\lambda+\rho\rangle. \label{super-LLT-formula}\end{align}

Let $F = \prod_{a=0}^{n-1} f(a)$ and $H = \prod_{a=0}^{n-1} h(a)$. Then, \[H_+ = \sum_{k=1}^\infty \sum_{j=1}^N\frac{1}{k} (x_i^nF)^k  -  \sum_{k=1}^\infty \sum_{j=1}^N\frac{1}{k} (-g(0)x_i^{n-1}y_iF)^k = L(\boldsymbol{\theta})- L(\boldsymbol{\pi}),\] and \[H_- = \sum_{k=1}^\infty \sum_{j=1}^N\frac{1}{k} (y_i^nH)^k  -  \sum_{k=1}^\infty \sum_{j=1}^N\frac{1}{k} (-(g(0))^{-1}y_i^{n-1}x_iH)^k = L(\boldsymbol{\eta})- L(\boldsymbol{\tau}),\] where \[\theta_i = x_i^nF, \hspace{10pt} \pi_i = -g(0) x_i^{n-1}y_i F, \hspace{10pt} \eta_i = w_i^nH, \hspace{10pt} \tau_i = -g_*(0) w_i^{n-1}z_i H.\] Therefore, by Theorems \ref{charged-Delta-lattice-Hamiltonian} and \ref{charged-Gamma-lattice-Hamiltonian}, we have

\begin{theorem} \label{Partition-Function-is-Supersymmetric-LLT}
\[Z(\mathfrak{S}^q_{\lambda+\rho/\mu+\rho}) = \prod_{i=1}^N A_i^{M+1} B_i^{\ell(\lambda)} \mathcal{G}_{\lambda/\mu}(\boldsymbol{\theta}|\boldsymbol{\pi}; q),\] where $\theta_i = x_i^nF$ and $\pi_i = -g(0) x_i^{n-1}y_i F$, and similarly \[Z(\mathfrak{S}^{*,q}_{\lambda+\rho/\mu+\rho}) = \prod_{i=1}^N A_i^{-(M+1)} B_i^{-\ell(\lambda)} \mathcal{G}_{\lambda/\mu}(\boldsymbol{\eta}|\boldsymbol{\tau}; q^{-1}),\] where $\eta_i = w_i^nH$ and $\tau_i = -g_*(0) w_i^{n-1}z_i H$.
\end{theorem}

By varying $\boldsymbol{x}$ and $\boldsymbol{y}$, we obtain almost every value of the supersymmetric LLT polynomial as a partition function of $\mathfrak{S}^q$ and $\mathfrak{S}^{*,q}$. The exceptions are the values $\theta_i=0$ or $\pi_i=0$, since setting $x_i=0, y_i=0$, $f(a)=0$, or $h(a)=0$ causes the model to degenerate, and Theorem \ref{charged-Delta-lattice-Hamiltonian} no longer holds in this setting. Unfortunately, the case $\boldsymbol{\pi}=0$ of the classical LLT polynomials is one of these exceptional cases.

\subsection{Cauchy identity}

Using results from previous sections, one can prove some similar facts involving charged lattice models, $q$-Fock space and supersymmetric LLT polynomials to those proved in Sections \ref{corollaries-section} and \ref{boundary-conditions-section} about uncharged models, classical Fock space, and supersymmetric Schur polynomials. These include Fock space operators for general side boundary conditions, as well as branching, Pieri, and Cauchy identities. As these proofs take similar forms to those in Sections \ref{corollaries-section} and \ref{boundary-conditions-section}, we will only prove the Cauchy identity, and leave the rest to the interested reader.

A similar Cauchy identity for specializations are proved by Lam \cite{Lam-LLT} for LLT polynomials and Brubaker, Buciumas, Bump, and Gustafsson \cite{BBBG-Hamiltonian} for metaplectic symmetric functions. Our proof technique is similar to the proofs of those results.

See also Curran, Frechette, Yost-Wolff, Zhang, and Zhang \cite{CFYWZZ-super-LLT} for an interesting exploration of this Cauchy identity in the context of lattice models. Our Proposition \ref{super-LLT-Cauchy} is similar to their Corollary 3.3.

Let $L_+(\boldsymbol{x}|\boldsymbol{y}) := L_+(\boldsymbol{x}^n) - L_+(\boldsymbol{y}^n)$. Similarly, let \[L_-(\boldsymbol{x}|\boldsymbol{y}) := L_-(\boldsymbol{x}^n) - L_-(\boldsymbol{y}^n), \hspace{20pt} \text{where} \hspace{20pt} L_-(\boldsymbol{x}) = \sum_{k=1}^\infty \sum_{j=1}^N \frac{1}{k}x_j^k J_{-k}.\] Note that by adjointness \cite[Proposition~4.9]{BBBG-Hamiltonian}, we also have \begin{align*}\mathcal{G}_{\lambda/\mu}[\boldsymbol{x}|\boldsymbol{y}] = \langle\lambda+\rho|e^{L_-(\boldsymbol{x}|\boldsymbol{y})}|\mu+\rho\rangle.\end{align*}

Let \[\Omega(\boldsymbol{x}|\boldsymbol{y}; \boldsymbol{z}|\boldsymbol{w}) := \prod_{t=0}^{n-1}\prod_{i,j} \frac{(1-v^tx_i^nw_j^n)(1-v^ty_i^nz_j^n)}{(1-v^tx_i^nz_j^n)(1-v^ty_i^nw_j^n)}.\]

\begin{proposition}[Supersymmetric LLT Cauchy identity] \label{super-LLT-Cauchy}
For any strict partitions $\lambda$ and $\mu$, \begin{align}\sum_\nu \mathcal{G}_{\lambda/\nu}[\boldsymbol{x}|\boldsymbol{y}] \mathcal{G}_{\mu/\nu}[\boldsymbol{z}|\boldsymbol{w}] = \Omega(\boldsymbol{x}|\boldsymbol{y}; \boldsymbol{z}|\boldsymbol{w}) \sum_\nu \mathcal{G}_{\nu/\mu}[\boldsymbol{x}|\boldsymbol{y}] \mathcal{G}_{\nu/\lambda}[\boldsymbol{z}|\boldsymbol{w}], \label{supersymmetric-LLT-Cauchy-identity-equation}\end{align} where the sums are over all strict partitions $\nu$.
\end{proposition}

Note that both $\mathcal{G}_{\lambda/\nu}$ and $\mathcal{G}_{\nu/\lambda}$ are zero unless $|\lambda|-|\nu|$ is a multiple of $n$.

\begin{proof}

We evaluate the Hamiltonian $\langle\mu+\rho|e^{L_-(\boldsymbol{z}|\boldsymbol{w})}e^{L_+(\boldsymbol{x}|\boldsymbol{y})}|\lambda+\rho\rangle$ in two ways. First, \[\langle\mu+\rho|e^{L_-(\boldsymbol{z}|\boldsymbol{w})}e^{L_+(\boldsymbol{x}|\boldsymbol{y})}|\lambda+\rho\rangle = \sum_\nu \langle\mu+\rho|e^{L_-(\boldsymbol{z}|\boldsymbol{w})}|\nu+\rho\rangle \langle \nu+\rho|e^{L_+(\boldsymbol{w}|\boldsymbol{y})}|\lambda+\rho\rangle = \sum_\nu \mathcal{G}_{\lambda/\nu}[\boldsymbol{x}|\boldsymbol{y}]\mathcal{G}_{\mu/\nu}[\boldsymbol{z}|\boldsymbol{w}].\]
Next, we apply the commutation relations between $L_+$ and $L_-$. Let $s_k^{(j)} = x_j^{nk} - y_j^{nk}$, $t_{-k}^{(j)} = z_j^{nk} - w_j^{nk}$, and $s_k = \sum_j s_k^{(j)}$, $t_{-k} = \sum_j t_{-k}^{(j)}$. Recall that \[[J_k,J_l] = k\cdot \frac{1-v^{n|k|}}{1-v^{|k|}} \delta_{k,-l}.\] Then, \begin{align*}\langle\mu+\rho|e^{L_-(\boldsymbol{z}|\boldsymbol{w})}e^{L_+(\boldsymbol{x}|\boldsymbol{y})}|\lambda+\rho\rangle &= \exp\left(\sum_{k\ge 1} k \frac{1-v^{nk}}{1-v^k}\cdot s_k t_{-k}\right) \cdot \langle\mu+\rho|e^{L_+(\boldsymbol{x}|\boldsymbol{y})}e^{L_-(\boldsymbol{z}|\boldsymbol{w})}|\lambda+\rho\rangle \\&= \prod_{i,j} \exp\left(\sum_{k\ge 1} k \frac{1-v^{nk}}{1-v^k}\cdot s_k^{(i)}t_{-k}^{(j)}\right)\cdot\sum_\nu \mathcal{G}_{\nu/\mu}[\boldsymbol{x}|\boldsymbol{y}]\mathcal{G}_{\nu/\lambda}[\boldsymbol{z}|\boldsymbol{w}].\end{align*}

Now, \begin{align*} \sum_{k\ge 1} k \frac{1-v^{nk}}{1-v^k}\cdot s_k^{(i)}t_{-k}^{(j)} &= \sum_{k\ge 1} \frac{1}{k} \frac{1-v^{nk}}{1-v^k}\cdot (x_i^{nk} - y_i^{nk})(z_j^{nk} - w_j^{nk}) \\&= \sum_{t=0}^{n-1}\sum_{k\ge 1} \frac{1}{k} v^{tk}(x_i^{nk}z_j^{nk} - x_i^{nk}w_j^{nk} - y_i^{nk}z_j^{nk} + y_i^{nk}w_j^{nk}) \\&= \log\prod_{t=0}^{n-1} \frac{(1-v^tx_i^nw_j^n)(1-v^ty_i^nz_j^n)}{(1-v^tx_i^nz_j^n)(1-v^ty_i^nw_j^n)},\end{align*} and combining this with the first equation gives (\ref{supersymmetric-LLT-Cauchy-identity-equation}).
\end{proof}

If we specialize $x_i^n\mapsto x_i^nF, y_i^n\mapsto -g(0)x_i^{n-1}y_iF, z_i^n\mapsto w_i^nH$, and $w_i^n\mapsto -g_*(0)w_i^{n-1}z_iH$, then we obtain a Cauchy identity for our partition functions $Z(\mathfrak{S}^q)$ and $Z(\mathfrak{S}^{*,q})$.

If we instead let $y_i\mapsto 0, w_i\mapsto 0$, then we obtain the Cauchy identity for classical LLT polynomials \cite[Theorem~26]{Lam-LLT}, while if we let $y_i^n\mapsto vx_i^n, w_i^n\mapsto vz_i^n$, we obtain the Cauchy identity for metaplectic symmetric functions \cite[Theorem~5.10]{BBBG-Hamiltonian}. Both of those results are given for the special case $\lambda=\mu=\delta$, where $\delta$ is an \emph{$n$-core partition}.

\appendix

\section{Charged Model Equations for Solvability} \label{appendix-equations}

Below, we list the inequivalent equations obtained from (\ref{YBE-diagram}), varying $\alpha,\beta,\gamma,\delta,\epsilon,\eta$ across all decorated spins. Charges $k$ and $m$ are taken modulo $n$. Solving these equations gives the conditions in Theorem \ref{charge-solvability-theorem}.

\[\vx{B}_1(k) = \vx{B}_1(k+1), \hspace{20pt} \vx{B}_2(k) = \vx{B}_2(k+1), \hspace{20pt} \forall k;\]
Set $\vx{B}_1 := \vx{B}_1(k), \vx{B}_2 := \vx{B}_2(k)$ (independent of $k$).
\[\frac{\vx{C}_1(k+1)}{\vx{C}_1(k)} = \frac{\vx{a_1^{(j)}}\vx{b_2^{(i)}}(k)}{\vx{a_1^{(i)}}\vx{b_2^{(j)}}(k)} = \frac{\vx{a_2^{(i)}}(k)\vx{b_1^{(j)}}}{\vx{a_2^{(j)}}(k)\vx{b_1^{(i)}}}, \hspace{20pt} \forall k\ne 0;\]
\[\frac{\vx{C}_2(k+1)}{\vx{C}_2(k)} = \frac{\vx{a_1^{(i)}}\vx{b_2^{(j)}}(k)}{\vx{a_1^{(j)}}\vx{b_2^{(i)}}(k)} = \frac{\vx{a_2^{(j)}}(k)\vx{b_1^{(i)}}}{\vx{a_2^{(i)}}(k)\vx{b_1^{(j)}}}, \hspace{20pt} \forall k\ne 0;\]
\[\vx{A}_2^\times(k+1,m+1) = \vx{A}_2^\times(k,m), \hspace{20pt} \forall k,m;\]
\[\frac{\vx{A}_2^\times(0,m)}{\vx{B}_1} = \frac{\vx{b_2^{(j)}}(m)}{\vx{a_2^{(j)}}(m)}, \hspace{20pt} \forall m\ne 0;\]
\[\frac{\vx{A}_2^\times(k,0)}{\vx{B}_2} =  \frac{\vx{a_2^{(i)}}(k)}{\vx{b_2^{(i)}}(k)}, \hspace{20pt} \forall k\ne 0;\]
\[\frac{\vx{A}_2(k+1,m+1)}{\vx{A}_2(k,m)} = \frac{\vx{b_2^{(j)}}(k)\vx{b_2^{(i)}}(m)}{\vx{b_2^{(i)}}(k)\vx{b_2^{(j)}}(m)} = \frac{\vx{a_2^{(j)}}(k)\vx{a_2^{(i)}}(m)}{\vx{a_2^{(i)}}(k)\vx{a_2^{(j)}}(m)}, \hspace{20pt} \forall k,m;\]
Set $\vx{A}_2 := \vx{A}_2(k,k) = \vx{A}^\times_2(k,k)$ (independent of $k$).
\[\frac{\vx{A}_2(k,0)}{\vx{C}_2(k+1)} = \frac{\vx{a_2^{(i)}}(k)\vx{c_2^{(j)}}}{\vx{a_2^{(j)}}(k)\vx{c_2^{(i)}}}, \hspace{20pt} \frac{\vx{A}_2(k+1,1)}{\vx{C}_2(k)} = \frac{\vx{b_2^{(j)}}(k)\vx{c_1^{(i)}}}{\vx{b_2^{(i)}}(k)\vx{c_1^{(j)}}}, \hspace{20pt} \forall k\ne 0;\]
\[\frac{\vx{A}_2(0,k)}{\vx{C}_1(k+1)} = \frac{\vx{b_2^{(j)}}(k)\vx{c_2^{(i)}}}{\vx{b_2^{(i)}}(k)\vx{c_2^{(j)}}}, \hspace{20pt} \frac{\vx{A}_2(1,k+1)}{\vx{C}_1(k)} = \frac{\vx{a_2^{(i)}}(k)\vx{c_1^{(j)}}}{\vx{a_2^{(j)}}(k)\vx{c_1^{(i)}}}, \hspace{20pt} \forall k\ne 0;\]
\[\frac{\vx{C}_2(1)}{\vx{C}_1(0)} = \frac{\vx{c_1^{(j)}}\vx{c_2^{(i)}}}{\vx{c_1^{(i)}}\vx{c_2^{(j)}}}, \hspace{20pt} \frac{\vx{C}_1(1)}{\vx{C}_2(0)} = \frac{\vx{c_1^{(i)}}\vx{c_2^{(j)}}}{\vx{c_1^{(j)}}\vx{c_2^{(i)}}}, \hspace{20pt} \forall k\ne 0;\]
\[\vx{C}_2(1)\vx{b_2^{(i)}}(0)\vx{a_1^{(j)}} = \vx{b_2^{(j)}}(0)\vx{a_1^{(i)}}\vx{C}_2(0) + \vx{c_1^{(j)}}\vx{c_2^{(i)}}\vx{B}_2;\]
\[\vx{C}_1(1)\vx{a_1^{(i)}}\vx{b_2^{(j)}}(0) + \vx{B}_2\vx{c_1^{(i)}}\vx{c_2^{(j)}} = \vx{a_1^{(j)}}\vx{b_2^{(i)}}(0)\vx{C}_1(0);\]
\[\vx{C}_1(1)\vx{b_1^{(i)}}\vx{a_2^{(j)}}(0) = \vx{b_1^{(j)}}\vx{a_2^{(i)}}(0)\vx{C}_1(0) + \vx{c_2^{(j)}}\vx{c_1^{(i)}}\vx{B}_1;\]
\[\vx{C}_2(1)\vx{a_2^{(i)}}(0)\vx{b_1^{(j)}} + \vx{B}_1\vx{c_2^{(i)}}\vx{c_1^{(j)}} = \vx{a_2^{(j)}}(0)\vx{b_1^{(i)}}(0)\vx{C}_2(0);\]
\[\vx{A}_1\vx{c_2^{(i)}}\vx{a_1^{(j)}} = \vx{c_2^{(j)}}\vx{a_1^{(i)}}\vx{C}_2(0) + \vx{b_1^{(j)}}\vx{c_2^{(i)}}\vx{B}_2;\]
\[\vx{C}_1(1)\vx{a_1^{(i)}}\vx{c_1^{(j)}} + \vx{B}_2\vx{c_1^{(i)}}\vx{b_1^{(j)}} = \vx{a_1^{(j)}}\vx{c_1^{(i)}}(0)\vx{A}_1;\]
\[\vx{A}_1\vx{b_1^{(i)}}\vx{c_2^{(j)}} = \vx{c_2^{(j)}}\vx{a_1^{(i)}}\vx{B}_2 + \vx{b_1^{(j)}}\vx{c_2^{(i)}}\vx{C}_1(0);\]
\[\vx{B}_1\vx{a_1^{(i)}}\vx{c_1^{(j)}} + \vx{C}_2(1)\vx{c_1^{(i)}}\vx{b_1^{(j)}} = \vx{c_1^{(j)}}\vx{b_1^{(i)}}\vx{A}_1;\]
\[\vx{A}_2\vx{c_1^{(i)}}\vx{a_2^{(j)}}(0) = \vx{c_1^{(j)}}\vx{a_2^{(i)}}(0)\vx{C}_1(0) + \vx{b_2^{(j)}}(0)\vx{c_1^{(i)}}\vx{B}_1;\]
\[\vx{C}_2(1)\vx{a_2^{(i)}}(0)\vx{c_2^{(j)}} + \vx{B}_1\vx{c_2^{(i)}}\vx{b_2^{(j)}}(0) = \vx{a_2^{(j)}}(0)\vx{c_2^{(i)}}\vx{A}_2;\]
\[\vx{A}_2\vx{b_2^{(i)}}(0)\vx{c_1^{(j)}} = \vx{b_2^{(j)}}(0)\vx{c_1^{(i)}}\vx{C}_2(0) + \vx{c_1^{(j)}}\vx{a_2^{(i)}}(0)\vx{B}_2;\]
\[\vx{C}_1(1)\vx{c_2^{(i)}}\vx{b_2^{(j)}}(0) + \vx{B}_2\vx{a_2^{(i)}}(0)\vx{c_2^{(j)}} = \vx{c_2^{(j)}}\vx{b_2^{(i)}}(0)\vx{A}_2;\]

\bibliographystyle{siam}
\bibliography{bibliography.bib}

\end{document}